\documentclass[11pt,twoside,a4paper]{amsart}
\usepackage{tikz}
\usetikzlibrary{patterns}

\usepackage{amsfonts,amsmath,amsthm,amssymb,latexsym,amsthm,newlfont,enumerate,color}
\usepackage{mathtools}
\usepackage{microtype}

\newif\ifslide

\theoremstyle{plain}

\ifdefined\spacer
\newtheorem{theorem}{Theorem}
 \else
\newtheorem{theorem}{Theorem}[section]
\fi

\newtheorem{corollary}[theorem]{Corollary}
\newtheorem{lemma}[theorem]{Lemma}

\newtheorem{setup}[theorem]{Setup}

\newtheorem*{theorem*}{Theorem}

\newtheorem{proposition}[theorem]{Proposition}
\newtheorem{definition-lemma}[theorem]{Definition-Lemma}

\newtheorem{red-question}[theorem]{\textcolor{red}{Question}}

\theoremstyle{definition}
\newtheorem{definition}[theorem]{Definition}
\newtheorem{remark}[theorem]{Remark}
\newtheorem{example}[theorem]{Example}

\DeclareMathOperator{\Cl}{Cl}

\def\red{\textcolor{red}}
\def\blue {\textcolor{blue}}

\def\ideal#1.{I_{#1}}
\def\ring#1.{\mathcal {O}_{#1}}

\def\fring#1.{\hat{\mathcal {O}}_{#1}}
\def\proj#1.{\mathbb {P}(#1)}
\def\pr #1.{\mathbb {P}^{#1}}
\def\dpr #1.{\hat{\mathbb {P}}^{#1}}
\def\af #1.{\mathbb A^{#1}}
\def\Hz #1.{\mathbb F_{#1}}
\def\Hbz #1.{\overline{\mathbb F}_{#1}}
\def\fb#1.{\underset #1 {\times}}
\def\rest#1.{\underset {\ \ring #1.} \to \otimes}
\def\au#1.{\operatorname {Aut}\,(#1)}
\def\deg#1.{\operatorname {deg } (#1)}

\def\pic#1.{\operatorname {Pic}\,(#1)}
\def\pico#1.{\operatorname{Pic}^0(#1)}
\def\picg#1.{\operatorname {Pic}^G(#1)}
\def\ner#1.{NS (#1)}
\def\rdown#1.{\llcorner#1\lrcorner}
\def\rfdown#1.{\lfloor{#1}\rfloor}
\def\rup#1.{\ulcorner{#1}\urcorner}
\def\rcup#1.{\lceil{#1}\rceil}

\def\n1#1.{\operatorname {N_1}(#1)}  
\def\cn1#1.{\overline{\operatorname {N^1}(#1)}} 
\def\cone#1.{\operatorname {NE}(#1)}     
\def\ccone#1.{\overline{\operatorname {NE}}(#1)}
\def\none#1.{\operatorname {NF}(#1)}
\def\cnone#1.{\overline{\operatorname {NF}}(#1)}
\def\mone#1.{\operatorname {NM}(#1)} 
\def\cmone#1.{\overline{\operatorname {NM}}(#1)}

\def\coef#1.{\frac{(#1-1)}{#1}}
\def\vit#1.{D_{\langle #1 \rangle}}
\def\mm#1.{\overline {M}_{0,#1}}
\def\H1#1.{H^1(#1,{\ring #1.})}
\def\ac#1.{\overline {\mathbb F}_{#1}}

\def\adj#1.{\frac {#1-1}{#1}}
\def\spn#1.{\overline{#1}}
\def\pek#1.#2.{\Cal P^{#1}(#2)}
\def\plk#1.#2.{\Cal P^{\leq #1}(#2)}
\def\ev#1.{\operatorname{ev_{#1}}}
\def\ilist#1.{{#1}_1,{#1}_2,\dots}
\def\bminv#1.{(\nu_1,s_1;\nu_2,s_2;\dots ;\nu_{#1},s_{#1};\nu_{r+1})}
\def\zinv#1.{(\nu_1,s_1;\nu_2,s_2;\dots ;\nu_{#1},s_{#1};0)}
\def\iinv#1.{(\nu_1,s_1;\nu_2,s_2;\dots ;\nu_{#1},s_{#1};\infty)}

\def\scr #1.{\mathcal #1}

\def\llist#1.#2.{{#1}_1,{#1}_2,\dots,{#1}_{#2}}
\def\ulist#1.#2.{{#1}^1,{#1}^2,\dots,{#1}^{#2}}
\def\lomitlist#1.#2.{{#1}_1,{#1}_2,\dots,\hat {{#1}_i}, \dots, {#1}_{#2}}
\def\lomitlistz#1.#2.{{#1}_0,{#1}_1,\dots,\hat {{#1}_i}, \dots, {#1}_{#2}}
\def\loc#1.#2.{\Cal O_{#1,#2}}
\def\fderiv#1.#2.{\frac {\partial #1}{\partial #2}}
\def\deriv#1.#2.{\frac {d #1}{d #2}}

\def\map#1.#2.{#1 \longrightarrow #2}
\def\rmap#1.#2.{#1 \dasharrow #2}
\def\emb#1.#2.{#1 \hookrightarrow #2}
\def\non#1.#2.{\text {Spec }#1[\epsilon]/(\epsilon)^{#2}}
\def\Hi#1.#2.{\text {Hilb}^{#1}(#2)}
\def\sym#1.#2.{\operatorname {Sym}^{#1}(#2)}
\def\Hb#1.#2.{\text {Hilb}_{#1}(#2)}
\def\Hm#1.#2.{\Hom_{#1}(#2)}
\def\prd#1.#2.{{#1}_1\cdot {#1}_2\cdots {#1}_{#2}}
\def\Bl #1.#2.{\operatorname {Bl}_{#1}#2}
\def\pl #1.#2.{#1^{\otimes #2}}
\def\mgn#1.#2.{\overline {M}_{#1,#2}}
\def\ialist#1.#2.{{#1}_1 #2 {#1}_2, #2\dots}
\def\pair#1.#2.{\langle #1, #2\rangle}
\def\vandermonde#1.#2.{\left|
\begin{matrix}
1 & 1 & 1 & \dots & 1\\
{#1}_1 & {#1}_2 & {#1}_3 & \dots & {#1}_{#2}\\
{#1}_1^2 & {#1}_2^2 & {#1}_3^2 & \dots & {#1}_{#2}^2\\
\vdots & \vdots & \vdots & \ddots & \vdots\\
{#1}_1^{#2-1} & {#1}_2^{#2-1} & {#1}_2^{#2-1} & \dots & {#1}_{#2}^{#2-1}\\
\end{matrix}
\right|
}
\def\vandermondet#1.#2.{\left|
\begin{matrix}
1 & {#1}_1   & {#1}_1^2 & \dots & {#1}_1^{#2-1}\\
1 & {#1}_2   & {#1}_2^2 & \dots & {#1}_2^{#2-1}\\
1 & {#1}_3   & {#1}_3^2 & \dots & {#1}_3^{#2-1}\\
\vdots & \vdots & \vdots & \ddots & \vdots\\
1 & {#1}_{#2}& {#1}_{#2}^2 & \dots & {#1}_{#2}^{#2-1}\\
\end{matrix}
\right|
}
\def\gr#1.#2.{\mathbb{G}(#1,#2)}

\def\alist#1.#2.#3.{{#1}_1 #2 {#1}_2 #2\dots #2 {#1}_{#3}}
\def\zlist#1.#2.#3.{#1_0 #2 #1_1 #2\dots #2 #1_{#3}}
\def\lomitlist30#1.#2.#3.{{#1}_0,{#1}_1 #2 \dots #2\hat {{#1}_i} #2\dots #2 {#1}_{#3}}
\def\lmap#1.#2.#3.{#1 \overset{#2}{\longrightarrow} #3}
\def\mes#1.#2.#3.{#1 \longrightarrow #2 \longrightarrow #3}
\def\ses#1.#2.#3.{0\longrightarrow #1 \longrightarrow #2 \longrightarrow #3 \longrightarrow 0}
\def\les#1.#2.#3.{0\longrightarrow #1 \longrightarrow #2 \longrightarrow #3}
\def\res#1.#2.#3.{#1 \longrightarrow #2 \longrightarrow #3\longrightarrow 0}
\def\Hi#1.#2.#3.{\text {Hilb}^{#1}_{#2}(#3)}
\def\ten#1.#2.#3.{#1\underset {#2}{\otimes} #3}
\def\lomitlist30#1.#2.#3.{{#1}_0 #2 {#1}_1 #2 \dots #2 \hat {{#1}_i} #2 \dots #2 {#1}_{#3}}
\def\mderiv#1.#2.#3.{\frac {d^{#3} #1}{d #2^{#3}}}

\def\Hom{\operatorname{Hom}}

\def\dim{\operatorname{dim}}

\def\deg{\operatorname{deg}}

\def\Pic{\operatorname{Pic}}

\def\Sing{\operatorname{Sing}}

\def\mult{\operatorname{mult}}

\def\ord{\operatorname{ord}}

\def\rest{\operatorname{res}}

\def\vol{\operatorname{vol}}

\def\e{\Cal E}

\def\e1{E_1}
\def\e2{E_2}

\def\Q{\mathbb Q}

\def\mapdown#1{\big\downarrow\rlap{$\vcenter{\hbox{$\scriptstyle#1$}}$}}

\def\mapse#1{
{\vcenter{\hbox{$\mathop{\smash{\raise1pt\hbox{$\diagdown$}\!\lower7pt
\hbox{$\searrow$}}\vphantom{p}}\limits_{#1}\vphantom{\mapdown{}}$}}}}

\def\VR#1.{height#1pt&\omit&&\omit&&\omit&&\omit&&\omit&\cr}

\def\VRT#1.{height#1pt&\omit&&\omit&\cr}

\usepackage[colorlinks=true,pagebackref,hyperindex]{hyperref}
\usepackage{amsmath}
\usepackage{amssymb}
\usepackage{amsthm}
\usepackage[all,cmtip]{xy}
\usepackage{comment}

\newcommand{\bA}{\ensuremath{\mathbb{A}}}

\newcommand{\bC}{\ensuremath{\mathbb{C}}}
\newcommand{\bD}{\ensuremath{\mathbb{D}}}
\newcommand{\bE}{\ensuremath{\mathbb{E}}}

\newcommand{\bH}{\ensuremath{\mathbb{H}}}

\newcommand{\bP}{\ensuremath{\mathbb{P}}}
\newcommand{\bQ}{\ensuremath{\mathbb{Q}}}
\newcommand{\bR}{\ensuremath{\mathbb{R}}}

\newcommand{\bZ}{\ensuremath{\mathbb{Z}}}

\newcommand{\cL}{\ensuremath{\mathcal{L}}}
\newcommand{\cM}{\ensuremath{\mathcal{M}}}

\newcommand{\cO}{\ensuremath{\mathcal{O}}}

\newcommand{\vecbul}{\vec{\bullet}}

\DeclareMathOperator{\Image}{Im}

\DeclareMathOperator{\CaCl}{CaCl}
\DeclareMathOperator{\Coker}{Coker}
\DeclareMathOperator{\Ker}{Ker}

\newcommand{\dps}{\displaystyle}
\newcommand{\bul}{\bullet}

\title{K-stability of Fano weighted hypersurfaces \\ via plt flags and convex geometry}
\date{\today}

\author{Livia Campo} 
\address{Institut f\"{u}r Mathematik, Universit\"{a}t Wien, Oskar-Morgenstern-Platz 1, 1090 Wien, Austria}
\email{livia.campo@univie.ac.at}

\author{Kento Fujita}
\address{Department of Mathematics, Graduate School of Science, 
Osaka University, Toyonaka, Osaka 560-0043, Japan} 
\email{fujita@math.sci.osaka-u.ac.jp} 

\author{Taro Sano}
\address{Department of Mathematics, Graduate School of Science, 
Kobe university, 
1-1, Rokkodai, Nada-ku, Kobe 657-8501, Japan} 
\email{tarosano@math.kobe-u.ac.jp} 

\author{Luca Tasin}
\address{Dipartimento di Matematica F.\ Enriques, Universit\`a degli Studi di Milano, Via Cesare Saldini 50, 20133 Milano, Italy} 
\email{luca.tasin@unimi.it}

\begin{document}



\begin{abstract}
We develop a framework to study the K-stability of weighted Fano hypersurfaces based on a combination of birational and convex-geometric techniques. 
As an application, we prove that all quasi-smooth weighted Fano hypersurfaces of index~1 with at most two weights greater than~1 are K-stable. 
We also construct several examples of K-unstable quasi-smooth weighted Fano hypersurfaces of low indices. 

To prove these results, we establish lower bounds for stability thresholds using the method of Abban--Zhuang, which reduces the problem to lower-dimensional cases. A key feature of our approach is the use of plt flags that are not necessarily admissible.
\end{abstract}
\maketitle

\section{Introduction}

K-stability has become a central notion in the study of Fano varieties,  serving as the algebro-geometric counterpart of the existence of K\"{a}hler--Einstein metrics (see \cite{CDS1,CDS2,CDS3, MR3352459}).
It provides a bridge between birational geometry, moduli theory, and complex differential geometry, 
and plays a crucial role in the classification of Fano varieties and their degenerations. For a comprehensive treatment of the topic, we refer the reader to \cite{Xubook}.
For instance, Fano manifolds are classified in dimension 2 and 3 and their K-stability has been intensively studied (see \cite{Tian90}, \cite{Calabiproblem}, and references therein).
One of the most effective approaches to studying K-stability is through \emph{stability thresholds} associated with the anticanonical divisor~\(-K_X\) 
\cite{MR3956698, MR3715806, MR3896135, MR4067358, LXZ22}. 
Among these, the stability threshold offers a particularly concrete numerical criterion: a klt Fano variety~\(X\) is K-stable if and only if $\delta(X; -K_X) > 1$.

Among Fano varieties, hypersurfaces in weighted projective spaces form a particularly natural and rich class. 
They generalize smooth hypersurfaces in ordinary projective spaces while admitting a wide range of singularities and geometric behaviors (cf. \cite{Mori75,Dolgachev,Fletcher00}). 
While it is expected that every smooth Fano hypersurface of degree at least~\(3\) is K-stable 
(see Subsection~\ref{ss:known-standard} for the known results), 
the situation in the weighted setting is considerably more intricate and delicate (see Subsection~\ref{ss:known-weighted}). 

In this paper, we initiate a systematic study of the K-stability of quasi-smooth weighted Fano hypersurfaces. 
Our main goal is to develop effective techniques to estimate their stability thresholds and to identify explicit families of K-stable examples. 
The approach we propose combines methods from birational geometry and convex geometry, providing a framework that can be applied to a broad range of situations.
As a first application we prove the following. 

\begin{theorem}\label{thm:d=ak+1}
Let $a >1$ and $n \ge 3$ be integers.  
Let $X=X_d \subset \bP(1^{n+1},a)$ be a quasi-smooth Fano weighted hypersurface of degree $d$, dimension $n$ and index $1$. Then 
\[
\delta\left(X; \cO_X(1)\right) \ge  \frac{n+1}{n}>1.
\]
In particular, $X$ is K-stable. 
\end{theorem}

As one expects, the hard part in the proof of Theorem \ref{thm:d=ak+1} is to understand the local stability threshold $\delta_P(X,\cO(1))$ in the case $P=[0:\cdots:0:1] \in X$.
If $P$ is a generalized Eckardt point (cf.\ Section \ref{section_eckardt}), 
then we obtain a sharp estimate in the following sense. 

\begin{example}[See Corollary \ref{corollary_ninja}]
Let $X_d \subset \bP(1^{n+1},a)$ be a quasi-smooth weighted hypersurface of dimension $n\ge 3$ and degree $d=ak+1 \ge n$ with $k \in \mathbb N$. Assume that the point $P=[0:\cdots:0:1]$ is a generalized Eckardt point for $X$, i.e., $X$ is defined by a polynomial $f$ of the form 
$$
f=x_{n+1}^kx_0+x_{n+1}^{k-1}f_{a+1}(x_0,\dots,x_n)+\cdots+
f_{ak+1}(x_0,\dots,x_n)
$$
such that $x_0 \mid f_{at+1}$ for any $t=1,\ldots,k-1$ (cf. Example \ref{example_eckardt-qsm}). Then 
\begin{equation}
    \delta_P\left(X; \cO_X(1)\right)=\frac{n(n+1)}{ak+n}.
\end{equation}
\end{example}

As a consequence of this result, we construct examples of K-unstable Fano quasi-smooth weighted hypersurfaces whose indices are smaller than their dimensions; see Example~\ref{example:unstable}. 
To the best of our knowledge, these provide the first examples with such properties in higher dimensions 
and highlight a notable difference from the case of smooth hypersurfaces in the standard projective space. 
Compare also the examples with \cite{KW21}, \cite[Theorem 6.1]{ST21}.

If $P$ is not a generalized Eckardt point, we do not obtain a sharp estimate, but we can still derive several lower bounds.
 See Propositions \ref{p:surfacesingular} and \ref{p:surface_dbig-new}.

\medskip

Next, we consider the case in which exactly two weights are bigger than 1. 

\begin{theorem}\label{thm:wt2}
Let $2\leq a\leq b$ and $n\geq 3$ be integers. 
Let $X=X_d\subset\bP(1^n,a,b)$ be a quasi-smooth Fano weighted 
hypersurface of degree $d$, dimension $n$ and index $1$. Then 
\[
\delta\left(X; \cO_X(1)\right) \ge  \frac{n+1}{n+\frac{1}{a}}>1.
\]
In particular, $X$ is K-stable. 
\end{theorem}

\medskip

When more than two weights are bigger than 1, the analysis of K-stability becomes subtler. 
Even for explicit families, such as
\[
X_{2n+3} \subset \mathbb{P}(1^{n-1},2,2,n+1)
\]
with $n\geq 4$, 
the K-stability of all members remains an open problem. 
In this direction, we obtain the following result for general quasi-smooth weighted hypersurfaces.

\begin{corollary}\label{cor:imperial}
Let $n, c_1 \in \bZ_{>0}$ with $n \ge 3$ and $c_1 \ge \frac{n+2}{2} $. Let $X_d \subset \bP= \bP(1^{c_1}, a_{c_1}, \ldots ,a_{n+1})$ be a quasi-smooth Fano weighted hypersurface of degree $d$ and index $1$ such that $2 \le a_{c_1}\le \cdots \le a_{n+1}$.  
Assume the following: 
\begin{itemize}
    \item $d \equiv 1 \mod a_i$ for all $c_1 \le i \le n+1$. 
    \item $X_d$ is general in the linear system $|\cO_{\bP}(d)|$. 
\end{itemize}
Then $X_d$ is K-stable. 
\end{corollary}

We briefly explain the strategy for the proofs of Theorems \ref{thm:d=ak+1} and 
\ref{thm:wt2}. 
In order to estimate the local stability threshold for quasi-smooth and well-formed 
hypersurfaces $X$ in weighted projective spaces $\bP$, the \emph{Abban--Zhuang method} 
\cite{AZ22} is one of the most powerful techniques. For a given point $p\in X$, 
if there exists a quasi-smooth prime divisor $Y$ 
on $X$ passing through $p$ of small degree, then Abban--Zhuang's 
method enables us to reduce estimating the local stability threshold of $X$ at $p$ to 
the local stability threshold of $Y$ at $p$. 

In many cases, we can reduce the situation to quasi-smooth surfaces 
by using the Abban--Zhuang method. 
One of the main tools to estimate the local stability threshold for surfaces 
is the \emph{Okounkov body} of the (weighted) blowups of the surfaces 
associated with an admissible flag.  
In fact, Abban--Zhuang \cite[Lemma 3.2]{AZ23} gave an interesting relationship between 
the \emph{Seshadri constant} and the local stability threshold for smooth surfaces, 
which can be proved by looking at the position of the barycenter of suitable 
Okounkov bodies. In Section \ref{section:barycenter}, we will give a generalization of it 
such that given points have quotient singularities (Corollaries \ref{corollary_gravity}, \ref{corollary_lobster}). 
Moreover, in Section \ref{section_singular}, we will give sharp estimates of 
Seshadri constants at specific quotient singularities. 
Theorem \ref{thm:d=ak+1} can be proved by using this technique. 

On the other hand, in general, we cannot take a suitable quasi-smooth prime divisor on $X$ 
in order to apply Abban--Zhuang's method. 
In order to resolve this difficulty, 
in Section~\ref{section:w-blowup} we introduce another fundamental tool, i.e.\ \emph{the 
standard weighted blowup of weighted projective spaces}. The blowup 
is a weighted generalization of the blowup of ordinary projective spaces along linear 
subspaces. In several cases, even if a prime divisor on $X$ passing through a given point $p$ 
is not quasi-smooth, its \emph{strict transform} to the standard weighted blowup 
can be quasi-smooth, see Corollary \ref{corollary_2wt-cutting}. 

Theorem \ref{thm:wt2} can be proven by combining the above techniques. 
Beyond its immediate role in the present work, we expect that this construction will prove useful in broader investigations involving weighted and toric birational geometry.

\subsection{Structure of the paper} We briefly outline how the paper is organized. In Section \ref{section:prelim}, 
we recall the notion of (local) stability thresholds and the main properties we are going to use.
In Section \ref{section:w-blowup} 
we introduce the notion of standard weighted blowups and see their basic 
properties. We also analyze quasi-smoothness and well-formedness for hypersurfaces 
in $\bQ$-factorial toric varieties in Subsection \ref{section_q-smooth}. 
In Sections \ref{section_cutting} and \ref{section_covering}, we see when divisors 
or covers in quasi-smooth weighted hypersurfaces are again quasi-smooth. 
The results are generalizations of the results in \cite[\S 3 and \S 5]{ST24}. 
Section \ref{section:barycenter} is one of the main technical keys of our paper. 
We analyze the positions of $2$-dimensional convex sets in Proposition 
\ref{proposition_gravity}. As in \cite[\S 11]{Fujita2023}, the local stability thresholds 
can be estimated via the information of the barycenters of suitable Okounkov bodies. 
In fact, as an immediate consequence of Proposition \ref{proposition_gravity}, 
we can get a weighted version of \cite[Lemma 3.2]{AZ23} in Corollary 
\ref{corollary_gravity}. In Section \ref{section_singular}, we consider the situation 
in Theorem \ref{thm:d=ak+1}. More precisely, we consider the case such that the ambient 
weighted projective space is $\bP(1^{n+1},a)$. In Subsection \ref{section_non-eckardt}, 
we consider the case that the hypersurface $X$ is not a generalized Eckardt point at the singular 
vertex. By looking at the standard weighted blowup 
along its singular vertex, we can get a sharp estimates for the Seshadri constant 
for general sections of original weighted hypersurfaces in Corollary 
\ref{corollary_sesh-non-ec}. In Subsection \ref{section_eckardt}, 
we consider the case that $X$ is a generalized Eckardt point at the singular vertex. 
In this case, the estimate of the local stability threshold can be explicitly described 
in Corollary \ref{corollary_ninja}. Sections \ref{section_wt1}, \ref{section_wt2} and 
\ref{section_imperial} are the proofs of Theorems \ref{thm:d=ak+1}, \ref{thm:wt2} 
and Corollary \ref{cor:imperial} respectively, by using the techniques in the previous 
sections.

\subsection{Known results for standard hypersurfaces}\label{ss:known-standard}

It is conjectured that every smooth Fano hypersurface  in the standard projective space \(X \subset \mathbb{P}^{n+1}\) of degree at least~\(3\) is K-stable. 
The conjecture has been verified when the Fano index of \(X\) is~\(1\) 
\cite{Pukhlikov98, Cheltsov01, MR3936640}, 
or~\(2\) \cite{AZ22}, for cubic fourfolds \cite{Liu22}, 
and in the range \(\dim X \ge \iota_X^3\) 
\cite[Theorem~B]{AZ23}, where $\iota_X$ is the Fano index of $X$.  Remarkably, the conjecture is also known for Fermat type hypersurfaces \cite{MR894378, Tian00, MR2212883, MR4309493}. By openness of K-stability \cite{MR4411858, MR4505846}, this implies that the general Fano hypersurface of degree at least 3 is K-stable.

\subsection{Known results for weighted hypersurfaces}\label{ss:known-weighted}

Quasi-smooth del Pezzo surfaces $X \subset \bP(a_0,\ldots,a_3)$ of index 1 have been classified in \cite{Johnson-Kollar2} and the existence of K\"{a}hler--Einstein metrics has been proven in all cases (cf.\ Ibid, \cite{MR1926877,CJS10,MR4225810}). Note that there exist unstable examples with index 2 \cite{KW21}. 

There are 95 families of terminal quasi-smooth Fano 3-fold hypersurfaces $X \subset \bP(a_0,\ldots,a_4)$ of index 1, see \cite{Johnson-Kollar,CCC11,Fletcher00}. By \cite{Cheltsov08,MR2499678}, the general element of each family is K-stable. Using the relation with birational superrigidity (see \cite{SZ19,KOW18,KOW23}), the K-stability of all members of 75 families have been proven, as reported in \cite[Table 7]{KOW23}.  
On the other hand, \cite{CO} established the K-stability in the remaining cases of the 95 families. K-stability for the non-terminal quasi-smooth Fano threefold hypersurfaces of index 1 is an interesting open problem.

In higher dimension, the picture is very much incomplete, some partial results can be found in \cite{Johnson-Kollar, MR4056840}. It is worth recalling the Johnson--Koll\'ar criterion \cite{Johnson-Kollar2} where the authors showed that if $X_d \subset \bP(a_0,\ldots,a_{n+1})$ is a Fano weighted hypersurface of  index 1 with $a_0 \le \ldots \le a_{n+1}$ such that 
\[
d  < \frac{n+1}{n}a_0a_1,
\]
then $X_d$ admits a K\"ahler--Einstein metric. Such criterion is useful when all weights are relatively big. 

Assuming that $a_i \mid d$ for all $i$, in \cite[Corollary 1.5]{ST21} (see also \cite{LST}), it is shown that general Fano hypersurfaces of degree $d$ and $\iota_X < \dim X$ are K-stable.  Finally, the main result of \cite{ST24} says that given $X \subset \bP(a_0,\ldots,a_n)$, if there is $r$  such that $a_r >1$ and $a_r | d$, then 
$$
\delta(X; \cO_X(1)) \ge \frac{(n+1)a_{r}}{d}.
$$
This implies, in particular, that all smooth Fano weighted hypersurfaces of index 1 or 2 are K-stable. It also 
gives a different proof of the K-stability of all members of 82 out of 95 families of quasi-smooth terminal Fano $3$-fold hypersurfaces $X \subset \bP(a_0,\ldots,a_4)$ of index $1$.

\medskip

\textbf{Acknowledgment.}
L.C.\ was partially supported by \"{O}FG International Communication Grant IK-00001446. 
K.F.\ was supported by JSPS KAKENHI Grant Number 22K03269, 
Royal Society International Collaboration Award 
ICA\textbackslash 1\textbackslash 23109 and Asian Young Scientist Fellowship. 
T.S was partially supported by JSPS KAKENHI Grant Number JP23K03032 and Long-term Overseas
Teaching Fellowship Program for Young Researchers from Kobe university. 
L.T.\ was supported by PRIN2020 research grant ”2020KKWT53” and is a member of the GNSAGA group of INdAM.

\section{Preliminaries}\label{section:prelim}

We recall some notation and definitions from K-stability. Firstly, let us recall 
the notion of the (local) stability thresholds. Whenever we are dealing with log pairs 
$(X,\Delta)$, their boundaries $\Delta$ are always assumed to be effective and of 
rational coefficients. We follow the terminologies in \cite{Fujita:2024aa}.

\begin{definition}[{\cite{MR3896135, MR4067358, AZ22}}]\label{definition_delta}
Let $(X,\Delta)$ be a projective klt pair and let $\eta\in X$ be 
a scheme-theoretic point. Take any big $\bQ$-line bundle $L$ on $X$. 
For any prime divisor $F$ over $X$ (i.e., there exists a projective and 
birational morphism 
$\sigma\colon\tilde{X} \to X$ with $\tilde{X}$ normal such that 
$F$ can be realized as a prime $\bQ$-Cartier divisor 
on $\tilde{X}$), 
let $A_{X,\Delta}(F):=1+\ord_F(K_{\tilde{X}}-\sigma^*(K_X+\Delta))$ be 
the log discrepancy of $(X,\Delta)$ along $F$, and let us set 
\begin{eqnarray*}
S(L; F)&:=&\frac{1}{\vol_X(L)}\int_0^\infty\vol_{\tilde{X}}
(\sigma^*L-x F)dx, \\
T(L; F)&:=&\min\left\{x\in\bR_{>0}\,\,|\,\,\vol_{\tilde{X}}
(\sigma^*L-x F)=0\right\}. 
\end{eqnarray*}
The \emph{stability threshold} $\delta(X,\Delta; L)$ of $L$ 
(resp., \emph{the local stability threshold} $\delta_\eta(X,\Delta; L)$ of $L$ at $\eta$) 
is defined to be the infimum of the values 
\[
\frac{A_{X,\Delta}(F)}{S(L; F)},
\]
where $F$ run through 
all prime divisors over $X$ (resp., all prime divisors over $X$ such that 
the centers contain $\eta$). 
\end{definition}

We use the notion of the Veronese equivalence as follows.

\begin{definition}\label{defn:VeroneseEquiv}(\cite{Fujita2023}, 
\cite[Definition 2.1]{Fujita:2024aa}) 
Let $X$ be a normal projective variety and $L_1, \ldots , L_r$ be $\bQ$-Cartier $\bQ$-divisors on $X$, that is, $L_1, \ldots , L_r \in \CaCl(X) \otimes_{\bZ} \bQ=: \CaCl(X)_{\bQ}$, where $\CaCl(X)$ is the group of Cartier divisors on $X$ modulo linear equivalence. 
For $\vec{x} \in \bZ_{\ge 0}^r$ and $\vec{L}:= (L_1, \ldots , L_r) \in \CaCl (X)_{\bQ}^r$, let $\vec{x} \cdot \vec{L}:= \sum_{i=1}^r x_i L_i$. 
\begin{enumerate}
\item\label{item:mZgraded} Let $m \in \bZ_{>0}$ be such that $mL_1, \ldots , mL_r$ have lifts in $\CaCl(X)$ and fix them. 
We say that $V_{m \vecbul}$ is an {\it $(m \bZ_{\ge 0})^r$-graded linear series associated to $L_1, \ldots , L_r$} if 
it is a collection $\{V_{m \vec{a}} \}_{\vec{a} \in \bZ_{\ge 0}^r}$ of linear subspaces 
\[
V_{m \vec{a}} \subset H^0(X, m \vec{a} \cdot \vec{L})
\]
such that $V_{m \vec{0}} = \bC$ and $\dps{V_{m \vec{a}} \cdot V_{m \vec{b}} \subset V_{m(\vec{a} + \vec{b})}}$ for all $\vec{a} , \vec{b} \in \bZ_{\ge 0}^r$. \\ 
For $k \in \bZ_{>0}$, we can define a $(km \bZ_{\ge 0})^r$-graded linear series $V_{km \vecbul}$ by setting 
$
V_{km \vec{a}}:= V_{m (k\vec{a})} 
$ for all $\vec{a} \in \bZ_{\ge 0}^r$. 
\item Let $m \in \bZ_{>0}$ and $V_{m \vecbul}$ be as in (\ref{item:mZgraded}). 
Let $m' \in \bZ_{>0}$ and $V_{m' \vecbul}$ also be such objects (and fix lifts $m'L_1, \ldots , m'L_r \in \CaCl(X)$). 
\\ 
We say that $V_{m \vecbul}$ and $V_{m' \vecbul}$ are {\it Veronese equivalent} if there is $d \in mm' \bZ_{>0}$ such that the linear equivalence 
$(d/m)mL_i \sim (d/m')m'L_i$ holds for $1 \le i \le r$ and  
\[
V_{\left( \frac{d}{m}\right) m \vecbul} =  V_{\left(\frac{d}{m'}\right) m'\vecbul}
\]   
holds as $(d \bZ_{\ge 0})^r$-graded linear series under the above linear equivalence. 
We can check that this defines an equivalence relation (on the set of the graded linear series on $X$  
associated to $L_1, \ldots , L_r$ as in (\ref{item:mZgraded})). 
We write $V_{\vecbul}$ for the Veronese equivalence class of $V_{m \vecbul}$. 
\item Let $m \in \bZ_{>0}$ be as in (\ref{item:mZgraded}). Then we have an $(m\bZ_{\ge 0})^r$-graded linear series $H^0(m\vecbul \cdot \vec{L})$ 
defined by 
\[
H^0(m \vec{a} \cdot \vec{L}):= H^0(X, \vec{a} \cdot m \vec{L}). 
\]
We write $H^0(\vecbul \cdot \vec{L})$ for the class of $H^0(m \vecbul \cdot \vec{L})$ up to the Veronese equivalence and call it {\it the Veronese equivalence class of the complete linear series $H^0(\vecbul \cdot \vec{L})$}.    
\end{enumerate}
\end{definition}


\begin{definition}[{\cite{Fujita2023, 
Fujita:2024aa}}]\label{definition_plt-flag}
Let $(X,\Delta)$ be an $n$-dimensional normal projective variety. 
\begin{enumerate}
\item 
Let $F$ be a prime 
divisor over $X$. We say that $F$ is \emph{plt-type} over $(X,\Delta)$ 
if there exists a projective and birational morphism 
$\sigma\colon\tilde{X} \to X$ with $\tilde{X}$ normal such that 
$F$ can be realized as a prime $\bQ$-Cartier divisor 
on $\tilde{X}$ which is anti-ample over $X$ and the pair 
$(\tilde{X},\tilde{\Delta}+F)$ is plt, where $\tilde{\Delta}$ is defined 
to be 
\[
\sigma^*(K_X+\Delta)=K_{\tilde{X}}+\tilde{\Delta}+(1-A_{X,\Delta}(F))F. 
\]
We call the morphism $\sigma$
the \emph{plt blowup associated with $F$}. 
\item 
We say that 
\[
Y_\bullet\colon X=Y_0\triangleright Y_1 \triangleright\cdots\triangleright Y_j
\]
is a \emph{plt flag} over $(X,\Delta)$ if $(Y_0,\Delta_0):=(X,\Delta)$, 
$Y_i$ is a plt-type prime divisor over $(Y_{i-1},\Delta_{i-1})$ 
and the klt pair $(Y_i,\Delta_i)$ is inductively defined by 
$K_{Y_i}+\Delta_i:=(K_{\tilde{Y}_{i-1}}+\tilde{\Delta}_{i-1}+Y_i)|_{Y_i}$
for any $1\leq i\leq j$, where $\sigma_{i-1}\colon\tilde{Y}_{i-1}\to 
Y_{i-1}$ is the plt blowup of $(Y_{i-1},\Delta_{i-1})$ associated with 
$Y_i$. If moreover $j=n$, then we say that the plt flag over $(X,\Delta)$ 
is \emph{complete}. 
\item\label{item:refinement} 
Let $Y \subset X$ be a normal prime $\bQ$-Cartier divisor 
and $L_1, \ldots ,L_r$ $\bQ$-Cartier $\bQ$-divisors on $X$.  
Let $m \in \bZ_{>0}$ be such that there are lifts $mL_1, \ldots ,m L_r \in \CaCl(X)$ and $mY$ are Cartier. 
Let $V_{m\vec{\bullet}}$ be an $(m\bZ_{\ge 0})^r$-graded linear series on $X$ associated to $L_1, \ldots , L_r$.   
Then we define an $(m\bZ_{\ge 0})^{r+1}$-graded linear series $V^{(Y)}_{m \vec{\bullet}}$ associated to 
$
L_1|_Y, \ldots , L_r|_Y, -Y|_Y \in \CaCl(Y)_{\bQ}
$
by 
\[
V^{(Y)}_{m(\vec{a},j)}:= \Image \left( V_{m \vec{a}} \cap \left( jmY+ H^0(X, \vec{a} \cdot m \vec{L} - jm Y)\right) \to H^0(Y, \vec{a} \cdot m \vec{L}|_Y - jm Y|_Y) \right). 
\]
We see that the Veronese equivalence class of $V^{(Y)}_{m \vec{\bullet}}$ is well-defined by the class $V_{\vecbul}$ of $V_{m\vecbul}$.  
Hence we denote the class of $V^{(Y)}_{m \vec{\bullet}}$ by $V^{(Y)}_{\vec{\bullet}}$ 
and call it the {\it refinement of $V_{\vecbul}$ by $Y$}. 

\item 
Let $L_1, \ldots , L_r$, $V_{\vecbul}$ and $m \in \bZ_{>0}$ be as in (\ref{item:refinement}).  Let $Y$ be a plt-type prime divisor over $(X,\Delta)$ and $\sigma \colon \tilde{X} \to X$ be the associated plt blowup. 
Then we can naturally define an $(m\bZ_{\ge 0})^{r+1}$-graded linear series $\sigma^* V_{m\vec{\bullet}}$ by pulling back the $(m\bZ_{\ge 0})^r$-graded linear series $V_{m\vec{\bullet}}$ by $\sigma$ as in \cite[Definition 2.3]{Fujita:2024aa}. 
Then we see that the Veronese equivalence class $\sigma^* V_{\vecbul}$ of $\sigma^* V_{m\vecbul}$ is well-defined and call it the {\it pull-back of $V_{\vec{\bullet}}$}. 
We can also define the refinement 
$V^{(Y)}_{\vec{\bullet}}$ as the refinement of $\sigma^* V_{\vec{\bullet}}$ by $Y$, that is, 
\[
V^{(Y)}_{\vec{\bullet}}:= (\sigma^* V_{\vec{\bullet}})^{(Y)}. 
\]
\item
For a non-complete plt flag $Y_\bullet\colon X=Y_0\triangleright Y_1 
\triangleright\cdots\triangleright Y_j$, a prime divisor $F$ over $Y_j$, 
and a big $\bQ$-line bundle $L$ on $X$, we can define 
the value $S\left(L; Y_1\triangleright\cdots\triangleright Y_j
\triangleright F\right)$ by 
\[
S\left(L; Y_1\triangleright\cdots\triangleright Y_j
\triangleright F\right):= S\left( V^{\left(Y_1\triangleright\cdots\triangleright Y_j
\right)}_{\vec{\bullet}}; F \right)
\]
as in \cite[Definition 4.7]{Fujita:2024aa}, where 
$V_{\vec{\bullet}}^{\left(Y_1\right)}:= H^0(\bullet \cdot L)^{(Y_1)}$ is the 
refinement of the class of complete linear series associated to $L$ by $Y_1$, and 
the class of the series 
$V_{\vec{\bullet}}^{\left(Y_1\triangleright\cdots\triangleright Y_i\right)}$ 
for $2\leq i\leq j$ is inductively defined to be 
\[
V_{\vec{\bullet}}^{\left(Y_1\triangleright\cdots\triangleright Y_i\right)}
:=\left(V_{\vec{\bullet}}^{\left(Y_1\triangleright\cdots\triangleright 
Y_{i-1}\right)}\right)^{(Y_i)}.
\]

Let $\eta\in Y_j$ be a 
scheme-theoretic point. 
Similar to Definition \ref{definition_delta}, let 
\[
\delta\left(X,\Delta\triangleright Y_1\triangleright\cdots\triangleright 
Y_j; L\right) \text{ (resp., $\delta_\eta\left(X,\Delta\triangleright 
Y_1\triangleright\cdots\triangleright Y_j; L\right)$)}
\]
be the infimum of the values 
\[
\frac{A_{Y_j,\Delta_j}(F)}{S\left(L; 
Y_1\triangleright\cdots\triangleright Y_j\triangleright F\right)},
\]
where $F$ run through all prime divisors over $Y_j$ (resp., all 
prime divisors over $Y_j$ such that the centers contain $\eta$).

\item\label{item:Okbodydfn} 
Assume that $Y_\bullet\colon 
X=Y_0\triangleright Y_1 \triangleright\cdots\triangleright Y_n$ 
is a complete plt flag over $(X,\Delta)$, and take any big 
$\bQ$-line bundle $L$ on $X$. The \emph{Okounkov body 
$\mathbf{O}_{Y_\bullet}(L)\subset\bR_{\geq 0}^n$ of $L$ associated 
with the flag $Y_\bullet$} is defined as in 
\cite[Definition 4.9]{Fujita:2024aa}. Note that the definition is 
just a natural generalization of the standard definition in \cite{LM09}. 
Note that, as in \cite[Definition 4.9]{Fujita:2024aa} (cf.\ 
\cite[Proposition 3.12]{Fujita2023}), 
the value $S\left(L; Y_1\triangleright\cdots \triangleright Y_j\right)$
is equal to the $j$-th coordinate of the barycenter of 
$\mathbf{O}_{Y_\bullet}(L)$. 
\end{enumerate}
\end{definition}

\begin{remark}\label{remark_higher-S}
\begin{enumerate}
\item 
Assume that $(X,\Delta)$ is a klt log Fano pair (i.e., the divisor 
$-(K_X+\Delta)$ is ample). Then, as in 
\cite{MR3956698, MR3715806, MR3896135, MR4067358, LXZ22}, the pair 
$(X,\Delta)$ is K-stable (resp., K-semistable) if and only if 
$\delta(X,\Delta):=\delta\left(X,\Delta; -(K_X+\Delta)\right)>1$ (resp., $\geq 1$) 
holds. 
\item 
The value $S\left(L; Y_1\triangleright\cdots \triangleright Y_j\right)$ 
can be computed for several special cases. 
See \cite[\S 7 and \S 8]{Fujita:2024aa}, 
In fact, we will use \cite[Theorem 8.8]{Fujita:2024aa} in Section 
\ref{section_singular}. 
\end{enumerate}
\end{remark}

The following theorem is known to be a special case of 
\emph{Abban--Zhuang's method}; applying \cite[Theorem 3.2]{AZ22} for 
klt pairs to be $(Y_{j-1},\Delta_{j-1})$, plt blowups to be 
$\sigma_j$ and graded linear series to be 
$H^0(\bullet L)^{(Y_1\triangleright\cdots\triangleright Y_{j-1})}$. 
See also Lemma \ref{lem:cutting}.

\begin{theorem}[{\cite[Theorem 3.2]{AZ22} (see also \cite[Theorem 12.3]{Fujita:2024aa})}]\label{theorem:AZ}
Let $(X,\Delta)$ be a projective klt pair, let $L$ be a big $\bQ$-line bundle 
on $X$, let $Y_\bullet\colon 
X=Y_0\triangleright Y_1 \triangleright\cdots\triangleright Y_j$ be 
a plt flag over $(X,\Delta)=(Y_0,\Delta_0)$, let $\sigma_{i-1}\colon\tilde{Y}_{i-1}
\to Y_{i-1}$ be the associated plt blowups of $(Y_{i-1},\Delta_{i-1})$, and let 
$(Y_i,\Delta_i)$ be the naturally induced klt pair from the plt blowup for 
any $1\leq i\leq j$. Take any scheme-theoretic 
point $\eta\in Y_{j-1}$ such that the center of $Y_j$ contains $\eta$. 
Then we have 
\[
\delta_\eta\left(X,\Delta\triangleright 
Y_1\triangleright\cdots\triangleright Y_{j-1}; 
L\right)\geq\min\left\{\frac{A_{Y_{j-1},\Delta_{j-1}}(Y_j)}{S\left(
L; Y_1\triangleright\cdots\triangleright Y_j\right)},\quad
\inf_{\eta'}\delta_{\eta'}\left(X,\Delta\triangleright 
Y_1\triangleright\cdots\triangleright Y_j; 
L\right)
\right\},
\]
where $\eta'\in Y_j$ run through all scheme-theoretic points with 
$\sigma_{j-1}(\eta')=\eta$. 
\end{theorem}

\section{Blowing-up weighted projective spaces and their hypersurfaces}\label{section:w-blowup}
The main goal of this section is to introduce and study the notion of weighted blowups of weighted projective spaces. For the theory of toric varieties, we follow the notions in 
\cite{CLS}. 

\subsection{Weighted projective spaces}\label{section:wps}

In this subsection, we recall some facts about weighted projective spaces and their hypersurfaces, mainly to fix notation and for later reference. Almost all results are well-known to experts. 

\begin{definition}\label{definition_wps}
Let $s\in\bZ_{>0}$, and let $a_0,\dots,a_s\in\bZ_{>0}$ be with 
$\gcd\left(a_0,\dots,a_s\right)=1$. 
\begin{enumerate}
\item 
Consider $\bZ^{s+1}=\bigoplus_{i=0}^s\bZ\mathbf{e}_i$ and 
\[
N:=\bZ^{s+1}/\bZ\sum_{i=0}^s a_i\mathbf{e}_i.
\]
We usually denote the canonical projection by $\pi_N\colon
\bZ^{s+1}\twoheadrightarrow N$. 
Set $u_i:=\pi_N\left(\mathbf{e}_i\right)\in N$ for any $0\leq i\leq s$. 
Let $\Sigma$ be the fan in $N_\bR$ defined by 
\[
\Sigma:=\left\{\operatorname{Cone}(\mathcal{C})\subset N_\bR\,\,|\,\,
\mathcal{C}\subsetneq\left\{u_0,\dots,u_s\right\}\right\}.
\]
We set $\bP(a_0,\dots,a_s):=X_\Sigma$ as in \cite[Example 3.1.17]{CLS} and call it {\it a weighted projective space}. 
We usually set $D_i:=V(u_i)\subset\bP(a_0,\dots,a_s)$ for any 
$0\leq i\leq s$, where we set 
\[
V\left(u_{i_1},\dots,u_{i_k}\right):=
V\left(\bR_{\geq 0}u_{i_1}+\dots+\bR_{\geq 0}u_{i_k}\right)
\subset\bP(a_0,\dots,a_s)
\]
as in \cite[\S 3.2]{CLS}. 
\item 
We say that $(a_0,\dots,a_s)$ (or, $\bP(a_0,\dots,a_s)$) is 
\emph{well-formed} if $\gcd(a_0,\dots,\hat{a}_i,\dots,a_s)=1$
for any $0\leq i\leq s$. 
\end{enumerate}
\end{definition}

The following is well-known. 

\begin{proposition}[{\cite[Lemma 5.7 and 
Corollary 5.9]{Fletcher00}}]\label{proposition_well-formed}
Let $a_0,\dots,a_s\in\bZ_{>0}$ satisfying $\gcd(a_0,\dots,a_s)=1$. 
For any $0\leq i\leq s$, set 
\[
g_i:=\gcd\left(a_0,\dots,\hat{a}_i,\dots,a_s\right),\quad
g:=\prod_{i=0}^s g_i,\quad
a'_i:=\frac{a_ig_i}{g}. 
\]
\begin{enumerate}
\item 
The tuple $\left(a'_0,\dots,a'_s\right)$ is well-formed, and 
$\prod_{i=0}^s a_i=g^{s-1}\cdot\prod_{i=0}^sa'_i$ holds. 
Set 
\[
N':=\bZ^{s+1}/\bZ\sum_{i=0}^s
a'_i\mathbf{e}_i.
\]
Then 
$u'_i:=\pi_{N'}\left(\mathbf{e}_i\right)$
is a primitive element in $N'$ for any $0\leq i\leq s$. 
\item 
There is an isomorphism 
\[
\pi_{N',N}\colon N'\to N
\]
between lattices such that $g_iu_i$ maps to $u_i$ for any $0\leq i\leq s$. 
In particular, there is a natural isomorphism 
between $\bP\left(a'_0,\dots,a'_s\right)$ and $\bP(a_0,\dots,a_s)$. 
\end{enumerate}
\end{proposition}

\begin{definition}\label{definition_cox-wps}
Assume that $(a_0,\dots,a_s)$ is well-formed and set 
$\bP:=\bP(a_0,\dots,a_s)$. 
\begin{enumerate}
\item 
By \cite[Exercise 4.1.5]{CLS}, we have $\operatorname{Cl}(\bP)\simeq\bZ$. 
We usually denote by $\cO_\bP(1)\in \operatorname{Cl}(\bP)$
 the ample generator of $\operatorname{Cl}(\bP)$. 
If we regard $\cO_\bP(1)$ as an element in 
$\operatorname{{Pic}}(\bP)_\bQ$, then we write 
$\left[\cO_{\bP}(1)\right]\in\operatorname{{Pic}}(\bP)_\bQ$. 
\item 
We usually denote $\bC[x_0,\dots,x_s]$ the Cox ring (i.e., the total 
coordinate ring in the sense of \cite[\S 5.2]{CLS}) of $\bP$ such that 
$x_i$ corresponds to $u_i\in N$. In particular, $\bC[x_0,\dots,x_s]$ 
admits a $\bZ$-grading with $\deg x_i=a_i$. 
The variables $x_0,\dots,x_s$ are called by the \emph{(weighted) 
homogeneous coordinates} of $\bP$. 
By \cite[Proposition 5.3.7]{CLS}, we have the canonical isomorphism 
\[
\bigoplus_{d\in\bZ}H^0\left(\bP,\cO_\bP(d)\right)\simeq
\bC[x_0,\dots,x_s]
\]
of graded $\bC$-algebras. Moreover, the element 
$x_i\in H^0\left(\bP,\cO_\bP(a_i)\right)$ corresponds to the prime divisor 
$D_i\subset\bP$. In other words, $D_i=(x_i=0)\subset\bP$ holds.
More generally, for any nonzero homogeneous element 
$f\in\bC[x_0,\dots,x_s]$ of degree $d$, it gives the $\bQ$-Cartier Weil 
divisor $(f=0)\in|\cO_\bP(d)|$ on $\bP$ 
as in \cite[Proposition 5.3.7]{CLS}. We call it the 
\emph{(weighted) hypersurface} of degree $d$ defined by $f$. 
\end{enumerate}
\end{definition}

We see several basic properties of weighted projective spaces. 

\begin{proposition}\label{proposition_finite-wps}
Assume that $(a_0,\dots,a_s)$ is well-formed. We set 
$N:=\bZ^{s+1}/\bZ\sum_{i=0}^sa_i\mathbf{e}_i$, $u_i:=\pi_N(\mathbf{e}_i)$, 
$\bP:=\bP(a_0,\dots,a_s)$ as in Definitions \ref{definition_wps} and 
\ref{definition_cox-wps}. For any $0\leq i\leq s$, take $\bar{a}_i,e_i
\in\bZ_{>0}$ with $a_i=e_i\bar{a}_i$. Obviously, 
$\left(\bar{a}_0,\dots,\bar{a}_s\right)$ is well-formed. 
We set 
\[
\bar{N}:=\bZ^{s+1}/\bZ\sum_{i=0}^s\bar{a}_i\mathbf{e}_i,\quad
\bar{u}_i:=\pi_{\bar{N}}\left(\mathbf{e}_i\right)\in\bar{N}, \quad
\bar{\bP}:=\bP\left(\bar{a}_0,\dots,\bar{a}_s\right). 
\]
Set $f_{\bar{N},N}\colon\bZ^{s+1}\to \bZ^{s+1}$ defined by the diagonal matrix 
$\operatorname{diag}(e_0,\dots,e_s)$. 
Then we get a lattice homomorphism $\pi_{\bar{N},N}\colon \bar{N}\to N$
satisfying $\pi_{\bar{N},N}\circ\pi_{\bar{N}}=\pi_N\circ f_{\bar{N},N}$. 
From the homomorphism $\pi_{\bar{N}, N}$, we get the natural toric finite 
morphism $\tau\colon\bar{\bP}\to \bP$ of degree $\prod_{i=0}^s e_i$. 
\begin{enumerate}
\item 
We have 
\[
\tau^*\left[\cO_\bP(1)\right]=\left[\cO_{\bar{\bP}}(1)\right], \quad
K_{\bar{\bP}}=\tau^*\left(K_\bP+\sum_{i=0}^s
\left(1-\frac{1}{e_i}\right)D_i\right). 
\]
In particular, we have 
$\left(\cO_\bP(1)^{\cdot s}\right)=1/\prod_{i=0}^s a_i$.
\item 
For any nonzero homogeneous $f\in\bC\left[x_0,\dots,x_s\right]$ of degree 
$d$, the pullback $\tau^*(f=0)\subset\bar{\bP}$ of the Weil divisor 
$(f=0)\subset\bP$ satisfies that 
\[
\tau^*(f=0)=\left(f\left(\bar{x}_0^{e_0},\dots,
\bar{x}_s^{e_s}\right)=0\right)\subset\bar{\bP}, 
\]
where $\bC\left[\bar{x}_0,\dots,\bar{x}_s\right]$ be the Cox ring 
of $\bar{\bP}$. 
\end{enumerate}
\end{proposition}

\begin{proof}
Since $\pi_{\bar{N}, N}\left(\bar{u}_i\right)=e_iu_i$, we have 
$\tau^*D_i=e_i\bar{D}_i$ for any $0\leq i\leq s$. Moreover, the 
morphism $\tau$ is defined by 
\[
\left[\bar{x}_0:\cdots:\bar{x}_s\right]\mapsto
\left[\bar{x}_0^{e_0}:\cdots:\bar{x}_s^{e_s}\right]. 
\]
Thus the assertion follows. The equality 
$\left(\cO_\bP(1)^{\cdot s}\right)=1/\prod_{i=0}^s a_i$ follows directly 
if we set $e_i=a_i$ and $\bar{a}_i=1$ for all $0\leq i\leq s$. 
\end{proof}

\begin{proposition}\label{proposition_sublinear}
Assume that $(a_0,\dots,a_s)$ is well-formed and set 
$\bP:=\bP(a_0,\dots,a_s)$. Take any $0\leq r\leq s-2$, and set 
\[
Z:=\left(D_0\cap\dots\cap D_r\right)_{\operatorname{red}}
=V\left(u_0,\dots,u_r\right)\subset\bP.
\]
Moreover, for any $r+1\leq j\leq s$, let us set 
\begin{eqnarray*}
&&h:=\gcd(a_{r+1},\dots,a_s), \quad
a''_j:=\frac{a_j}{h}, \\
&&g_j:=\gcd\left(a''_{r+1},\dots,\hat{a}''_j,\dots,a''_s\right), \quad
g:=\prod_{j=r+1}^s g_j, \quad a'_j:=\frac{a''_j g_j}{g}. 
\end{eqnarray*}
Then we have $Z\simeq\bP':=\bP\left(a'_{r+1},\dots,a'_s\right)$. 
Moreover, we have 
\[
\left[\cO_\bP(1)\right]|_Z\sim_\bQ\left[\cO_{\bP'}\left(
\frac{1}{gh}\right)\right]
\]
under the above isomorphism. 
\end{proposition}

\begin{proof}
Set 
\[
\tau:=\bR_{\geq 0}u_0+\cdots+\bR_{\geq 0}u_r\in\Sigma.
\]
Note that $Z=V(\tau)$. Consider the sublattice 
$N_\tau:=N\cap\bR\tau\subset N$ and the quotient $N(\tau):=N/N_\tau$ 
as in \cite[Lemma 3.2.4]{CLS}. Note that we have a commutative diagram with the horizontal exact sequences  
\[
\xymatrix{
0 \ar[r] & \sum_{i=0}^r\bZ u_i \ar[r] \ar@{^{(}->}[d]^{\iota} & N \ar[r] \ar[d] & \bZ^{s-r}/\bZ\sum_{j=r+1}^s a_j\mathbf{e}_j  \ar[r] \ar@{->>}[d]^{\rho} & 0 \\
0 \ar[r] & N_{\tau} \ar[r] & N \ar[r] & N(\tau)  \ar[r] & 0
}. 
\]
and see that $\Coker \iota \simeq \Ker \rho$ 
by the snake lemma. 
Let $\operatorname{mult}(\tau)$ be the order of the cokernel. Since both $N(\tau)$ and 
$\bZ^{s-r}/\bZ\sum_{j=r+1}^s a''_j\mathbf{e}_j$ are torsion-free of 
rank $s-r-1$, we get the natural isomorphism 
\[
\bZ^{s-r}/\bZ\sum_{j=r+1}^s a''_j\mathbf{e}_j\simeq N(\tau)
\]
from the surjection $\rho$, and $\operatorname{mult}(\tau)=h$ holds. 
By \cite[Theorem 3.2.6]{CLS} and Proposition 
\ref{proposition_well-formed}, we get $Z\simeq\bP'$. 
Set $\alpha\in\bQ_{>0}$ satisfying $\left[\cO_\bP(1)\right]|_Z\sim_\bQ
\left[\cO_{\bP'}(\alpha)\right]$. By \cite[p.~100]{Fulton93}, we have 
\[
D_0\cdots D_r=\frac{1}{\operatorname{mult}(\tau)}Z=\frac{1}{h}Z.
\]
Thus, by Proposition \ref{proposition_finite-wps}, we get 
\begin{eqnarray*}
&&\frac{h}{a_{r+1}\cdots a_s}=h\left(\cO_\bP(1)^{\cdot s-r-1}\cdot
\cO_\bP(a_0)\cdots\cO_\bP(a_r)\right)\\
&=&\left(\cO_\bP(1)^{\cdot s-r-1}\cdot Z\right)
=\frac{\alpha^{s-r-1}}{a'_{r+1}\cdots a'_s}
=\frac{g^{s-r-1}h^{s-r}\alpha^{s-r-1}}{a_{r+1}\cdots a_s}.
\end{eqnarray*}
Thus we get the equality $\alpha=1/(gh)$. 
\end{proof}

\begin{lemma}\label{lemma_aut-wps}
Assume that $(a_0,\dots,a_s)$ is well-formed and set 
$\bP:=\bP(a_0,\dots,a_s)$. Take 
$\varphi\in H^0\left(\bP,\cO_\bP(a_0)\right)$ general. 
(More generally, $\varphi\in H^0\left(\bP,\cO_\bP(a_0)\right)$ is assumed 
to be quasi-linear in the sense of \cite[Definition 2.17]{KOW23}.)
Then 
there exists an isomorphism $\iota\colon\bP\to\bP$ such that 
$\iota^*D_0=\left(\varphi=0\right)$. 
\end{lemma}

\begin{proof}
Since $\varphi$ is general, we may assume that 
$\varphi=x_0+\varphi_0(x_1,\dots,x_s)$ with 
$\varphi_0(x_1,\dots,x_s)\in 
H^0\left(\bP,\cO_\bP(a_0)\right)$. 
If we take $\iota$ to be 
\[
\left[x_0:x_1:\cdots:x_r\right]\mapsto
\left[x_0+\varphi_0(x_1,\dots,x_s):x_1:\cdots:x_r\right]
\]
then the $\iota$ is obviously an isomorphism, and 
we get $\iota^*D_0=\left(\varphi=0\right)$. 
\end{proof}

\begin{proposition}\label{proposition_divisor_wps}
Assume that $s\geq 2$, 
$(a_0,\dots,a_s)$ is well-formed, and set 
$\bP:=\bP(a_0,\dots,a_s)$. For any $0\leq i_1<i_2\leq s$, we set 
\[
D_{i_1i_2}:=V\left(u_{i_1},u_{i_2}\right)\subset\bP. 
\]
We focus on the divisor $D_0\subset\bP$. For any $1\leq i\leq s$, 
let us set 
\[
g_i:=\gcd\left(a_1,\dots,\hat{a}_i,\dots,a_s\right), \quad
g:=\prod_{i=1}^s g_i, \quad a'_i:=\frac{a_ig_i}{g}. 
\]
(By Proposition \ref{proposition_sublinear}, we have 
$D_0\simeq\bP\left(a'_1,\dots,a'_s\right)$ and 
$\left[\cO_\bP(1)\right]|_{D_0}\sim_\bQ\left[\cO_{D_0}(1/g)\right]$.)
\begin{enumerate}
\item
For any $1\leq i\leq s$, we have $D_0\cdot D_i=\frac{1}{g_i}D_{0i}$. 
In particular, we have 
\[
\left(K_\bP+D_0\right)|_{D_0}=K_{D_0}+\sum_{i=1}^s
\left(1-\frac{1}{g_i}\right)D_{0i}. 
\]
\item
Take any nonzero homogeneous $f\in\bC\left[x_0,\dots,x_s\right]$ of 
degree $d$, and set $X:=(f=0)\subset\bP$. Assume that $D_0\not\subset X$ 
and $(a_1,\dots,a_s)$ is well-formed. 
Then, the $\bQ$-Cartier $\bQ$-divisor $X|_{D_0}$ on $D_0$ is the Weil 
divisor in $|\cO_{\bP(a_1,\dots,a_s)}(d)|$ defined by the 
homogeneous polynomial 
\[
f|_{D_0}:=f(0,x_1,\dots,x_s)\in\bC[x_1,\dots,x_s]. 
\]
\end{enumerate}
\end{proposition}

\begin{proof}
As in the proof of Proposition \ref{proposition_sublinear}, we have 
\[
\operatorname{mult}\left(\bR_{\geq 0}u_0+\bR_{\geq 0}u_i\right)=g_i. 
\]
Thus the assertion (1) follows from \cite[p.~100]{Fulton93}
and the proof of \cite[Lemma 2.3.1]{FS2020}. 

Let us consider (2). We note that $g_1=\cdots=g_s=1$ from the assumption. 
Thus, by (1), both $D_0$ and $D_i$ are Cartier divisors at the 
generic point of $D_{0i}$. (In fact, by \cite[5.15]{Fletcher00}, 
it is known that $\bP$ is nonsingular at the generic point 
of $D_{0i}$.) Thus the $\bQ$-Cartier $\bQ$-divisor $X|_{D_0}$ is 
equal to the Weil divisor $\left(f(0,x_1,\dots,x_s)=0\right)$
for all codimension $1$ points in $D_0$. Thus they are equal 
to each other. 
\end{proof}

\begin{remark}\label{remark_divisor-wps}
In general, $X|_{D_0}$ is not a Weil divisor. For example, if 
$\bP=\bP(1,1,2)$ and $X=(x_1=0)$, then $X|_{D_0}$ is a $\bQ$-divisor 
of degree $1/2$ on $D_0\simeq\bP^1$. 
\end{remark}

\begin{definition}\label{definition_Ba}
Assume that $\left(a_0,\dots,a_s\right)$ is well-formed, and 
set $\bP:=\bP\left(a_0,\dots,a_s\right)$. Take any $p\in\bP$. 
\begin{enumerate}
\item 
For any $0\leq i\leq s$, let 
\[
H^0\left(\bP,\cO_{\bP}(a_i)_p\right)\subset
H^0\left(\bP,\cO_{\bP}(a_i)\right)
\]
be the subspace consisting of the homogeneous polynomials vanishing at 
$p$. We also write $\left|\cO_{\bP}(a_i)_p\right|\subset
\left|\cO_{\bP}(a_i)\right|$ for the associated sub-linear series. 
\item
For any $a\in\bZ_{>0}$, let us set 
\[
B_a^\bP:=\bigcap_{a_i\leq a}
\operatorname{Bs}\left|\cO_{\bP}(a_i)\right|\subset\bP
\]
with the reduced structure. For example, $B_1^\bP\subset\bP$ is the 
reduced structure of $\operatorname{Bs}\left|\cO_\bP(1)\right|$. 
Similarly, let us set 
\[
B_{a,p}^\bP:=\bigcap_{a_i\leq a}
\operatorname{Bs}\left|\cO_{\bP}(a_i)_p\right|\subset\bP
\]
with the reduced structure. 
\end{enumerate}
\end{definition}

\begin{lemma}\label{lemma_Ba-taro}
Under the assumption in Definition \ref{definition_Ba}, 
for any $0\leq i\leq s$, we have 
\[
\operatorname{Bs}\left|\cO_\bP(a_i)_p\right|\subset B_1^\bP\cup 
B_{a_i,p}^\bP. 
\]
\end{lemma}

\begin{proof}
We may assume that 
\[
1=a_0=\cdots=a_{c_1-1}<a_{c_1}\leq\cdots\leq a_s
\]
with $0\leq c_1\leq s$. If $c_1=0$, then $B_1^\bP=\bP$. Thus we may assume 
that $c_1\geq 1$. Since 
\[
\operatorname{Bs}\left|\cO_\bP(a_i)_p\right|\subset
\operatorname{Bs}\left|\cO_\bP(1)_p\right|, 
\]
we may further assume that $p\not\in B_1^\bP$. 
After coordinate changes, we may also assume that $a_i<a_{i+1}$ 
(or $i=s$), and 
$ p=\left[1:0:\cdots:0\right]\in\bP
$. Then we have 
\[
B_1^\bP=\left(x_0=\cdots=x_{c_1-1}=0\right),\quad
B_{a_i,p}^\bP=\left(x_1=\cdots=x_i=0\right).
\]
Take any point 
\[
q=\left[q_0:\cdots:q_s\right]\in\bP\setminus
\left(B_1^\bP\cup B_{a_i,p}^\bP\right). 
\]
If $q_0\neq 0$, then there exists $1\leq j\leq i$ such that $q_j\neq 0$. 
Since the polynomial $x_0^{a_i-a_j}\in H^0\left(\bP,\cO_\bP(a_i)_p\right)$
does not vanish at $q$, we have $q\not\in\operatorname{Bs}\left|
\cO_\bP(a_i)_p\right|$. If $q_0=0$, then there exists 
$1\leq j\leq c_1-1$ such that $q_j\neq 0$. 
Since the polynomial $x_j^{a_i}\in H^0\left(\bP,\cO_\bP(a_i)_p\right)$
does not vanish at $q$, we have $q\not\in\operatorname{Bs}\left|
\cO_\bP(a_i)_p\right|$. 
\end{proof}

\subsection{Standard weighted blowups}\label{section_swb}

We define the notion of standard weighted blowups of weighted projective spaces, and see basic properties. Throughout Subsection \ref{section_swb}, we fix the following: 

\begin{setup}\label{setup:swb}
Let $s,r\in\bZ_{>0}$ with $1\leq r\leq s-1$, and let 
$a_0,\dots,a_s\in\bZ_{>0}$ with $(a_0,\dots,a_s)$ well-formed. 
For any $0\leq i\leq r$, we set 
\begin{eqnarray*}
&&h:=\gcd\left(a_0,\dots,a_r\right), \quad
a''_i:=\frac{a_i}{h}, \\
&&g_i:=\gcd\left(a''_0,\dots,\hat{a}''_i,\dots,a''_r\right), \quad
g:=\prod_{i=0}^r g_i, \quad
a'_i:=\frac{a''_ig_i}{g}.
\end{eqnarray*}
Similarly, for any $r+1\leq j\leq s$, we set 
\begin{eqnarray*}
&&h':=\gcd\left(a_{r+1},\dots,a_s\right), \quad
a''_j:=\frac{a_j}{h'}, \\
&&g_j:=\gcd\left(a''_{r+1},\dots,\hat{a}''_j,\dots,a''_s\right), \quad
g':=\prod_{j=r+1}^s g_j, \quad
a'_j:=\frac{a''_jg_j}{g'}.
\end{eqnarray*}
Moreover, as in Definition \ref{definition_wps}, we fix the canonical surjection 
\[
\pi_N\colon \bZ^{s+1}=\bigoplus_{i=0}^s\bZ\mathbf{e}_i
\twoheadrightarrow N:=\bZ^{s+1}/\bZ\sum_{i=0}^s a_i\mathbf{e}_i, 
\]
set $u_i:=\pi_N(\mathbf{e}_i)$ $(0\leq i\leq s)$, and consider the fan 
\[
\Sigma:=\left\{\operatorname{Cone}(\mathcal{C})\subset N_\bR\,\,|\,\,
\mathcal{C}\subsetneq\left\{u_0,\dots,u_s\right\}\right\}
\]
in $N_\bR$, set $\bP:=X_\Sigma=\bP(a_0,\dots,a_s)$, and 
let $\bC[x_0,\dots,x_s]$ be the Cox ring of $\bP$ with $\deg x_i=a_i$. 
Set $D_i:=(x_i=0)\subset\bP$ for any $0\leq i\leq s$. 
We also fix the canonical surjection 
\[
\pi_{N'}\colon \bZ^{r+1}=\bigoplus_{i=0}^r\bZ\mathbf{e}_i
\twoheadrightarrow N':=\bZ^{r+1}/\bZ\sum_{i=0}^r a'_i\mathbf{e}_i, 
\]
set $u'_i:=\pi_{N'}(\mathbf{e}_i)$ $(0\leq i\leq r)$, and consider the fan 
\[
\Sigma':=\left\{\operatorname{Cone}(\mathcal{C}')\subset N'_\bR\,\,|\,\,
\mathcal{C}'\subsetneq\left\{u'_0,\dots,u'_r\right\}\right\}
\]
in $N'_\bR$, set $\bP':=X_{\Sigma'}=\bP(a'_0,\dots,a'_r)$, and 
let $\bC[x'_0,\dots,x'_r]$ be the Cox ring of $\bP'$ with 
$\deg x'_i=a'_i$. Set $D'_i:=(x'_i=0)\subset\bP'$ for any $0\leq i\leq r$. 
\end{setup}

Note that, by Proposition \ref{proposition_well-formed}, all $u_i\in N$
(resp., all $u'_i\in N'$) are primitive elements. 

\begin{lemma}\label{lemma_ray}
The element 
\[
v_{s+1}:=-\frac{1}{h}\sum_{j=r+1}^s a''_j u_j
=\frac{1}{h'}\sum_{i=0}^r a''_i u_i \in N_\bQ
\]
is a primitive element in $N$. We note that the element $v_{s+1}\in N$ 
can be characterized as a primitive element in $N$ such that 
\[
v_{s+1}\in \sum_{i=0}^r\bQ_{\geq 0}u_i\quad\text{and}\quad
v_{s+1}\in\left(\sum_{i=0}^r\bQ u_i\right)
\cap\left(\sum_{j=r+1}^s\bQ u_j\right) \simeq \bQ.
\]
\end{lemma}

\begin{proof}
We firstly see that $v_{s+1}$ is an element in $N$. Since $\gcd(h,h')=1$, 
there exist $k,k'\in\bZ$ such that $h'k-hk'=1$ holds. Set 
\[
\left(c_0,\dots,c_r\right):=k\left(a''_0,\dots,a''_r\right), \quad
\left(c_{r+1},\dots,c_s\right):=k'\left(a''_{r+1},\dots,a''_s\right). 
\]
Then we can directly check that 
\[
\pi_N\left(\sum_{i=0}^sc_i\mathbf{e}_i\right)= k(h' v_{s+1}) + k'(-hv_{s+1})= v_{s+1}. 
\]
Thus, it is enough to show that $v_{s+1}\in N$ is primitive. 
Take any $d\in\bZ_{>0}$, $q\in\bZ$ and 
$\left(c'_0,\dots,c'_s\right)\in\bZ^{s+1}$ satisfying 
\[
d\left(c'_0,\dots,c'_s\right)-q(a_0,\dots,a_s)=(c_0,\dots,c_s). 
\]
It is enough to show that $d=1$. Since 
\begin{eqnarray*}
d\left(c'_0,\dots,c'_r\right)&=&(hq+k)\left(a''_0,\dots,a''_r\right), \\
d\left(c'_{r+1},\dots,c'_s\right)&=&
(h'q+k')\left(a''_{r+1},\dots,a''_s\right), 
\end{eqnarray*}
we have $d\mid (hq+k)$ and $d\mid (h'q+k')$. Since 
$h'(hq+k)-h(h'q+k')=1$, we must have $d=1$.
\end{proof}

\begin{lemma}\label{lemma_ray-ray}
For any $0\leq i\leq s$, we have 
\[
\operatorname{mult}\left(\bR_{\geq 0}u_i+\bR_{\geq 0}v_{s+1}\right)=g_i.
\]
\end{lemma}

\begin{proof}
We only consider the case $0\leq i\leq r$, since the other case 
$r+1\leq i\leq s$ can be proven in a completely same way. 
We may assume that $i=0$. From the definition of multiplicity, we have 
\begin{eqnarray*}
\operatorname{mult}\left(\bR_{\geq 0}u_0+\bR_{\geq 0}v_{s+1}\right)
&=&\frac{1}{h'}\operatorname{mult}\begin{pmatrix}
1 & 0 & \cdots & 0 & 0 & \cdots & 0\\ 
a''_0 & a''_1& \cdots & a''_r & 0 & \cdots & 0\\
a_0 & a_1& \cdots & a_r & a_{r+1} & \cdots & a_s
\end{pmatrix}\\
&=&\operatorname{mult}
\begin{pmatrix}
1 & 0 & \cdots & 0 & 0 & \cdots & 0\\ 
0 & a''_1& \cdots & a''_r & 0 & \cdots & 0\\
0 & 0& \cdots & 0 & a''_{r+1} & \cdots & a''_s
\end{pmatrix}=\frac{1}{g_0},
\end{eqnarray*}
since $\left(a''_1,\dots,a''_r\right)$ is divisible exactly $g_0$ times. 
\end{proof}

\begin{definition}\label{definition_swb}
Set 
\[
v_j:=\begin{cases}
u_j & \text{if }r+1\leq j\leq s, \\
v_{s+1} & \text{for }j=s+1.
\end{cases}
\]
Let $\tilde{\Sigma}$ be the star subdivision of $\Sigma$ by $v_{s+1}$. 
In other words, the set of maximal dimensional cones of $\tilde{\Sigma}$ 
consists of 
\[
\operatorname{Cone}\left(u_0,\dots,\hat{u}_i,\dots,u_r,
v_{r+1},\dots,\hat{v}_j,\dots,v_{s+1}\right)
\]
with $0\leq i\leq r$ and $r+1\leq j\leq s+1$. 
We set $\tilde{\bP}:=X_{\tilde{\Sigma}}$, and let 
$\psi\colon\tilde{\bP}\to\bP$ be the natural toric morphism. 
We call $\psi$ the \emph{standard weighted blowup of $\bP$ along 
$\bP(a_{r+1},\dots,a_s)$} (or, \emph{along $(x_0=\cdots=x_r=0)$}). 
We usually denote 
$
\bC\left[x_0,\dots,x_r,y_{r+1},\dots,y_s,z\right]
$
the Cox ring of $\tilde{\bP}$, where 
\begin{eqnarray*}
\tilde{D}_i&:=&V(u_i)=(x_i=0) \quad (0\leq i\leq r), \\
\tilde{D}_j&:=&V(v_j)=(y_j=0) \quad (r+1\leq j\leq s), \\
\bE&:=&V(v_{s+1})=(z=0).
\end{eqnarray*}
Note that the divisor $\bE\subset\tilde{\bP}$ is the exceptional divisor 
of $\psi$ and $\psi(\bE)=(x_0=\cdots=x_r=0)\subset\bP$. 
\end{definition}

\begin{definition}\label{definition_wbs}
Consider the canonical surjection 
\[
\pi_{N''}\colon 
\bZ^{r+1}=\bigoplus_{i=0}^r\bZ\mathbf{e}_i\twoheadrightarrow
N'':=\bZ^{r+1}/\bZ\sum_{i=0}^r a''_i\mathbf{e}_i 
\]
and set $u''_i:=\pi_{N''}(\mathbf{e}_i)$ for any $0\leq i\leq r$. 
The standard projection $\bZ^{s+1}\to\bZ^{r+1}$ onto the first $(r+1)$ coordinates induces the surjection 
$\pi_{N,N''}\colon N\twoheadrightarrow N''$. Moreover, 
by Proposition \ref{proposition_well-formed}, there exists an 
isomorphism $\pi_{N', N''} \colon N'\simeq N''$ with $g_i u'_i\mapsto u''_i$ for any 
$0\leq i\leq r$. 
Thus, the composition $\pi_{N,N'}:= (\pi_{N', N''})^{-1} \circ \pi_{N, N''} \colon N\twoheadrightarrow N'$ satisfies 
\begin{eqnarray*}
u_i&\mapsto& g_i u'_i \quad (0\leq i\leq r), \\
v_j &\mapsto& 0 \quad (r+1\leq j\leq s+1). 
\end{eqnarray*}
Hence the homomorphism $\pi_{N,N'}$ is compatible with the fans 
$\tilde{\Sigma}$ and $\Sigma'$ (in the sense of 
\cite[Definition 3.3.1]{CLS}). The induced toric morphism 
$\pi\colon\tilde{\bP}\to \bP'$ is said to be the 
\emph{induced weighted bundle structure} of $\tilde{\bP}$. 
Since $\pi_{N, N'}$ is surjective, we have 
$\pi_*\cO_{\tilde{\bP}}=\cO_{\bP'}$. 
\end{definition}

\begin{proposition}\label{proposition_degree}
There is an isomorphism $\operatorname{Cl}\left(\tilde{\bP}\right)
\simeq\bZ^2$ such that 
\begin{itemize}
\item 
the divisor $\tilde{D}_i$ maps to $(a''_i,0)$ for any $0\leq i\leq r$, 
\item 
the divisor $\tilde{D}_j$ maps to $(0,a''_j)$ for any $r+1\leq j\leq s$, 
and 
\item 
the divisor $\bE$ maps to $(-h',h)$. 
\end{itemize}
\end{proposition}

\begin{proof}
Set $\left(\bZ^{s+1}\right)^\vee=\bigoplus_{i=0}^s\bZ\mathbf{e}_i^\vee$, 
where $\left\{\mathbf{e}_i^\vee\right\}_{i=0}^s$ is the dual basis 
of $\left\{\mathbf{e}_i\right\}_{i=0}^s$. Then $M:=N^\vee$ is equal to 
\[
\left\{\sum_{i=0}^sm_i\mathbf{e}_i^\vee\in\left(\bZ^{s+1}\right)^\vee
\,\,\Big|\,\,\sum_{i=0}^s m_ia_i=0\right\}. 
\]
Set 
\[
\bZ^{s+2}=\bZ^{\tilde{\Sigma}(1)}:=\bigoplus_{i=0}^r\bZ[u_i]\oplus
\bigoplus_{j=r+1}^{s+1}\bZ[v_j]. 
\]
By \cite[Theorem 4.1.3]{CLS}, there is an exact sequence 
\[
0\to M\xrightarrow{\iota_{\tilde{\Sigma}}}
\bZ^{\tilde{\Sigma}(1)}\to\operatorname{Cl}\left(\tilde{\bP}\right)\to 0,
\]
where 
\[
\iota_{\tilde{\Sigma}}(m)=\left(\langle m,u_0\rangle,\dots,
\langle m,u_r\rangle,\langle m,v_{r+1}\rangle,\dots,
\langle m,v_{s+1}\rangle\right).
\]
Consider the homomorphism $\tilde{\iota}_{\tilde{\Sigma}}\colon
\left(\bZ^{s+1}\right)^\vee\to\bZ^{\tilde{\Sigma}(1)}$
given by 
\[
\begin{pmatrix}
1 & & \\
& \ddots& \\
&& 1 \\
c_0 &\cdots & c_s 
\end{pmatrix}, 
\]
where $(c_0,\dots,c_s)\in\bZ^{s+1}$ be as in the proof of Lemma 
\ref{lemma_ray}. Then, the homomorphism $\iota_{\tilde{\Sigma}}$ is 
equal to the restriction of $\tilde{\iota}_{\tilde{\Sigma}}$ to $M$. 
Therefore, we have the following exact diagram: 
\[
\xymatrix{
& 0 \ar[d] & 0 \ar[d] & & \\
0 \ar[r] & M \ar[r]^-{\iota_{\tilde{\Sigma}}} \ar[d] & 
\bZ^{\tilde{\Sigma}(1)} 
\ar[d]^-{\operatorname{id}} \ar[r] 
& \operatorname{Cl}\left(\tilde{\bP}\right) \ar[r] \ar[d] &0\\
0 \ar[r] &\left(\bZ^{s+1}\right)^\vee 
\ar[r]_-{{\tilde{\iota}}_{\tilde{\Sigma}}} \ar[d]_-{(a_0,\dots,a_s)}
& \bZ^{\tilde{\Sigma}(1)} \ar[r]_-{(-c_0,\dots,-c_s,1)} \ar[d] & \bZ 
\ar[d] \ar[r] & 0\\
&\bZ \ar[d] & 0 &0&\\
&0.&&&
}
\]
By the snake lemma, we get the following short exact sequence: 
\[
0\to \bZ\to \operatorname{Cl}\left(\tilde{\bP}\right)\to\bZ\to 0. 
\]
By easy diagram chases, there exists an isomorphism 
$\operatorname{Cl}\left(\tilde{\bP}\right)\simeq\bZ^2$ 
such that 
\begin{eqnarray*}
\left[u_i\right] &\mapsto& a''_i(h,k)\quad (0\leq i\leq r), \\
\left[v_j\right] &\mapsto& a''_j(h',k')\quad (r+1\leq j\leq s), \\
\left[v_{s+1}\right] &\mapsto& (0,-1),
\end{eqnarray*}
where $\left[u_i\right], \left[v_j\right], \left[v_{s+1}\right] \in \Cl (\tilde{\bP})$ are induced by the natural surjection $\bZ^{\tilde{\Sigma}(1)}\to\operatorname{Cl}(\tilde{\bP})$ and  $k,k'\in\bZ$ are as in the proof of Lemma \ref{lemma_ray}. 
Together with the isomorphism 
\[
\begin{pmatrix}
-k' & h'\\
k & -h
\end{pmatrix}\colon\bZ^2\to\bZ^2,
\]
we get the assertion. 
\end{proof}

\begin{definition}\label{definition_cl}
From the isomorphism 
$\operatorname{Cl}\left(\tilde{\bP}\right)\simeq\bZ^2$ in 
Proposition \ref{proposition_degree}, we write the element 
$\cO_{\tilde{\bP}}(\alpha,\beta)\in
\operatorname{Cl}\left(\tilde{\bP}\right)$ which 
corresponds to $(\alpha,\beta)\in\bZ^2$. Moreover, the 
$\bZ^2$-grading of the Cox ring 
$\bC[x_0,\dots,x_r,y_{r+1},\dots,y_s,z]$ is, 
\[
\deg x_i=\left(a''_i,0\right), \quad
\deg y_j=\left(0,a''_j\right), \quad
\deg z=\left(-h',h\right).
\]
\end{definition}

\begin{lemma}\label{lemma_cl-cl}
For any $0\leq i\leq r$, we have \[
\psi^*D_i=\tilde{D}_i+\frac{a_i}{hh'}\bE,\quad
\pi^*D'_i=g_i\tilde{D}_i.
\]
For any $r+1\leq j\leq s$, we have $\psi^*D_j=\tilde{D}_j$. 
In particular, we have 
\[
\psi^*\left[\cO_\bP(1)\right]=\left[\cO_{\tilde{\bP}}
\left(0,\frac{1}{h'}\right)\right], \quad
\pi^*\left[\cO_{\bP'}(1)\right]=\left[\cO_{\tilde{\bP}}
\left(g,0\right)\right].
\]
Moreover, the morphisms $\psi$ and $\pi$
are obtained by 
\begin{eqnarray*}
\psi\colon\left[x_0:\cdots:x_r; y_{r+1}:\cdots:y_s:z\right]
&\mapsto& \left[x_0:\cdots:x_r:y_{r+1}:\cdots:y_s\right],\\
\pi\colon\left[x_0:\cdots:x_r; y_{r+1}:\cdots:y_s:z\right]
&\mapsto& \left[x_0^{g_0}:\cdots:x_r^{g_r}\right].
\end{eqnarray*}
\end{lemma}

\begin{proof}
Note that $v_{s+1}=\frac{1}{h'}\sum_{i=0}^ra''_i u_i$ 
and the surjection $\pi_{N,N'}\colon N\twoheadrightarrow N'$
maps $u_i$ to $g_i u'_i$. Thus the assertion is trivial. 
\end{proof}

\begin{proposition}\label{proposition_pi-star}
For any $d\in\bZ_{\geq 0}$, there is a natural isomorphism 
\begin{eqnarray*}
\pi^*\colon H^0\left(\bP',\cO_{\bP'}(d)\right)
&\to& H^0\left(\tilde{\bP}, \cO_{\tilde{\bP}}(gd,0)\right), \\
f\left(x'_0,\dots,x'_r\right)&\mapsto&
f\left(x_0^{g_0},\dots,x_r^{g_r}\right).
\end{eqnarray*}
(Thus, we often write $\pi^*\left|\cO_{\bP}(d)\right|$ 
in place of $\left|\cO_{\tilde{\bP}}(gd,0)\right|$.)
\end{proposition}

\begin{proof}
Since 
\begin{eqnarray*}
&&\left\{f\in\bC\left[x_0,\dots,x_r,y_{r+1},\dots,y_s,z\right]
\,\,|\,\,\deg f=(gd,0)\right\}\\
&\simeq&
\left\{f\in\bC\left[x''_0,\dots,x''_r\right]\,\,|\,\,
\deg f= gd\right\}
\end{eqnarray*}
with $\deg x''_i=a''_i$, we get the assertion by 
\cite[Lemma 5.7]{Fletcher00}. 
\end{proof}

In the following, we shall see that the standard weighted blow-up and the induced weighted bundle structure can be pulled back to the toric finite cover $\tau\colon\bar{\bP}\to\bP$ as in Proposition \ref{proposition_finite-wps}. 
 
\begin{lemma}\label{lemma_finite-swb}
For any $0\leq i\leq s$, take $e_i, \bar{a}_i\in\bZ_{>0}$ 
such that $a_i=e_i\bar{a}_i$. Note that 
$\left(\bar{a}_0,\dots,\bar{a}_s\right)$ is well-formed. 
We set $\bar{\bP}:=\bP\left(\bar{a}_0,\dots,\bar{a}_s\right)$
and let $\tau\colon\bar{\bP}\to\bP$ be the toric finite morphism as in Proposition 
\ref{proposition_finite-wps}. 
Let $\bar{\psi}\colon\tilde{\bar{\bP}}\to\bar{\bP}$ be the 
standard weighted blowup of $\bar{\bP}$ along 
$\bP\left(\bar{a}_{r+1},\dots,\bar{a}_s\right)$ and 
let $\bar{\bE}\subset\tilde{\bar{\bP}}$ be the exceptional 
divisor of $\bar{\psi}$. 
Moreover, let $\bar{\pi}\colon\tilde{\bar{\bP}}\to\bar{\bP}'$ be the 
induced weighted bundle structure of $\tilde{\bar{\bP}}$. 
For any $0\leq i\leq r$, set 
\begin{eqnarray*}
&&\bar{h}:=\gcd\left(\bar{a}_0,\dots,\bar{a}_r\right), \quad
\bar{a}_i'':=\frac{\bar{a}_i}{\bar{h}}, \\
&&\bar{g}_i:=\gcd\left(\bar{a}_0'',\dots,\hat{\bar{a}}''_i,\dots,
\bar{a}''_r\right), \quad
\bar{g}:=\prod_{i=0}^r\bar{g}_i, \quad
\bar{a}'_i:=\frac{\bar{a}''_i\bar{g}_i}{\bar{g}}.
\end{eqnarray*}
For any $r+1\leq j\leq s$, set 
\begin{eqnarray*}
&&\bar{h}':=\gcd\left(\bar{a}_{r+1},\dots,\bar{a}_s\right), \quad
\bar{a}_j'':=\frac{\bar{a}_j}{\bar{h}'}, \\
&&\bar{g}_j:=\gcd\left(\bar{a}_{r+1}'',\dots,\hat{\bar{a}}''_j,\dots,
\bar{a}''_s\right), \quad
\bar{g}':=\prod_{j=r+1}^s\bar{g}_j, \quad
\bar{a}'_j:=\frac{\bar{a}''_j\bar{g}_j}{\bar{g}'}.
\end{eqnarray*}
\begin{enumerate}
\item 
There are finite toric morphisms $\tilde{\tau}\colon\tilde{\bar{\bP}}\to
\tilde{\bP}$ and $\tau'\colon\bar{\bP}'\to \bP'$ 
such that the diagram commutes: 
\[
\xymatrix{
\tilde{\bar{\bP}} \ar[r]^-{\bar{\pi}} \ar[d]_-{\bar{\psi}} 
\ar[rd]^-{\tilde{\tau}} & \bar{\bP}' \ar[rd]^-{\tau'} & \\
\bar{\bP} \ar[rd]_-{\tau} & \tilde{\bP} \ar[d]^-\psi \ar[r]_-{\pi} & \bP' \\
& \bP.&
}\]
\item 
We have 
\begin{eqnarray*}
\tilde{\tau}^*\tilde{D}_i&=&e_i\tilde{\bar{D}}_i \quad (0\leq i\leq r), \\
\tilde{\tau}^*\tilde{D}_j&=&e_j\tilde{\bar{D}}_j \quad (r+1\leq j\leq s),\\
\tilde{\tau}^*\bE&=&\frac{hh'}{\bar{h}\bar{h}'}\bar{\bE}.
\end{eqnarray*}
In particular, we have 
\[
\tilde{\tau}^*\left[\cO_{\tilde{\bP}}(\alpha,\beta)\right]
\sim_\bQ\left[\cO_{\tilde{\bar{\bP}}}\left(
\frac{h}{\bar{h}}\alpha, \frac{h'}{\bar{h}'}\beta\right)\right].
\]
\item
For any $0\leq i\leq r$, we have 
\[
(\tau')^*D'_i=\frac{e_ig_i}{\bar{g}_i}\bar{D}'_i. 
\]
In particular, we have 
\[
(\tau')^*\left[\cO_{\bP'}(1)\right]\sim_\bQ
\left[\cO_{\bar{\bP}'}\left(\frac{gh}{\bar{g}\bar{h}}\right)\right].
\]
\item 
The morphisms $\tilde{\tau}$ and $\tau'$ are given by the 
following: 
\begin{eqnarray*}
\tilde{\tau}\colon
\left[\bar{x}_0:\cdots:\bar{x}_r;\bar{y}_{r+1}
:\cdots:\bar{y}_s:\bar{z}\right]
&\mapsto&
\left[\bar{x}_0^{e_0}:\cdots:\bar{x}_r^{e_r};
\bar{y}_{r+1}^{e_{r+1}}
:\cdots:\bar{y}_s^{e_s}:\bar{z}^{\frac{hh'}{\bar{h}\bar{h}'}}\right], \\
\tau'\colon
\left[\bar{x'_0}:\cdots:\bar{x'_r}\right]
&\mapsto&
\left[{\bar{x'_0}}^{\frac{e_0g_0}{\bar{g}_0}}:\cdots:{\bar{x'_r}}^{\frac{e_rg_r}{\bar{g}_r}}\right].
\end{eqnarray*}
\end{enumerate}
\end{lemma}

\begin{proof}
Let $\pi_{\bar{N},N}\colon \bar{N}\twoheadrightarrow N$ be as in 
Proposition \ref{proposition_finite-wps}. Then we can directly 
check that 
\[
\pi_{\bar{N},N}\left(\bar{v}_{s+1}\right)
=\frac{hh'}{\bar{h}\bar{h}'}v_{s+1}. 
\]
Thus we get the morphism $\tilde{\tau}$. Together with 
Proposition \ref{proposition_finite-wps}, we get the assertion (2). 
We remark that the induced weighted bundle structure $\bar{\pi}$ is 
the unique nontrivial fibration structure of $\tilde{\bar{\bP}}$. 
Thus $\bar{\pi}$ can be obtained by the Stein factorization of 
the morphism $\pi\circ\tilde{\tau}$. Thus we get the desired 
toric finite morphism $\tau'$ satisfying the assertion (1). 
The assertion (3) follows immediately from (2) and 
Lemma \ref{lemma_cl-cl}. 
The remaining assertion (4) is trivial from (2) and (3). 
\end{proof}

\begin{corollary}\label{corollary_intersection} 
Let $\psi \colon \tilde{\bP} \to \bP$ be the standard weighted blowup as in Definition \ref{definition_swb} and 
$\pi \colon \tilde{\bP} \to \bP'$ be the induced weighted bundle structure as in Definition \ref{definition_wbs}. 
\begin{enumerate}
\item 
For any $0\leq k\leq s$, we have
\[
\left(\cO_{\tilde{\bP}}(1,0)^{\cdot k}\cdot
\cO_{\tilde{\bP}}(0,1)^{\cdot s-k}\right)=\begin{cases}
\frac{h^k(h')^{s-k}}{a_0\cdots a_s} & \text{if }0\leq k\leq r, \\
0 & \text{if }r+1\leq k\leq s.
\end{cases}
\]
\item 
The two restriction morphisms 
\begin{eqnarray*}
\pi|_{\bE}\colon\bE&\to&\bP\left(a'_0,\dots,a'_r\right), \\
\psi|_{\bE}\colon\bE&\to&\bP\left(a'_{r+1},\dots,a'_s\right)
\end{eqnarray*}
give the isomorphism 
\[
\pi|_{\bE}\times\psi|_{\bE}\colon \bE\simeq 
\bP\left(a'_0,\dots,a'_r\right)\times
\bP\left(a'_{r+1},\dots,a'_s\right).
\]
Moreover, under the isomorphism, we have 
\begin{eqnarray*}
\left[\cO_{\tilde{\bP}}(\alpha,\beta)\right]|_{\bE}&\sim_\bQ&
\left[\cO_{\bP\left(a'_0,\dots,a'_r\right)\times
\bP\left(a'_{r+1},\dots,a'_s\right)}
\left(\frac{\alpha}{g},\frac{\beta}{g'}\right)\right], \\
\bE|_{\bE}&\sim_\bQ&
\left[\cO_{\bP\left(a'_0,\dots,a'_r\right)\times
\bP\left(a'_{r+1},\dots,a'_s\right)}
\left(-\frac{h'}{g},\frac{h}{g'}\right)\right].
\end{eqnarray*}
\end{enumerate}
\end{corollary}

\begin{proof}
(1) Let us apply Lemma \ref{lemma_finite-swb} for $\bar{a}_i=1$, 
$e_i=a_i$ for all $0\leq i\leq s$. Then $\bar{h}=\bar{h}'=1$. Moreover, 
we have $\bar{\bP}=\bP^s$ and $\tilde{\bar{\bP}}$ is just the 
ordinary blowup of $\bP^s$ along an $(s-r-1)$-dimensional 
linear subspace (see also Example \ref{example_typical}). In particular, we have 
\[
\tilde{\bar{\bP}}\simeq\bP_{\bP^r}
\left(\cO_{\bP^r}^{\oplus s-r}\oplus\cO_{\bP^r}(1)\right)
\]
and $\bar{\pi}$ is the $\bP^{s-r}$-bundle. Moreover, 
$\cO_{\tilde{\bar{\bP}}}(1,0)\simeq\pi^*\cO_{\bP^r}(1)$ and 
$\cO_{\tilde{\bar{\bP}}}(1,0)$ is isomorphic to the tautological 
line bundle with respect to the above bundle sturucture. 
This immediately implies that 
\[
\left(\cO_{\tilde{\bar{\bP}}}(1,0)^{\cdot k}\cdot
\cO_{\tilde{\bar{\bP}}}(0,1)^{\cdot s-k}\right)=\begin{cases}
1 & \text{if }0\leq k\leq r, \\
0 & \text{if }r+1\leq k\leq s.
\end{cases}
\]
Note that $\deg\tilde{\tau}=\deg\tau=\prod_{i=0}^s e_i=\prod_{i=0}^s a_i$. 
Thus the assertion (1) follows from Lemma~\ref{lemma_finite-swb}. 

(2) Set $\gamma:=\pi|_{\bE}\times\psi|_{\bE}$. Then $\gamma$ is a finite 
morphism. By Proposition \ref{proposition_sublinear}, we have 
\[
\left[\cO_\bP(1)\right]|_{\bP\left(a'_{r+1},\dots,a'_s\right)}
\sim_\bQ\left[\cO_{\bP\left(a'_{r+1},\dots,a'_s\right)}
\left(\frac{1}{g'h'}\right)\right].
\]
Thus we get 
\begin{eqnarray*}
&&\left(\pi^*\cO_{\bP'}(1)^{\cdot r}
\cdot\psi^*\cO_{\bP}\left(g'h'\right)^{\cdot s-r-1}\cdot\bE\right)\\
&=&\deg\gamma\cdot
\left(\cO_{\bP'}(1)^{\cdot r}\right)\cdot
\left(\cO_{\bP}\left(g'h'\right)^{\cdot s-r-1}\right)
=\frac{\deg\gamma}{a'_0\cdots a'_s}. 
\end{eqnarray*}
On the other hand, by (1) and Lemma \ref{lemma_cl-cl}, we have 
\begin{eqnarray*}
&&\left(\pi^*\cO_{\bP'}(1)^{\cdot r}
\cdot\psi^*\cO_{\bP}\left(g'h'\right)^{\cdot s-r-1}\cdot\bE\right)\\
&=&
\left(\cO_{\tilde{\bP}}(g,0)^{\cdot r}\cdot
\cO_{\tilde{\bP}}\left(0,g'\right)^{\cdot s-r-1}
\cdot\cO_{\tilde{\bP}}\left(-h',h\right)\right)
=\frac{g^rh^{r+1}(g')^{s-r-1}(h')^{s-r}}{a_0\cdots a_s}. 
\end{eqnarray*}
By combining the above two equalities, we get $\deg\gamma=1$. 
The remaining assertions are trivial.
\end{proof}

\begin{remark}\label{remark_section}
Assume that $r=s-1$. Then $\bE\simeq\bP'$ is a section of $\pi$. 
For any nonzero $f\in H^0\left(\bP',\cO_{\bP'}(d)\right)$, we can 
consider $(\pi^*f=0)\in\left|\cO_{\tilde{\bP}}(gd,0)\right|$ as in 
Proposition \ref{proposition_pi-star}, i.e., $\pi^*f
=f\left(x_0^{g_0},\dots,x_r^{g_r}\right)$. We can directly check that 
the restriction $(\pi^*f=0)|_{\bE}$, which is a $\bQ$-Cartier 
$\bQ$-divisor on $\bE$, is a Weil divisor on $\bE$ defined by 
$(f=0)\subset\bP'$ under the isomorphism $\pi|_{\bE}$. 
Moreover, if $(a_0,\dots,a_r)$ is well-formed (i.e., if 
$g_0=\cdots=g_r=1$), then, for any $\tilde{f}\in H^0\left(\tilde{\bP}, 
\cO_{\tilde{\bP}}(\alpha,\beta)\right)$, the restriction 
$\left(\tilde{f}=0\right)|_\bE$ is also a Weil divisor on $\bE$ 
defined by $\left(\tilde{f}\left(x_0,\dots,x_r,1,0\right)=0\right)
\subset\bE$ under the isomorphism $\bE\simeq\bP'$. 
\end{remark}

\begin{lemma}\label{lemma_aut-swb}
Assume that $g_0=1$. If $\varphi'\in H^0\left(\bP',
\cO_{\bP'}(a'_0)\right)$ is general, then there are 
isomorphisms $\tilde{\iota}\colon\tilde{\bP}\to\tilde{\bP}$ 
and $\iota'\colon\bP'\to\bP'$ satisfying 
$\iota'\circ\pi=\pi\circ\tilde{\iota}$, 
$(\iota')^*D'_0=\left(\varphi'=0\right)$ and 
$\tilde{\iota}^*\tilde{D}_0=\left(\pi^*\varphi'=0\right)$, 
where $D_0'= (x_0'=0) \subset \bP'$ and  $\tilde{D}_0= (x_0=0) \subset \tilde{\bP}$ are as in Setup \ref{setup:swb} and Definition \ref{definition_swb}, and   
\[
\pi^*\varphi'\in H^0\left(\tilde{\bP},\cO_{\tilde{\bP}}
(a''_0,0)\right)
\]
is as in Proposition \ref{proposition_pi-star}. 
\end{lemma}

\begin{proof}
The proof is almost same as the proof of Lemma 
\ref{lemma_aut-wps}. We may assume that $\varphi'=x'_0+
\varphi'_0(x'_1,\dots,x'_r)$ with $\varphi'_0\in 
H^0\left(\bP',\cO_{\bP'}(a'_0)\right)$. Note that 
\[
\pi^*\varphi'=x_0+\varphi'_0(x_1^{g_1},\dots,x_r^{g_r}).
\]
We set 
\[
\iota'\colon\left[x'_0:x'_1:\cdots:x'_r\right]
\mapsto
\left[x'_0+\varphi'_0\left(x'_1,\dots,x'_r\right):x'_1
:\cdots:x'_r\right],
\]
\begin{eqnarray*}
\tilde{\iota}\colon
\left[x_0:x_1:\cdots:x_r;y_{r+1}:\cdots:y_s:z\right]\\
\mapsto
\left[x_0+\varphi_0'\left(x_1^{g_1},\dots,x_r^{g_r}\right)
:x_1:\cdots:x_r;y_{r+1}:\cdots:y_s:z\right].
\end{eqnarray*}
Then, by Lemma \ref{lemma_finite-swb}, 
the morphisms satisfy the properties in 
Lemma \ref{lemma_aut-swb}. 
\end{proof}

The standard weighted blow-up $\psi \colon \tilde{\bP} \to \bP$ and the induced weighted bundle structure $\pi \colon \tilde{\bP} \to \bP'$ induce such morphisms for $D_0$ and $D_s$ as follows. 

\begin{proposition}\label{proposition_divisor-swb}
Assume that $s\geq 2$. 
\begin{enumerate}
\item 
For any $1\leq i\leq s$, set 
\[
\dot{g}_i:=\gcd\left(a_1,\dots,\hat{a}_i,\dots,a_s\right), \quad
\dot{g}:=\prod_{i=1}^s\dot{g}_i, \quad
\dot{a}_i:=\frac{a_i\dot{g}_i}{\dot{g}}.
\]
We know that $D_0\simeq\bP\left(\dot{a}_1,\dots,\dot{a}_s\right)$
by Proposition \ref{proposition_divisor_wps}. 
If $r=1$, then $\psi|_{\tilde{D}_0}\colon \tilde{D}_0\to D_0$ is 
an isomorphism. Assume that $r\geq 2$. Then 
\[
\xymatrix{
\tilde{D}_0 \ar[d]_-{\psi|_{\tilde{D}_0}} \ar[r]^-{\pi|_{\tilde{D}_0}} 
& D'_0 \\
 D_0&
}\]
is the standard weighted blowup of 
$D_0=\bP\left(\dot{a}_1,\dots,\dot{a}_s\right)$ along 
$\bP\left(\dot{a}_{r+1},\dots,\dot{a}_s\right)$, together with the 
weighted bundle structure. We have 
\[
\bE\cdot\tilde{D}_0=\frac{1}{g_0}\bE_{D_0},
\]
where $\bE_{D_0}\subset\tilde{D}_0$ is the exceptional divisor 
of $\psi|_{\tilde{D}_0}$, and 
\[
\left[\cO_{\tilde{\bP}}(\alpha,\beta)\right]|_{\tilde{D}_0}\sim_\bQ
\left[\cO_{\tilde{D}_0}\left(
\frac{\gcd\left(\dot{a}_{r+1},\dots,\dot{a}_s\right)}{g_0h'}\alpha,
\frac{\gcd\left(\dot{a}_1,\dots,\dot{a}_r\right)}{g_0h}\beta\right)\right].
\]
\item 
For any $0\leq i\leq s-1$, set 
\[
\ddot{g}_i:=\gcd\left(a_0,\dots,\hat{a}_i,\dots,a_{s-1}\right), \quad
\ddot{g}:=\prod_{i=0}^{s-1}\ddot{g}_i, \quad
\ddot{a}_i:=\frac{a_i\ddot{g}_i}{\ddot{g}}.
\]
We know that $D_s\simeq\bP\left(\ddot{a}_0,\dots,\ddot{a}_{s-1}\right)$
by Proposition \ref{proposition_divisor_wps}. 
If $r=s-1$, then $\psi|_{\tilde{D}_s}\colon \tilde{D}_s\to D_s$ is 
an isomorphism and $\tilde{D}_s$ is a section of $\pi$. 
Assume that $r\leq s-2$. Then 
\[
\xymatrix{
\tilde{D}_s \ar[d]_-{\psi|_{\tilde{D}_s}} \ar[r]^-{\pi|_{\tilde{D}_s}} 
& \bP' \\
 D_s&
}\]
is the standard weighted blowup of 
$D_s=\bP\left(\ddot{a}_0,\dots,\ddot{a}_{s-1}\right)$ along 
$\bP\left(\ddot{a}_{r+1},\dots,\ddot{a}_{s-1}\right)$, together with the 
weighted bundle structure. We have 
\[
\bE\cdot\tilde{D}_s=\frac{1}{g_s}\bE_{D_s},
\]
where $\bE_{D_s}\subset\tilde{D}_s$ is the exceptional divisor 
of $\psi|_{\tilde{D}_s}$, and 
\[
\left[\cO_{\tilde{\bP}}(\alpha,\beta)\right]|_{\tilde{D}_s}\sim_\bQ
\left[\cO_{\tilde{D}_s}\left(
\frac{\gcd\left(\ddot{a}_{r+1},\dots,\ddot{a}_{s-1}\right)}{g_sh'}\alpha,
\frac{\gcd\left(\ddot{a}_0,\dots,\ddot{a}_r\right)}{g_sh}\beta\right)\right].
\]
\end{enumerate}
\end{proposition}

\begin{proof}
We only prove (1) since the proof of (2) is same as the proof of (1). 
We may assume that $r\geq 2$. The assertion 
$\bE\cdot\tilde{D}_0=\frac{1}{g_0}\bE_{D_0}$ follows from Lemma 
\ref{lemma_ray-ray}. Set 
\[\dot{N}'':=N/\bZ u_0\simeq\bZ^s/\bZ\sum_{i=1}^s a_i\mathbf{e}_i.
\]
Let $\dot{u}''_i\in\dot{N}''$ (resp., $\dot{v}_j''\in\dot{N}''$) be the images 
of $u_i\in N$ (resp., $v_j\in N$) for all $1\leq i\leq r$ (resp., 
$r+1\leq j\leq s$). By Proposition \ref{proposition_well-formed}, 
there is an isomorphism $\dot{N}\simeq\dot{N}''$ of lattices 
and there are primitive elements 
\[
\dot{u}_1,\dots,\dot{u}_r,\dot{v}_{r+1},\dots,\dot{v}_s\in\dot{N}
\]
satisfying 
\[
\dot{g}_i\dot{u}_i\mapsto\dot{u}''_i,\quad 
\dot{g}_j\dot{v}_j\mapsto\dot{v}''_j.
\]
The image $\dot{v}_{s+1}\in\dot{N}$ of $v_{s+1}\in N$ under the 
surjection $N\twoheadrightarrow\dot{N}''\simeq\dot{N}$ can be expressed as 
\[
\dot{v}_{s+1}=\frac{1}{h'}\sum_{i=1}^r a''_i\dot{u}_i
=-\frac{1}{h}\sum_{j=r+1}^s a''_j\dot{v}_j.
\]
The element $\dot{v}_{s+1}\in\dot{N}$ is a nonzero element 
satisfying 
\[
\dot{v}_{s+1}\in\sum_{i=1}^r\bQ_{\geq 0}\dot{u}_i\quad\text{and}\quad
\dot{v}_{s+1}\in\left(\sum_{i=1}^r\bQ\dot{u}_i\right)
\cap\left(\sum_{j=r+1}^s\bQ\dot{v}_j\right). 
\]
Moreover, the toric variety $\tilde{D}_0$ is defined by the fan in 
$\dot{N}_\bR$ whose set of maximal cones consists of 
\[
\operatorname{Cone}\left(\dot{u}_1,\dots,\hat{\dot{u}}_i,\dots,
\dot{u}_r,\dot{v}_{r+1},\dots,\hat{\dot{v}}_j,\dots,\dot{v}_{s+1}\right)
\]
with $1\leq i\leq r$ and $r+1\leq j\leq s+1$ by 
\cite[\S 3.2]{CLS}. Thus, as we pointed out in Lemma \ref{lemma_ray}, 
the morphism $\psi|_{\tilde{D}_0}\colon \tilde{D}_0\to D_0$
is nothing but the standard weighted blowup. 
Moreover, the morphism $\pi|_{\tilde{D}_0}\colon 
\tilde{D}_0\twoheadrightarrow D'_0$ is a fibration with connected fibers. 
Thus the morphism $\pi|_{\tilde{D}_0}$ is the induced weighted 
bundle structure. 
Let us set 
\[
\tilde{h}:=\gcd\left(\dot{a}_1,\dots,\dot{a}_r\right), \quad
\tilde{h}':=\gcd\left(\dot{a}_{r+1},\dots,\dot{a}_s\right). 
\]
We can take $d_0,d'_0\in\bQ_{>0}$ such that 
\[
\left[\cO_{\tilde{\bP}}(\alpha,\beta)\right]|_{\tilde{D}_0}\sim_\bQ
\left[\cO_{\tilde{D}_0}\left(
d_0\alpha,d'_0\beta\right)\right].
\]
holds for any $\alpha,\beta\in\bQ$. 
Since $\bE|_{\tilde{D}_0}=\frac{1}{g_0}\bE_{D_0}$, we get 
\[
\left[\cO_{\tilde{\bP}}\left(-h',h\right)\right]|_{\tilde{D}_0}\sim_\bQ
\left[\cO_{\tilde{D}_0}\left(-\frac{\tilde{h}'}{g_0},
\frac{\tilde{h}}{g_0}\right)\right].
\]
Thus we get 
\[
d_0=\frac{\tilde{h}'}{g_0h'},\quad d'_0=\frac{\tilde{h}}{g_0h}
\]
from the above equalities.
\end{proof}

\begin{corollary}\label{corollary_cut-swb}
Under the notation in Proposition \ref{proposition_divisor-swb}, 
assume moreover 
$r\geq 2$, $(a_1,\dots,a_s)$ is well-formed, and $g_0=1$. 
(For example, if $r\geq 2$ and 
$(a_1,\dots,a_r)$ is well-formed, then the above 
three conditions are satisfied.) Then we can easily check that 
$\dot{a}_i=a_i$ for any $1\leq i\leq s$ and 
\[
\frac{\tilde{h}'}{g_0h'}=\frac{\tilde{h}}{g_0h}=1. 
\]
Moreover, for any nonzero homogeneous polynomial 
\[
f\left(x_0,\dots,x_r,y_{r+1},\dots,y_s,z\right)
\in\bC\left[x_0,\dots,x_r,y_{r+1},\dots,y_s,z\right]
\]
of degree $(d_1,d_2)$ in the Cox ring of $\tilde{\bP}$, the restriction 
$(f=0)|_{\tilde{D}_0}$ of $(f=0)\subset\tilde{\bP}$ is a Weil divisor 
on $\tilde{D}_0$ in $\left|\cO_{\tilde{D}_0}(d_1,d_2)\right|$ 
defined by the homogeneous polynomial 
\[
f|_{\tilde{D}_0}:=
f\left(0,x_1\dots,x_r,y_{r+1},\dots,y_s,z\right)
\in\bC\left[x_1,\dots,x_r,y_{r+1},\dots,y_s,z\right]
\]
of degree $(d_1,d_2)$
in the Cox ring of $\tilde{D}_0$. 
\end{corollary}

\begin{proof}
By Proposition \ref{proposition_divisor-swb} (1) and Proposition 
\ref{proposition_divisor_wps}, the $\bQ$-Cartier $\bQ$-divisor 
$(f=0)|_{\tilde{D}_0}$ is Cartier along all codimension one points of 
$\tilde{D}_0$. Thus we get the assertion as in the proof of 
Proposition \ref{proposition_divisor_wps}. 
\end{proof}

\begin{example}\label{example_swb-pt}
Assume that $s=r+1$, $(a_0,\dots,a_r)$ is well-formed, 
and set 
\[
P:=\left[0:\cdots:0:1\right]\in\bP.
\]
Then the singularity of $P\in\bP$ is isomorphic to the quotient 
singularity \[
0\in\bA^{r+1}_{\mathbf{x}_0,\dots,\mathbf{x}_r}/\bZ_{a_{r+1}}\left(
a_0,\dots,a_r\right)\] in the sense of \cite[Definition 2.2.10]{Kawakita}. 
By Lemma \ref{lemma_ray}, the morphism $\psi\colon\tilde{\bP}\to\bP$ 
is nothing but the weighted blowup of $P\in\bP$ with weights 
$\operatorname{wt}\left(\mathbf{x}_0,\dots,\mathbf{x}_r\right)
=\frac{1}{a_{r+1}}\left(a_0,\dots,a_r\right)$ in the sense of 
\cite[Definition 2.2.11]{Kawakita}. 
\end{example}

\begin{example}\label{example_typical}
Assume that there exists $a\in\bZ_{>0}$ such that 
\[
a_0=\cdots=a_r=1 \quad\text{and}\quad a_{r+1}=\cdots=a_s=a. 
\]
Then we have 
\begin{eqnarray*}
&&h=1, \quad a''_i=a_i=1, \quad g_i=g=1\quad (0\leq i\leq r), \\
&&h'=a, \quad a''_j=a'_j=1, \quad g_j=g'=1 \quad (r+1\leq j\leq s).
\end{eqnarray*}
In this case, the lattice 
\[
N=\bZ^{s+1}/\bZ\left(\sum_{i=0}^r\mathbf{e}_i
+a\sum_{j=r+1}^s\mathbf{e}_j\right)
\]
satisfies that 
\[
N=\bigoplus_{i=1}^r\bZ u_i\oplus\bigoplus_{j=r+2}^{s+1}\bZ v_j
\]
with 
\[
v_{r+1}=-\sum_{j=r+2}^{s+1}v_j,\quad 
u_0=-\sum_{i=1}^r u_i + a v_{s+1}. 
\]
By \cite[Example 7.3.5]{CLS}, we can directly check that 
\[
\tilde{\bP}\simeq\bP_{\bP^r}\left(\cO_{\bP^r}^{\oplus s-r}\oplus
\cO_{\bP^r}(a)\right). 
\]
Moreover, the morphism $\pi$ is the usual projective space bundle. 
Let $\xi\in\operatorname{Pic}\left(\tilde{\bP}\right)$ be the 
tautological line bundle with respect to the above bundle structure. 
Then we have 
\[
\cO_{\tilde{\bP}}(\alpha,\beta)\simeq\pi^*\cO_{\bP^r}(\alpha)\otimes
\xi^{\otimes\beta}
\]
for any $\alpha,\beta\in\bZ$. 
\end{example}

\subsection{On quasi-smoothness}\label{section_q-smooth}

We first discuss some properties of quasi-smooth weighted hypersurfaces, and then study whether the quasi-smoothness is preserved under weighted blow-ups.  

\begin{definition}\label{definition_q-sm}
Let $N=\bZ^s$ be a lattice, let $\Theta$ be a simplicial and complete 
fan in $N_\bR$, let $\{u_1,\dots,u_k\}\subset N$ be the set of 
primitive generators of the $1$-dimensional cones of $\Theta$. 
Let $\bP:=X_\Theta$ be the proper and $\bQ$-factorial toric variety 
associated with $\Theta$, let $\bC[z_1,\dots,z_k]$ be the Cox ring 
of $\bP$, where $z_i$ corresponds to $u_i$, and let 
\[
Z(\Theta)\subset\bA^k=\bA^k_{z_1,\dots,z_k}
\]
be the irrelevant locus of $\Theta$ in the sense of \cite[\S 5.1]{CLS}. 
As in \cite[Theorem 5.1.11]{CLS}, there is the canonical quotient 
morphism 
\[
\bA^k\setminus Z\left(\Theta\right)\to\bP. 
\]
Let us set $\alpha_i:=\cO_\bP\left(V(u_i)\right)
\in\operatorname{Cl}\left(\bP\right)$ 
for any $1\leq i\leq k$. 
Take any nonzero $f\in\bC[z_1,\dots,z_k]$ which is homogeneous 
of degree $\beta\in\operatorname{Cl}(\bP)$, and let us set 
$X:=(f=0)\in\left|\cO_\bP(\beta)\right|$. 
\begin{enumerate}
\item 
The \emph{non-quasi-smooth locus} $\operatorname{NQS}(X)\subset X$ of 
the divisor $X\subset\bP$ is defined to be the image by 
$\bA^k\setminus Z\left(\Theta\right)\to\bP$ of 
\[
\left\{z\in\bA^k\setminus Z\left(\Theta\right)\,\,|\,\,
(f=0)\subset\bA^k\text{ is singular at }z\right\}. 
\]
We say that 
the divisor $X\subset \bP$ is said to be \emph{quasi-smooth} if 
$\operatorname{NQS}(X)=\emptyset$. 
\item (cf.\ \cite[Definition 2.1]{ST24})
The divisor $X\subset\bP$ is said to be \emph{well-formed} if the subset 
$\operatorname{Sing}\bP\cap X\subset X$ has codimension $\geq 2$. 
\end{enumerate}
\end{definition}

\begin{proposition}\label{proposition_diff}
If $X\subset\bP$ is quasi-smooth, then the pair $(\bP, X)$ is a plt pair 
and 
\[
\left(K_\bP+X\right)|_X=K_X+\sum_{\substack{\tau\in\Theta(2);\\ 
V(\tau)\subset X}}\left(1-\frac{1}{\operatorname{mult}(\tau)}\right)
V(\tau), 
\]
where $\Theta(2)$ is the set of $2$-dimensional cones in $\Theta$. 
In particular, if moreover $X\subset\bP$ is well-formed, then we have 
$K_X=\left(K_\bP+X\right)|_X$. 
\end{proposition}

\begin{proof}
Take any $s$-dimensional cone 
\[
\sigma=\bR_{\geq 0}u_1+\cdots+\bR_{\geq 0}u_s\in\Theta
\]
and any $2$-dimensional subcone 
$\tau=\bR_{\geq 0}u_1+\bR_{\geq 0}u_2\in\Theta$ 
with $V(\tau)\subset X$. Consider the affine open subset 
$\bA_\sigma\subset\bP$ of $\bP$ defined by $\sigma$. By 
\cite[Proposition 3.5]{BC94}, the algebraic group
\[
\mathbf{D}(\sigma):=\bC\left[\operatorname{Cl}(\sigma)\right]\subset\left(\bC^*\right)^s
\]
in \cite[Theorem 1.9]{BC94} is finite and small, and the quotient 
\[
q\colon \bA^s_{x_1\dots x_s}\to \bA^s/\mathbf{D}(\sigma)
\]
is isomorphic to $\bA_\sigma$. Moreover, the pullback of $D_i:=V(u_i)
\subset\bA_\sigma$ is equal to the hyperplane $(x_i=0)$ for any 
$1\leq i\leq s$, and the divisor $q^*\left(X|_{\bA_\sigma}\right)
\subset\bA^s$ is smooth. Since 
\[
K_{\bA^s}+q^*\left(X|_{\bA_\sigma}\right)=q^*\left(K_{\bA_\sigma}+
X|_{\bA_\sigma}\right), 
\]
the pair $(\bP,X)$ is plt. Moreover, by the ramification formula, 
the coefficient of the different 
of $X\subset\bP$ along $V(\tau)$ is equal to $1-\frac{1}{d}$, where 
$d\in\bZ_{>0}$ is defined to be 
\[
q^*V(\tau)=d\cdot(x_1=x_2=0)_{\operatorname{red}}. 
\]
On the other hand, as in the proof of Proposition 
\ref{proposition_divisor_wps}, we have 
\[
\left(K_\bP+D_1\right)|_{D_1}=K_{D_1}+\sum_{\substack{u\in\Theta(1);\\
\tau_1:=\bR_{\geq 0}u+\bR_{\geq 0}u_1\in\Theta(2)}}\left(
1-\frac{1}{\operatorname{mult}(\tau_1)}\right)V(\tau_1)
\]
where $D_1=V(u_1)$. Then the above $d$ must be equal to $\mult(\tau)$ since $D_1\subset\bP$ is also quasi-smooth. 
\end{proof}

For quasi-smooth hypersurfaces 
$X\in\left|\cO_{\bP}(\beta)\right|$, the well-formedness of 
$X$ can be interpreted in a combinatorial way, as in 
\cite[6.10]{Fletcher00}. 

\begin{lemma}\label{lemma_cl-toric}
Under the assumption in Definition \ref{definition_q-sm}, 
we have the following: 
\begin{enumerate}
\item 
For any $1\leq i\leq k$, we have 
\[
\operatorname{Cl}\left(\bP\right)=\sum_{j\in\left\{
1,\dots,k\right\}\setminus\{i\}}\bZ\alpha_j.
\]
\item 
Take any $2\leq t\leq s$. 
Assume that $\tau=\bR_{\geq 0}u_{i_1}+\cdots+\bR_{\geq 0}u_{i_t}
\in\Theta(t)$ for some $1\le t \le k$, where $\Theta(t)$ is the set 
of $t$-dimensional cones in $\Theta$. 
Then the quotient group 
\[
\mathcal{C}:=\operatorname{Cl}\left(\bP\right)
/\sum_{j\in\left\{1,\dots,k\right\}\setminus\left\{
i_1,\dots,i_t\right\}}\bZ\alpha_j
\]
is a finite Abelian group, and is of order 
$\operatorname{mult}(\tau)$. 
\end{enumerate}
\end{lemma}

\begin{proof}
(1) Since $u_i\in N$ is primitive, there exists $m\in M$ 
such that $\langle m,u_i\rangle=1$ holds, where $M:=N^\vee$. 
Thus we get 
\[
\alpha_i=-\sum_{j\in\left\{1,\dots,k\right\}\setminus\{i\}}
\langle m, u_j\rangle \alpha_j. 
\]

(2) 
We may assume that 
$\left\{i_1,\dots,i_t\right\}=\left\{1,\dots,t\right\}$. 
Let $M_\tau$ be the dual of 
\[
N_\tau:=N\cap \left(\bR u_1+\cdots+\bR u_t\right).
\]
Consider the homomorphism 
\begin{eqnarray*}
\phi\colon M &\to& \bZ^t \\
m &\mapsto& \left(
\langle m,u_1\rangle,\dots,\langle m,u_t\rangle\right). 
\end{eqnarray*}
Since $\phi$ naturally factors through the projection 
$M\twoheadrightarrow M_\tau$, the cokernel $\mathcal{C}'$ 
of $\phi$ is a finite group of order 
$\operatorname{mult}(\tau)$. 
Let us consider the diagram 
\[
\xymatrix{
0 \ar[r] & 0 \ar[r] \ar[d] & 
M \ar[r]^-{\operatorname{id}} \ar[d]^-\iota & M \ar[r] 
\ar[d]^-\phi& 0\\
0 \ar[r] & \ar[r] \bZ^{k-t} \ar[r] & \bZ^k \ar[r]_p & \bZ^t \ar[r] & 0,
}
\]
where 
$\iota(m):=\left(\langle m,u_i\rangle\right)_{1\leq i\leq k}$ 
and $p$ is the canonical projection to first $t$-th 
components. By the snake lemma, we get the following 
exact sequence: 
\[
\bZ^{k-t}\to \operatorname{Cl}\left(\bP\right)
\to \mathcal{C}'\to 0.
\]
The image of $\bZ^{k-t}\to \operatorname{Cl}\left(\bP\right)$
is nothing but $\sum_{j=t+1}^k\bZ\alpha_j$. 
Thus we have an isomorphism $\mathcal{C}\simeq\mathcal{C}'$. 
\end{proof}

\begin{proposition}\label{proposition_dimca}
Under the assumption in Definition \ref{definition_q-sm}, 
assume that $X\in\left|\cO_\bP(\beta)\right|$ is a 
quasi-smooth hypersurface in $\bP$. Then the following are 
equivalent: 
\begin{enumerate}
\item 
The divisor $X\subset\bP$ is well-formed. 
\item 
For any $\tau=\bR_{\geq 0}u_{i_1}
+\bR_{\geq 0}u_{i_2}\in\Theta(2)$, we have 
\[
\beta\in\sum_{j\in\left\{1,\dots,k\right\}\setminus\{
i_1,i_2\}}\bZ\alpha_j
\]
in $\operatorname{Cl}\left(\bP\right)$. 
\end{enumerate}
\end{proposition}

\begin{proof}
Assume that 
$\tau=\bR_{\geq 0}u_1+\bR_{\geq 0}u_2\in\Theta(2)$ 
and satisfies that $\operatorname{mult}(\tau)>1$. 
By Lemma \ref{lemma_cl-toric} (2), we have 
\[
\sum_{i=3}^k\bZ\alpha_i\subsetneq 
\operatorname{Cl}\left(\bP\right). 
\]
We can uniquely write 
\[
f=f_0\left(z_3,\dots,z_k\right)
+z_1g_1\left(z_3,\dots,z_k\right)
+z_2g_2\left(z_3,\dots,z_k\right)
+\sum_{i+j\geq 2}z_1^iz_2^jg_{i,j}\left(z_3,\dots,z_k\right).
\]
The non-well formedness condition $V(\tau)\subset X$ 
is equivalent to $f_0=0$. 

If $\beta\not\in\sum_{i=3}^k\bZ\alpha_i$, then we must have 
$f_0=0$. 
Conversely, assume that $f_0=0$. We can take a point 
$p=(0,0,p_3,\dots,p_k)\in\bA^k\setminus 
Z\left(\Theta\right)$ such that $p\in (f=0)$. 
Since 
\[
\frac{\partial f}{\partial z_i}(p)=0
\]
for any $3\leq i\leq k$, we have 
$g_1\left(p_3,\dots,p_k\right)\neq 0$ or 
$g_2\left(p_3,\dots,p_k\right)\neq 0$ since $X\subset\bP$ is 
quasi-smooth. This implies that 
\[
\beta-\alpha_1\text{ or }
\beta-\alpha_2\in\sum_{i=3}^k\bZ\alpha_i.
\]
We may assume that $\beta-
\alpha_2\in\sum_{i=3}^k\bZ\alpha_i$. 
If $\beta\in\sum_{i=3}^k\bZ\alpha_i$, then we have 
$\alpha_2\in\sum_{i=3}^k\bZ\alpha_i$. Thus we have 
$\sum_{i=2}^k\bZ\alpha_i\subsetneq\operatorname{Cl}
\left(\bP\right)$, a contradiction to 
Lemma \ref{lemma_cl-toric} (1). Thus we must have 
$\beta\not\in\sum_{i=3}^k\bZ\alpha_i$, and then 
we immediately get the desired equivalence. 
\end{proof}

\begin{example}\label{example_irrelevant}
\begin{enumerate}
\item 
Assume that $(a_0,\dots,a_s)$ is well-formed, and let $\Sigma$ in $N_\bR$ 
and let $\bP:=\bP(a_0,\dots,a_s)$ as in Definitions \ref{definition_wps}
and \ref{definition_q-sm}. 
Then $Z(\Sigma)=\{0\}\subset\bA^{s+1}_{x_0\cdots x_s}$
holds. By \cite[Proposition 2.3]{PST17}, if $s\geq 4$ and if $X\subset\bP$ 
is a quasi-smooth and well-formed hypersurface, then the class group 
of $X$ is freely generated by $\cO_\bP(1)|_X$. 
The sheaf is often denoted by $\cO_X(1)$ in this paper. 
Moreover, if we further assume that $X$ is Fano (i.e., $-K_X$ is ample), 
then the \emph{index} of $X$ 
is defined to be the positive integer $\iota$ satisfying 
$\cO_X(-K_X)\simeq\cO_X(\iota)$. By Proposition \ref{proposition_diff}, 
the index of $X$ 
is equal to $\sum_{i=0}^sa_i-d$, where $d$ is the degree of $X$. 
\item 
The irrelevant locus $Z\left(\tilde{\Sigma}\right)$ as in Definition \ref{definition_q-sm} of the fan $\tilde{\Sigma}$ as in Definition \ref{definition_swb} can be described as  
\[
Z\left(\tilde{\Sigma}\right)
=\left(x_0=\cdots=x_r=0\right)\cup\left(y_{r+1}=\cdots=y_s=z=0\right)
\subset\bA^{s+2}_{x_0\cdots x_r y_{r+1}\cdots y_s z}.
\]
\end{enumerate}
\end{example}

\begin{remark}\label{remark_cut-qs}
\begin{enumerate}
\item 
Under the assumption of Proposition \ref{proposition_divisor_wps}, assume 
that $(a_1,\dots,a_s)$ is well-formed, and take $X:=(f=0)\subset\bP$
as in Proposition \ref{proposition_divisor_wps} (2). 
Take a general $\varphi\in H^0\left(\bP,\cO_\bP(a_0)\right)$. 
By Lemma \ref{lemma_aut-wps}, there exists $\iota\in\operatorname{Aut}\bP$ 
such that $\iota^*D_0=(\varphi=0)$. It is trivial that the divisor 
\[
X|_{(\varphi=0)}\subset (\varphi=0)\simeq \bP(a_1,\dots,a_s)
\]
is quasi-smooth if and only if for any point $p\in (f=\varphi=0)\subset
\bA^{s+1}\setminus\{0\}$ the rank of the Jacobian matrix 
$J(f,\varphi)(p)$ is $2$, where 
\[
J(f,\varphi)(p):=\begin{pmatrix}
\frac{\partial f}{\partial x_0}(p) & \cdots & 
\frac{\partial f}{\partial x_s}(p)\\
\frac{\partial \varphi}{\partial x_0}(p) & \cdots & 
\frac{\partial \varphi}{\partial x_s}(p)
\end{pmatrix}.
\]
(We sometimes write $J\left(X, \varphi\right)(p)$ or $J\left(X, 
(\varphi=0)\right)(p)$ in place of $J(f,\varphi)(p)$ if we are only 
interested in its rank.)
\item 
Under the assumption in Corollary \ref{corollary_cut-swb}, take a 
nonzero homogeneous element 
$f\in\bC\left[x_0,\dots,x_r,y_{r+1},\dots,y_s,z\right]$ and set 
$X:=(f=0)\subset\tilde{\bP}$. Take a general $\varphi'\in 
H^0\left(\bP',\cO_{\bP'}(a'_0)\right)$. By Lemma \ref{lemma_aut-swb}, 
there exists $\tilde{\iota}\in \operatorname{Aut}\tilde{\bP}$ such that 
$\tilde{\iota}^*\left(\pi^*\varphi'=0\right)=\tilde{D}_0$ holds. 
It is also trivial that the Weil divisor 
\[
X|_{\tilde{\iota}^*\left(\pi^*\varphi'=0\right)}\subset
\tilde{\iota}^*\left(\pi^*\varphi'=0\right)\simeq\tilde{D}_0
\]
is quasi-smooth if and only if for any point 
\[
p\in \left(f=\pi^*\varphi'=0\right)\subset\bA^{s+1}\setminus 
Z\left(\tilde{\Sigma}\right), 
\]
the rank of the Jacobian matrix $J\left(f,\pi^*\varphi'\right)(p)$
is equal to $2$. 
\end{enumerate}
\end{remark}

\begin{proposition}\label{proposition_strict-transform}
Under the assumption in Setup \ref{setup:swb}, take 
$f\in H^0\left(\bP,\cO_{\bP}(d)\right)\setminus\{0\}$ with 
\[
f=\sum_{\substack{I=(m_0,\dots,m_r)\in\bZ^{r+1}_{\geq 0}\\
J=(m_{r+1},\dots,m_s)\in\bZ_{\geq 0}^{r-s}}}c_{I,J}x^Ix^J, 
\]
where $c_{I,J}\in\bC$, $x^I:=x_0^{m_0}\cdots x_r^{m_r}$ and 
$x^J:=x_{r+1}^{m_{r+1}}\cdots x_s^{m_s}$. Set $X:=(f=0)\subset\bP$. 
We also set 
\[
\deg'' I:=\sum_{i=0}^r a''_im_i,\quad
\deg'' J:=\sum_{j=r+1}^s a''_jm_j.
\]
Consider the values 
\[
d_0:=\min\left\{\deg''I\,\,|\,\,c_{I,J}\neq 0\right\}, \quad
d'_0:=\max\left\{\deg''J\,\,|\,\,c_{I,J}\neq 0\right\}. 
\]
Obviously, we have $hd_0+h'd'_0=d$. 
Then, the strict transform $\tilde{X}:=\psi^{-1}_*X\subset\tilde{\bP}$ 
is defined by 
\[
\tilde{f}=\sum_{I,J}c_{I,J}x^Iy^Jz^{\frac{\deg'' I-d_0}{h'}}
\in H^0\left(\tilde{\bP}, \cO_{\tilde{\bP}}(d_0,d'_0)\right).
\]
\end{proposition}

\begin{proof}
Since $\gcd(h,h')=1$, $\deg'' I-d_0$ is always divisible by $h'$. 
Moreover, we have 
\[
\frac{\deg'' I-d_0}{h'}=\frac{d'_0-\deg'' J}{h}. 
\]
Thus $\tilde{f}$ is homogeneous. We can directly see 
from Lemma \ref{lemma_cl-cl} that 
$\tilde{f}$ is not divisible by $z$ and 
$\psi_*\left(\tilde{f}=0\right)=X$.
Thus we get the assertion. 
\end{proof}

We remark that $\tilde{X}$ may not be quasi-smooth in general 
even when $X$ is quasi-smooth.

\begin{example}\label{example_non-qsm}
\begin{enumerate}
\item 
Assume that $s=3$, $r=2$, $(a_0,a_1,a_2,a_3)=(3,1,1,1)$. 
Let us consider the weighted homogeneous polynomial  
\[
f:=x_3^2x_1^2+x_3x_0+x_1^4+x_2^4
\]
of degree $4$. We can directly check that the hypersurface $X\subset
\bP$ defined by $f$ is quasi-smooth. However, since its strict 
transform $\tilde{X}\subset\tilde{\bP}$ is defined by 
\[
\tilde{f}=y^2x_1^2+zyx_0+z^2(x_1^4+x_2^4)
\]
of degree $(2,1)$ (where we set $y:=y_3$),
the $\tilde{X}\subset\tilde{\bP}$ is not quasi-smooth at 
the point $\left[x_0:x_1:x_2;y:z\right]=\left[0:0:1; 1:0\right]$. 
\item 
Assume that $s=4$, $r=2$, 
$(a_0,a_1,a_2,a_3,a_4)=(2,3,4,4,5)$. 
Let us consider the weighted homogeneous polynomial 
\[
f=x_0x_4^2+x_1x_3x_4+x_2x_3^2+x_0^6+x_1^4+x_2^3
\]
of degree $12$. We can directly check that the hypersurface 
$X\subset\bP$ defined by $f$ is quasi-smooth. Moreover, 
the $X$ contains the center $(x_0=x_1=x_2=0)\subset\bP$ of 
the standard weighted blowup. 
However, since its strict transform 
$\tilde{X}\subset\tilde{\bP}$ is defined by 
\[
\tilde{f}=x_0y_4^2+x_1y_3y_4z+x_2y_3^2z^2+
\left(x_0^6+x_1^4+x_2^3\right)z^{10}
\]
of degree $(2,10)$, the $\tilde{X}\subset\tilde{\bP}$ 
is not quasi-smooth at any point 
$\left[x_0:x_1:x_2;1:0:0\right]$ with 
$(x_0,x_1,x_2)\in\bC^3\setminus\{0\}$. 
\end{enumerate}
\end{example}

However, when $s=r+1$ and $a_0\leq\cdots\leq a_s$, then 
we have the following affirmative result: 

\begin{proposition}\label{proposition_qsm-wbl}
Assume that $s=r+1$ (this implies that $h=1$ and $h'=a_s$) 
and $a_0\leq\cdots\leq a_s$. Then, if $X\subset\bP$ is a 
quasi-smooth hypersurface, then its strict transform 
$\tilde{X}\subset\tilde{\bP}$ is also quasi-smooth. 
If we further assume that $(a_0,\dots,a_r)$ is well-formed, then 
$\tilde{X}|_\bE\subset\bE\simeq\bP'$ is a quasi-smooth hypersurface. 
\end{proposition}

\begin{proof}
Let $f$ be the defining equation of $X$ and set $d:=\deg f$. 
We can uniquely write 
\[
f=x_{r+1}^k f_{d-h'k}(x_0,\dots,x_r)+x_{r+1}^{k-1}f_{d-h'(k-1)}
(x_0,\dots,x_r)+\cdots+f_d(x_0,\dots,x_r),
\]
where $f_{d-h'j}(x_0,\dots,x_r)$ is a weighted homogeneous polynomial 
of degree $d-h'j$ for any $0\leq j\leq k$ with $f_{d-h'k}\neq 0$. 
Then $\tilde{X}\subset\tilde{\bP}$ is defined by the equation 
\[
\tilde{f}=y^kf_{d-h'k}(x_0,\dots,x_r)+zy^{k-1}f_{d-h'(k-1)}
(x_0,\dots,x_r)+\cdots+z^kf_d(x_0,\dots,x_r), 
\]
where we set $y:=y_{r+1}$. 

Assume that $d=h'k$. Then we may assume that $f_{d-h'k}=1$. 
Since $[0:\cdots:0:1]\not\in X$, it is trivial that $\tilde{X}$ is 
quasi-smooth since so is $X$. Note that $\tilde{X}|_\bE$ is the empty set in this case. 

Therefore, we can assume $d>h'k$ in the 
remaining case. In this case, $[0:\cdots:0:1]\in X$. 
By the quasi-smoothness of $X$ at the point, there exists 
$0\leq j\leq k$ such that $f_{d-h'j}$ is quasi-linear (in the sense 
of \cite[Definition 2.17]{KOW23}). 
On the other hand, since $a_0\leq\cdots\leq a_r\leq h'$, the $j$ 
must be equal to $k$. 
(Indeed, there exists $0\leq i\leq r$ such that $d-h'j=a_i$. 
If $j<k$, then 
\[
0<d-h'k\leq d-h'j-h'=a_i-h'\leq 0,
\]
a contradiction.) Take any point 
$
\tilde{p}=[x_0:\cdots:x_r; y:z]\in\tilde{X} 
$. 
If $z\neq 0$, then it is trivial that $\tilde{X}$ is quasi-smooth 
at $\tilde{p}$ follows from the quasi-smoothness of $X$ at 
the point $[x_0:\cdots:x_r:y]\in X$. Thus we may assume that 
$z=0$. Since $f_{d-h'k}$ is quasi-linear, there exists 
$0\leq i\leq r$ such that $d-h'k=a_i$ and the coefficient of $x_i$ in the polynomial 
$f_{d-h'k}$ is nonzero. Since 
\[
\frac{\partial\tilde{f}}{\partial x_i}\left(\tilde{p}\right)
=y^k\frac{\partial f_{d-h'k}}{\partial x_i}\left(\tilde{p}\right)\neq 
0, 
\]
the divisor $\tilde{X}$ is quasi-smooth at $\tilde{p}$. 
If we further assume that $(a_0,\dots,a_r)$ is well-formed, then 
$\tilde{X}|_\bE\subset\bE\simeq\bP'=\bP(a_0,\dots,a_r)$ is 
defined by $f_{d-h'k}(x_0,\dots,x_r)$ which is quasi-linear 
by the quasi-smoothness of $X$. Thus the assertion follows 
(cf.\ Lemma \ref{lemma_aut-wps}). 
\end{proof}

\section{Cutting process}\label{section_cutting}
We consider hyperplane sections of quasi-smooth hypersurfaces. 
To begin with, let us recall the following, which is a special case of the 
\emph{Abban--Zhuang's method} that we have already seen in 
Theorem \ref{theorem:AZ}. 

\begin{lemma}[{cf.\ \cite[Theorem 4.12]{ST24}}]\label{lem:cutting}
Let $X$ be an $n$-dimensional projective variety, 
let $\eta\in X$ be a scheme-theoretic point, let $\Delta$ be an effective 
$\bQ$-Weil divisor on $X$, let $L$ be a big $\Q$-line bundle on $X$, let 
$a\in\bQ_{>0}$, and let $Y\in|a L|$ be an irreducible and 
reduced divisor such that 
$\eta\in Y$ and the pair $\left(X,\Delta+Y\right)$ is plt at $\eta\in X$. 
Take an effective $\bQ$-Weil divisor $\Delta_Y$ on $Y$ such that 
\[
\left(K_X+\Delta+Y\right)|_Y=K_Y+\Delta_Y
\]
holds around $\eta$. 
Then we have 
\[
\delta_\eta\left(X,\Delta; L\right)\geq\min\left\{
a(n+1),\quad
\frac{n+1}{n}\delta_\eta\left(Y,\Delta_Y;L|_Y\right)\right\}. 
\]
\end{lemma}

\begin{proof}
The result is well-known to specialists. We give a proof for convenience. 
We may assume that $a=1$. Let $V_{\vec{\bullet}}$ be the Veronese 
equivalence class $H^0\left(\bullet L\right)$ of the complete linear 
series on $X$ associated to $L$ as in Definition \ref{defn:VeroneseEquiv},  
and let 
$W_{\vec{\bullet}}:=V_{\vec{\bullet}}^{(Y)}$ be the refinement of 
$V_{\vec{\bullet}}$ by $Y$ as in Definition \ref{definition_plt-flag}. 
By \cite[Theorem B and Corollary 3.6]{BFJ}, the $\bQ$-line bundle $L|_Y$ 
on $Y$ is big. Moreover, we have 
\[
\operatorname{vol}_X(L)=\operatorname{vol}_Y\left(L|_Y\right)
=\limsup_{k\to\infty}\frac{\dim\operatorname{Image}\left(
H^0\left(X,kL\right)\to H^0\left(Y, kL|_Y\right)\right)}{k^{n-1}/(n-1)!}. 
\]
In particular, by \cite[Lemma 3.3]{Fujita:2024aa}, 
the series $W_{\vec{\bullet}}$ is asymptotically equivalent to 
$\tilde{W}_{\vec{\bullet}}$ (cf. \cite[Definition 3.2]{Fujita:2024aa}), where $\tilde{W}_{\vec{\bullet}}$ is a $\bZ_{\ge 0}^2$-graded linear series defined by 
\[
\tilde{W}_{k,j}=\begin{cases}
H^0\left(Y, (k-j)L|_Y\right) & \text{if }0\leq j\leq k, \\
0 & \text{otherwise}. 
\end{cases}
\]
Take any prime divisor $F$ over $Y$ such that its center contains $\eta$. 
By \cite[Lemma 4.73]{Xubook} and \cite[Proposition 6.2]{Fujita:2024aa}, we have 
\begin{eqnarray*}
S\left(W_{\vec{\bullet}}; F\right)
&=&S\left(\tilde{W}_{\vec{\bullet}}; F\right)\\
&=&\frac{n}{\operatorname{vol}_Y(L|_Y)}\int_0^1\int_0^\infty
\operatorname{vol}_Y\left(x(L|_Y)-yF\right)dydx\\
&=&\frac{n}{\operatorname{vol}_Y(L|_Y)}\int_0^1x^n\int_0^\infty
\operatorname{vol}_Y\left(L|_Y-zF\right)dzdx\\
&=&\frac{n}{n+1}S\left(L|_Y; F\right). 
\end{eqnarray*}
This implies that 
\[
\delta_\eta\left(Y,\Delta_Y;W_{\vec{\bullet}}\right)
=\frac{n+1}{n}\delta_\eta\left(Y,\Delta_Y; L|_Y\right). 
\]
Since $A_{X,\Delta}(Y)=1$ and $S\left(L; Y\right)=\frac{1}{n+1}$, 
the assertion follows from \cite[Theorem 3.2]{AZ22}
(cf.\ Theorem \ref{theorem:AZ}; see also 
\cite[Theorem 12.3]{Fujita:2024aa}). 
\end{proof}

The following is a Bertini-type result, which is important in order to apply 
Abban--Zhuang's method (Lemma \ref{lem:cutting}) for quasi-smooth hypersurfaces. 

\begin{lemma}\label{lem:hyperplane_mod}
Let $n,r\in\bZ_{>0}$ with $r-1\leq n$. Take positive integers 
$1\leq a_0\leq\cdots\leq a_{n+1}$ with $a_r=1$, and set 
$\bP:=\bP\left(a_0,\dots,a_{n+1}\right)$. Take any quasi-smooth 
hypersurface $X\subset\bP$ and a point $p\in \bP$. Let $B_1^{\bP}$ and $B_{1,p}^{\bP}$ be as in Definition \ref{definition_Ba}. 
\begin{enumerate}
\item 
Assume that $\dim \left(B_1^\bP\cap X\right)\leq r-1$.
(For example, if $n\leq 2r-1$, then $\dim \left(B_1^\bP\cap X\right)
\leq r-1$.) 
Take a general $H\in 
\left|\cO_\bP(1)\right|$. 
Then, for any $q\in X\cap H$, the rank of the Jacobian matrix $J(X,H)(q)$ (cf. Remark \ref{remark_cut-qs}(1)) is equal to $2$. 
In particular, if $r\geq 2$, then $X|_H\subset 
H\simeq\bP(a_1,\dots,a_{n+1})$ is also quasi-smooth. 
\item 
Assume that $\dim\left(B_{1,p}^\bP\cap X\right)\leq r-2$.
(For example, if $n\leq 2r-3$, then we have $\dim\left( B_{1,p}^\bP\cap X\right)
\leq r-2$.) 
Take a general $H\in 
\left|\cO_\bP(1)_p\right|$. 
Then, for any $q\in X\cap H$, the rank of $J(X,H)(q)$ is equal to $2$. 
In particular, if $r\geq 2$, then $X|_H\subset 
H\simeq\bP(a_1,\dots,a_{n+1})$ is also quasi-smooth. 
\end{enumerate}
\end{lemma}

\begin{proof}
Let $f$ be the defining polynomial of $X\subset\bP$. 
Let $r+1\leq c_1\leq n+2$ with 
\[
1=a_0=\cdots=a_{c_1-1}<a_{c_1}\leq\cdots\leq a_{n+1}. 
\]

\noindent(1) We note that $\dim B_1^\bP=n+1-c_1$. 
(Thus, if $n\leq 2r-1$, then we have $\dim B_1^{\bP}\leq r-1$.)

Take any closed subset $\tilde{B}_{1,X}\subset\bA^{n+2}\setminus\{0\}$ 
with $\dim\tilde{B}_{1,X}=\dim \left(B_1^\bP\cap X\right)$ such that 
$\tilde{B}_{1,X}$ maps onto $B_1^\bP\cap X$ via the quotient morphism 
$\bA^{n+2}\setminus\{0\}\to\bP$. 
Since $X$ is quasi-smooth, the following map is a regular morphism: 
\begin{eqnarray*}
\varphi\colon \tilde{B}_{1,X}&\to&\bP^{n+1}_{z_0\cdots z_{n+1}}\\
\tilde{q}&\mapsto&\left[
\frac{\partial f}{\partial x_0}\left(\tilde{q}\right):\cdots:
\frac{\partial f}{\partial x_{n+1}}\left(\tilde{q}\right)\right].
\end{eqnarray*}
Let us set 
\[
P:=\left\{z_{c_1}=\cdots=z_{n+1}=0\right\}\subset
\bP^{n+1}_{z_0\cdots z_{n+1}}. 
\]
Since $\dim\tilde{B}_{1,X}\leq r-1<\dim P=c_1-1$, a general 
$g=\sum_{i=0}^{c_1-1}e_i x_i$ satisfies that 
\[
\left[e_0:\cdots:e_{c_1-1}:0:\cdots:0\right]\in 
P\setminus\varphi\left(\tilde{B}_{1,X}\right). 
\]
This implies that the rank of $J(f,g)(q)$ is $2$ for any 
$q\in B_1^\bP\cap X$. 
Thus the assertion (1) follows by Bertini's theorem. 

\noindent(2) We may assume that $p\not\in B_1^\bP$. 
By coordinate changes, we may 
assume that 
\[
p=\left[0:\cdots:0:1:p_{c_1}:\cdots:p_{n+1}\right]\in\bP. 
\]
We note that $\dim B_{1,p}^\bP=n+2-c_1$. (Thus, if $n\leq 2r-3$, 
then we have $\dim B_{1,p}^\bP\leq r-2$.)

Take any closed subset $\tilde{B}_{1,p,X}\subset\bA^{n+2}\setminus\{0\}$ 
with $\dim\tilde{B}_{1,p,X}=\dim \left(B_{1,p}^\bP\cap X\right)$ such that 
$\tilde{B}_{1,p,X}$ maps onto $B_{1,p}^\bP\cap X$ via the morphism 
$\bA^{n+2}\setminus\{0\}\to\bP$. As in (1), we can consider the morphism 
\begin{eqnarray*}
\varphi\colon\tilde{B}_{1,p,X}&\to&\bP^{n+1}_{z_0\cdots z_{n+1}}\\
\tilde{q}&\mapsto&\left[
\frac{\partial f}{\partial x_0}\left(\tilde{q}\right):\cdots:
\frac{\partial f}{\partial x_{n+1}}\left(\tilde{q}\right)\right].
\end{eqnarray*}
Let us set 
\[
P_p:=\left\{z_{c_1-1}=\cdots=z_{n+1}=0\right\}\subset
\bP^{n+1}_{z_0\cdots z_{n+1}}. 
\]
Since $\dim\tilde{B}_{1,p,X}\leq r-2<\dim P_p=c_1-2$, a general 
$g=\sum_{i=0}^{c_1-2}e_ix_i$ satisfies that 
\[
\left[e_0:\cdots:e_{c_1-2}:0:\cdots:0\right]\in 
P_p\setminus\varphi\left(\tilde{B}_{1,p,X}\right). 
\]
This implies that the rank of $J(f,g)(q)$ is $2$ for any 
$q\in B_{1,p}^\bP\cap X$. 
Thus, assertion (2) also follows by Bertini's theorem. 
\end{proof}

In the proof, we showed the quasi-smoothness of $H|_X$ along $B_1^\bP\cap X$
(or, along $B_{1,p}^\bP\cap X$). In fact, we have the following
example.

\begin{example}\cite[Example (4.2.1)]{Kollar95}\label{example_Kollar95}
Let $X$ be a smooth variety, $V$ a closed subset of $X$ and $L$ a linear system such that there exists an element $D \in L$ that is smooth along $V$.  It may happen that a general member $D' \in L$ is not smooth along $V$. 

Let $X = \bC^n$ and $L=|B|:= \{ (\lambda x_1 + \mu x_1 x_2 + \nu f =0) \subset \bC^n \mid \lambda, \mu, \nu \in \bC \}$ be the linear system on $\bC^n$ with coordinates $(x_1, \ldots , x_n)$, where 
$f \in \bC[x_3, \ldots, x_n]$ is a polynomial defining an isolated singularity $(f=0) \subset \bC^{n-2}$ at the origin. 
Then $D=(x_1 =0) \in |B|$ is smooth, but a general member $D'=(\lambda x_1 + \mu x_1 x_2 + \nu f =0)$ is singular at $(0, - \frac{\lambda}{\mu}, 0, \ldots , 0)$. 

\end{example}

\begin{proposition}\label{proposition_B1}
Let $n,d\geq 2$ and let $\bP:=\bP\left(a_0,\dots,a_{n+1}\right)$ be 
well-formed. Take a quasi-smooth hypersurface $X\subset\bP$ 
of degree $d$. Assume that $B_1^\bP\subset X$. 

Then, for any $0\leq k\leq n+1$ 
with $a_k\geq 2$, a general element $H\in\left|\cO_\bP(a_k)\right|$ 
satisfies that $X|_H\subset H$ is also quasi-smooth. 
\end{proposition}

\begin{proof}
Let $f$ (resp., $g$) be the defining equation of $X$ (resp., $H$). 
We may assume that 
$1=a_0=\cdots=a_{c_1-1}<a_{c_1}\leq \cdots\leq a_{n+1}$. 
Since $X$ is quasi-smooth, $d\geq 2$ and $B_1^\bP\subset X$, we have 
$c_1\geq 2$. We may assume that $g$ can be written as 
$g=x_k+g_0(x_0,\dots,x_{k-1})$. 
Take any $p\in B_1^\bP\cap (x_k=0)$. Since $B_1^\bP\subset X$ and $X$ is 
quasi-smooth, we have 
$\frac{\partial f}{\partial x_j}(p)=0$ for any $c_1\leq j\leq n+1$ and 
$\frac{\partial f}{\partial x_i}(p)\neq 0$ for some $0\leq i\leq c_1-1$. 
On the other hand, we have $\frac{\partial g}{\partial x_k}(p)\neq 0$. 
Thus the rank of $J(f,g)(p)$ is equal to $2$ for any 
$p\in B_1^\bP\cap (x_k=0)$. 
Since $\operatorname{Bs}\left|\cO_\bP(a_k)\right|\subset B_1^\bP
\cap (x_k=0)$, 
the restriction $X|_H\subset H$ is quasi-smooth by Bertini's theorem. 
\end{proof}

Now we consider hypersurfaces in standard weighted blowups. 
We fix the notation from now on. 

\begin{setup}\label{setup:cut-swb}
Under Setup \ref{setup:swb}, 
Definitions \ref{definition_swb}, \ref{definition_wbs} and 
\ref{definition_cl},  
we assume moreover that $n:=r=s-1$, and 
$\left(a_0,\dots,a_r\right)$ is well-formed with 
$
1=a_0=\cdots=a_{c_1-1}<a_{c_1}\leq\cdots\leq a_n, 
$
where $0\leq c_1\leq n$. 
(We do not have any assumption on $a_s=a_{n+1}$.)
Hence we have a diagram 
\[
\xymatrix{
\tilde{\bP} \ar[r]^-{\pi} \ar[d]_-{\psi} & \bP' = \bP(a_0, \ldots ,a_{s-1}) \\
\bP = \bP(a_0, \ldots , a_s) & 
}, 
\]
where $\psi$ is the standard weighted blowup and $\pi$ is the induced weighted bundle structure.  
Moreover, let us write $x_i:=x'_i$ for the coordinates of $\bP'$ 
($0\leq i\leq r$) and 
$y:=y_{n+1}$ for simplicity. (This makes no confusion since 
$g_0=\cdots=g_r=1$ by the well-formedness of $(a_0, \ldots , a_r)$.)
We fix a quasi-smooth hypersurface $\tilde{X}\subset\tilde{\bP}$ 
of degree $(c,k)$ ($c,k\in\bZ_{\geq 0}$) defined by 
$\left(\tilde{F}(x_0,\dots,x_n,y,z)=0\right)$. 
\end{setup}

\begin{lemma}\label{lemma:hyperplane-wbs}
\begin{enumerate}
\item 
Assume that 
\[
\dim\left(\pi^{-1}\left(B_1^{\bP'}\right)\cap\tilde{X}\right)\leq c_1-2. 
\]
(For example, if $n+3\leq 2c_1$, then $\dim
\pi^{-1}\left(B_1^{\bP'}\right)\leq c_1-2$ holds.)
Then, for any general $H'\in\left|\cO_{\bP'}(1)\right|$, we have 
$\operatorname{rank}J\left(\tilde{X}, \pi^*H'\right)\left(\tilde{q}\right)
=2$ for any $\tilde{q}\in\tilde{X}\cap \pi^*H'$. 
\item 
Take any point $p'\in\bP'$. Assume that 
\[
\dim\left(\pi^{-1}\left(B_{1,p'}^{\bP'}\right)\cap\tilde{X}\right)\leq c_1-3. 
\]
(For example, if $n+5\leq 2c_1$, then $\dim
\pi^{-1}\left(B_{1,p'}^{\bP'}\right)\leq c_1-3$ holds.)
Then, for any general $H'\in\left|\cO_{\bP'}(1)_{p'}\right|$, we have 
$\operatorname{rank}J\left(\tilde{X}, \pi^*H'\right)\left(\tilde{q}\right)
=2$ for any $\tilde{q}\in\tilde{X}\cap \pi^*H'$. 
\end{enumerate}
\end{lemma}

\begin{proof}
The proof is almost same as the proof of Lemma \ref{lem:hyperplane_mod}. 
We only see the proof of (2), since the proof of (1) is very similar to 
(2). 

We may assume that $p'\not\in B_1^{\bP'}$ by (1). After coordinate 
changes, we may assume that $p'=\left[0:\cdots:0:1:p'_{c_1}:\cdots:
p'_n\right]\in\bP'$. In particular, we have $\dim
\pi^{-1}\left(B_{1,p'}^{\bP'}\right)=n+2-c_1$. 
Take any closed subset 
\[
\tilde{B}_{1,p',\tilde{X}}\subset\bA^{n+3}\setminus 
Z\left(\tilde{\Sigma}\right)
\]
with \[
\dim \tilde{B}_{1,p',\tilde{X}}=
\dim\left(\pi^{-1}\left(B_{1,p'}^{\bP'}\right)\cap\tilde{X}\right)
\]
which maps onto $\pi^{-1}\left(B_{1,p'}^{\bP'}\right)\cap\tilde{X}$ 
under the quotient morphism $\bA^{n+3}\setminus Z\left(\tilde{\Sigma}
\right)\to\tilde{\bP}$. Then we can consider the morphism 
\begin{eqnarray*}
\varphi\colon \tilde{B}_{1,p',\tilde{X}}&\to&
\bP^{n+2}_{z_0\cdots z_{n+2}}\\
\tilde{q} &\mapsto&\left[\frac{\partial \tilde{f}}{\partial x_0}
\left(\tilde{q}\right):\cdots:
\frac{\partial \tilde{f}}{\partial x_n}
\left(\tilde{q}\right):\frac{\partial \tilde{f}}{\partial y}
\left(\tilde{q}\right):
\frac{\partial \tilde{f}}{\partial z}
\left(\tilde{q}\right)\right].
\end{eqnarray*}
Let us set the linear subspace
\[
P_{p'}:=\left(z_{c_1-1}=\cdots=z_{n+2}=0\right)\subset\bP^{n+2}
\]
of dimension $c_1-2$. 
For any general $g=\sum_{i=0}^{c_1-2}e_ix_i\in H^0\left(
\bP',\cO_{\bP'}(1)_{p'}\right)$, by the assumption, we have 
\[
\left[e_0:\cdots:e_{c_1-2}:0:\cdots:0\right]\in P_{p'}\setminus
\varphi\left(\tilde{B}_{1,p',\tilde{X}}\right).
\]
This implies that the rank of $J\left(\tilde{f}, 
\pi^*g\right)\left(\tilde{q}\right)$ is equal to $2$ for any 
$\tilde{q}\in\pi^{-1}\left(B_{1,p'}^{\bP'}\right)\cap\tilde{X}$. 
Thus the assertion follows from Bertini's theorem. 
\end{proof}

\begin{lemma}\label{lem:NQSwellformed_swb}
Fix $p'\in\bP'$ and take any $c_1\leq i\leq n$. Assume that 
\[
\dim\left(\pi^{-1}\left(B_{a_i,p'}^{\bP'}\right)\cap\tilde{X}\right)
\leq i-2. 
\]
(For example, if $a_i<a_{i+1}$ (or $i=n$) and $n+3\leq 2i$, then 
$\dim\pi^{-1}\left(B_{a_i,p'}^{\bP'}\right)\leq i-2$ holds.)
Then, for any general $H'\in\left|\cO_{\bP'}(a_i)_{p'}\right|$, we have 
$\operatorname{rank}J\left(\tilde{X}, \pi^*H'\right)\left(\tilde{q}\right)
=2$ for any $\tilde{q}\in\left(\pi^*H'\cap\tilde{X}\right)\setminus
\pi^{-1}\left(B_1^{\bP'}\right)$. 
\end{lemma}

\begin{proof}
We may assume that $c_1\geq 1$. 
If $p'\in B_1^{\bP'}$, then, since $\operatorname{Bs}\left|
\cO_{\bP'}(a_i)_{p'}\right|\subset B_1^{\bP'}$, the assertion follows 
from Bertini's theorem. Thus we may assume that $p'\not\in B_1^{\bP'}$. 
By coordinate changes, we may assume that 
$a_i<a_{i+1}$ (or $i=n$) and 
$p'=\left[1:0:\cdots:0\right]\in\bP'$. Then we have 
\[
B_1^{\bP'}=\left(x_0=\cdots=x_{c_1-1}=0\right)\subset\bP',\quad
B_{a_i,p'}^{\bP'}=\left(x_1=\cdots=x_i=0\right)\subset\bP'.
\]
In particular, we have $\dim\pi^{-1}\left(B_{a_i,p'}^{\bP'}\right)
=n+1-i$. 
Let us take a closed subset 
\[
\tilde{V}\subset\left(\bA^{n+3}_{x_0\cdots x_n y z}\cap(x_0=1)\right)
\setminus Z\left(\tilde{\Sigma}\right)
\]
satisfying 
\[
\dim \tilde{V}=\dim \left(\left(\pi^{-1}\left(B_{a_i,p'}^{\bP'}\right)
\cap\tilde{X}\right)\setminus\pi^{-1}\left(B_1^{\bP'}\right)\right)
\]
and $\tilde{V}$ maps onto $\left(\pi^{-1}\left(B_{a_i,p'}^{\bP'}\right)
\cap\tilde{X}\right)\setminus\pi^{-1}\left(B_1^{\bP'}\right)$
via the quotient morphism $\bA^{n+3}\setminus Z\left(\tilde{\Sigma}\right)
\to\tilde{\bP}$. We can consider the morphism 
\begin{eqnarray*}
\varphi\colon\tilde{V}&\to&\bP^{n+2}_{z_0\cdots z_{n+2}}\\
\tilde{q} &\mapsto&\left[\frac{\partial \tilde{f}}{\partial x_0}
\left(\tilde{q}\right):\cdots:
\frac{\partial \tilde{f}}{\partial x_n}
\left(\tilde{q}\right):\frac{\partial \tilde{f}}{\partial y}
\left(\tilde{q}\right):
\frac{\partial \tilde{f}}{\partial z}
\left(\tilde{q}\right)\right].
\end{eqnarray*}
Consider the linear subspace 
\[
P_{p'}:=\left(z_0=z_{i+1}=\cdots=z_{n+2}=0\right)\subset\bP^{n+2}
\]
of dimension $i-1$. Since $\dim\tilde{V}<\dim P_{p'}$, the 
intersection $\varphi\left(\tilde{V}\right)\cap P_{p'}$ is not 
dense in $P_{p'}$. 

Take a general $g\in H^0\left(\bP',\cO_{\bP'}(a_i)_{p'}\right)$, 
and let $e_j$ be the coefficient of $x_0^{a_i-a_j}x_j$ for any 
$1\leq j\leq i$. Then, for any $\tilde{q}\in\tilde{V}$, we have 
\begin{eqnarray*}
&&\frac{\partial g}{\partial x_0}\left(\tilde{q}\right)
=\frac{\partial g}{\partial x_{i+1}}\left(\tilde{q}\right)
=\cdots=\frac{\partial g}{\partial x_n}\left(\tilde{q}\right)=
\frac{\partial g}{\partial y}\left(\tilde{q}\right)=\frac{\partial g}{\partial z}\left(\tilde{q}\right)=0, \\
&&\frac{\partial g}{\partial x_j}\left(\tilde{q}\right)=e_j 
\quad\left(1\leq j\leq i\right). 
\end{eqnarray*}
Since $g$ is general, we have 
\[
\left[0:e_1:\cdots:e_i:0:\cdots:0\right]\in 
P_{p'}\setminus\left(\varphi\left(\tilde{V}\right)\cap P_{p'}\right). 
\]
Thus, for any $\tilde{q}\in
\left(\pi^{-1}\left(B_{a_i,p'}^{\bP'}\right)\cap\tilde{X}\right)
\setminus\pi^{-1}\left(B_1^{\bP'}\right)$, we have 
$\operatorname{rank}J\left(\tilde{f},\pi^*g\right)\left(\tilde{q}
\right)=2$. Therefore, the assertion follows 
from Lemma \ref{lemma_Ba-taro} and Bertini's theorem. 
\end{proof}

\begin{corollary}\label{corollary_2wt-cutting}
Under Setup \ref{setup:cut-swb}, assume moreover $n\geq 3$ and 
$a_0=\cdots=a_{n-1}=1$. Set $p':=\left[1:0:\cdots:0\right]\in\bP'$. 
For any $2\leq j\leq n$, let us set 
\begin{eqnarray*}
&&\tilde{\bP}_n:=\tilde{\bP}, \quad\bE_n:=\bE, \\
&&\tilde{\bP}_{j-1}:=\tilde{\bP}_j\cap\pi^*\left(x_j=0\right), \quad
\bE_{j-1}:=\bE_j|_{\tilde{\bP}_{j-1}}\quad(3\leq j\leq n). 
\end{eqnarray*}
By Proposition \ref{proposition_divisor-swb} and Example 
\ref{example_typical}, for any $2\leq j\leq n-1$, we have 
\[
\tilde{\bP}_j=\bP_{\bP^{j-1}_{x_0\cdots x_{j-1}}}\left(
\cO_{\bP^{j-1}}\oplus\cO_{\bP^{j-1}}(a_{n+1})\right)
\]
and $\bE_j\subset\tilde{\bP}_j$ is the negative section. 
Then, there exists $\varphi\in\operatorname{Aut}\left(\bP'\right)$ 
with $\varphi(p')=p'$ such that, after replacing $\tilde{X}\subset\bP$ 
with $(\pi^*\varphi)\left(\tilde{X}\right)\subset\tilde{\bP}$, if we set 
\begin{eqnarray*}
&&\tilde{X}_n:=\tilde{X}\subset\tilde{\bP}_n, \\
&&\tilde{X}_{j-1}:=\tilde{X}_j|_{\tilde{\bP}_{j-1}}\subset\tilde{\bP}_{j-1}
\quad(3\leq j\leq n), 
\end{eqnarray*}
then $\tilde{X}_j\subset\tilde{\bP}_j$ is a quasi-smooth hypersurface 
of degree $(c,k)$ for any $2\leq j\leq n-1$. 
Moreover, if we further assume that $\tilde{X}|_\bE\subset\bE\simeq
\bP(1^n,a_n)$ is a quasi-smooth hypersurface of degree $c$, then 
we can also take the above $\varphi$ such that satisfying moreover 
$\tilde{X}_j|_{\bE_j}\subset\bE_j\simeq\bP^j$ is smooth of degree $c$ 
for any $2\leq j\leq n-1$. 
\end{corollary}

\begin{proof}
By Lemma \ref{lem:NQSwellformed_swb}, any general 
$H'\in\left|\cO_{\bP'}(a_n)_{p'}\right|$ satisfies that 
$\tilde{X}|_{\pi^*H'}\subset\pi^*H'$ is quasi-smooth, since 
$\operatorname{Bs}\left|\cO_{\bP'}(a_{n})_{p'}\right|\cap 
B_1^{\bP'}=\emptyset$. Moreover, if $\tilde{X}|_\bE\subset\bE$ is 
quasi-smooth, then $\tilde{X}|_{\pi^*H'\cap \bE}\subset H'\simeq\bP^{n-1}$ 
is also smooth by Lemma \ref{lemma:hyperplane-wbs}. By considering the 
same procedures inductively, we get the assertion. 
\end{proof}

\section{Covering process}\label{section_covering}

\begin{proposition}[{cf.\ \cite[Proposition 
3.4]{ST24}}]\label{proposition_covering}
Let $(X,\Delta)$ and $(X',\Delta')$ be pairs of normal projective 
varieties with effective $\bQ$-Weil divisors. Assume that there exists 
a finite morphism $\tau\colon X'\to X$ such that 
\[
K_{X'}+\Delta'=\tau^*\left(K_X+\Delta\right). 
\]
Let $\eta\in X$ and $\eta'\in X'$ be a scheme-theoretic points with 
$\tau\left(\eta'\right)=\eta$ and $(X,\Delta)$ is klt at $\eta$. 
Take any big $\bQ$-line bundle $L$ on $X$ and take any effective 
$\bQ$-Weil divisor $\Delta_0$ on $X$ with $\Delta_0\leq \Delta$ and 
$K_X+\Delta_0$ $\bQ$-Cartier at $\eta$. Then we have 
\[
\delta_{\eta'}\left(X',\Delta';\tau^*L\right)\leq
\delta_\eta\left(X,\Delta_0;L\right). 
\]
\end{proposition}

\begin{proof}
As in the proof of \cite[Proposition 3.4]{ST24}, we have
\[
\delta_{\eta'}\left(X',\Delta';\tau^*L\right)\leq
\delta_\eta\left(X,\Delta;L\right). 
\]
For any prime divisor $E$ over $X$ with $\eta\in c_X(E)$, we have 
$A_{X,\Delta}(E)\leq A_{X,\Delta_0}(E)$. Thus the inequality 
\[
\delta_\eta\left(X,\Delta;L\right)\leq
\delta_\eta\left(X,\Delta_0;L\right) 
\]
is trivial and we obtain the assertion. 
\end{proof}

\begin{proposition}[{cf.\ 
\cite[Claim 5.4]{ST24}}]\label{proposition_ST-cover}
Let $n\in\bZ_{>0}$ and let $\bP:=\bP\left(a_0,\dots,a_{n+1}\right)$ 
be well-formed. Assume that $\left(a_1,\dots,a_{n+1}\right)$ is also 
well-formed. 
Set $D_0:=(x_0=0)\subset\bP$. Let $X\subset\bP$ be a quasi-smooth and 
well-formed hypersurface with $X|_{D_0}\subset D_0\left(\simeq 
\bP\left(a_1,\dots,a_{n+1}\right)\right)$ also quasi-smooth. 
Consider the morphism 
\begin{eqnarray*}
\tau\colon\bP':=\bP\left(1,a_1,\dots,a_{n+1}\right)&\to&\bP\\
\left[x_0:x_1:\cdots:x_{n+1}\right]&\mapsto&
\left[x_0^{a_0}:x_1:\cdots:x_{n+1}\right]
\end{eqnarray*}
as in Proposition \ref{proposition_finite-wps}, and set 
$X':=\tau^*X\subset\bP'$. Then $X'\subset\bP'$ is also a quasi-smooth 
and well-formed hypersurface, and 
\[
\delta_{p'}\left(X';\cO_{X'}(1)\right)\leq\delta_p\left(X;\cO_X(1)\right)
\]
holds for any $p\in X$ and $p'\in X'$ with $\tau(p')=p$. 
\end{proposition}

\begin{proof}
By \cite[Claim 5.4]{ST24}, 
the $X'\subset\bP'$ is quasi-smooth and well-formed. 
Moreover, we have 
\[
K_{X'}=\left(\tau|_{X'}\right)^*\left(K_X+\frac{a_0-1}{a_0}X|_{D_0}\right). 
\]
Thus the assertion follows by Proposition \ref{proposition_covering}. 
\end{proof}

\section{Barycenters of convex sets and stability thresholds on surfaces}\label{section:barycenter}

The following proposition gives some information on the barycenter of a convex subset in $\bR^2$ as in Figure \ref{fig:convex-set}. 












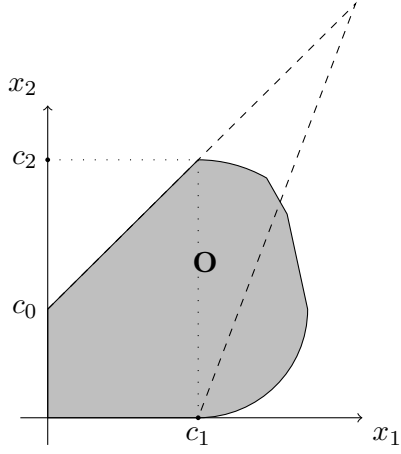
\begin{figure}
\centering
\begin{tikzpicture}[scale=0.90]
\filldraw[fill=lightgray,draw=black] (0,0) --(0,1.6) -- (2.2,3.8) 
arc(90:60:2) -- (3.5,3.0) -- (3.8,1.6) -- (3.8,1.6) arc(0:-90:1.6) -- (2.2,0)
-- cycle;
\draw[->] (-0.4,0) -- (4.6,0) node[below right] {$x_1$}; 
\draw[->] (0,-0.4) -- (0,4.6)  node[above left] {$x_2$};
\node at (current bounding box.center) {\(\mathbf{O}\)};
\draw[dashed] (0,1.6) -- (4.5,6.1) -- (2.2,0); 
\node[below] at (2.2,0) {$c_1$};
\node[left]  at (0,1.6) {$c_0$};
\node[left]  at (0,3.8) {$c_2$};
\fill (0,3.8) circle (1pt);
\draw[loosely dotted] (0, 3.8)--(2.2,3.8);
\fill (2.2,0) circle (1pt); 
\draw[loosely dotted] (2.2,0) -- (2.2, 3.8); 
\end{tikzpicture}
  \caption{A convex subset $\mathbf{O} \subset \mathbb{R}^2$. The coordinates of the barycenter of $\mathbf{O}$ are maximized when the figure of $\mathbf{O}$ is equal to the dashed area.}
  \label{fig:convex-set}
\end{figure}

\begin{proposition}\label{proposition_gravity}
Let $V, c_1, c_2\in\bR_{>0}$ be positive real numbers and let 
$c_0\in\bR_{\geq 0}$ be with $c_2\geq c_0$. 
Consider a compact convex set $\mathbf{O}\subset\bR^2$ with 
$\operatorname{vol}_{\bR^2}(\mathbf{O})=V$ satisfying 
\begin{eqnarray*}
&&\mathbf{O}\cap\left\{(x_1,x_2)\in\bR^2\,\,|\,\,x_1\leq c_1\right\}\\
&=&\left\{(x_1,x_2)\in\bR^2\,\,|\,\,0\leq x_1\leq c_1, \text{ and }0\leq x_2\leq\frac{c_2-c_0}{c_1}x_1+c_0\right\}. 
\end{eqnarray*}
(In particular, we have $c_1(c_0+c_2)/2\leq V$.) 
Let $(b_1,b_2)\in\mathbf{O}$ be the barycenter of the convex set 
$\mathbf{O}$. 
Let $p_1\colon\bR^2\to\bR$ be the projection to the first coordinate, set 
$p_1(\mathbf{O})=:[0,t_1]$ and 
let $g(x)$ be the length of $(p_1|_\mathbf{O})^{-1}(x)$ for any $x\in[0,t_1]$. 

\begin{enumerate}
\item 
We have 
\begin{eqnarray*}
b_1&\leq&
\frac{c_0^2c_1^2-4c_0c_1V+2c_1c_2V+4V^2}{6c_2V}\quad\text{and}\\
b_2&\leq&\frac{-c_0^3c_1^2+4c_0^2c_1V-4c_0V^2+4c_2V^2}{6c_1c_2V}. 
\end{eqnarray*}
\item 
Assume that 
\[
b_1=\frac{c_0^2c_1^2-4c_0c_1V+2c_1c_2V+4V^2}{6c_2V}.
\] 

Then we have $t_1=(2V-c_0c_1)/c_2$ and 
\[
g(x)=\begin{cases}
\frac{c_2}{c_1}x+c_0 & \text{if }x\in[0,c_1],\\
c_2\cdot\frac{t_1-x}{t_1-c_1} & \text{if }x\in[c_1, t_1].
\end{cases}
\]
\item 
If 
\[
b_2=\frac{-c_0^3c_1^2+4c_0^2c_1V-4c_0V^2+4c_2V^2}{6c_1c_2V},
\]
then we must have \[
b_1=\frac{c_0^2c_1^2-4c_0c_1V+2c_1c_2V+4V^2}{6c_2V}.
\] 
Moreover, $\mathbf{O}$ 
must be the convex hull of 
\[
\left\{(0,c_0), \quad (0,0), \quad (c_1,0), \quad
\left(\frac{2V-c_0c_1}{c_2},\frac{2c_2V-c_0(2V-c_0c_1)}{c_1c_2}\right)\right\}.
\]
\end{enumerate}
\end{proposition}

\begin{proof}
A part of the proof is essentially same as the proof in
\cite[Proposition 3.5 (2)]{Fujita:2024aa}. We give a complete proof for convenience. 
If $V=c_1(c_0+c_2)/2$, then the assertion is trivial. 
Thus we may assume that $V>c_1(c_0+c_2)/2$. 

Note that $g(x)$ is concave since $\mathbf{O}$ 
is convex by Brunn--Minkowski's theorem. 
By this and $g(c_1)=c_2$, it follows that 
\[
g(x)\geq c_2\cdot\frac{t_1-x}{t_1-c_1}
\]
for any $x\in[c_1,t_1]$. 
Therefore, we have
\[
V-\frac{c_1(c_0+c_2)}{2}=\int_{c_1}^{t_1}g(x)dx\geq
\int_{c_1}^{t_1}c_2\cdot\frac{t_1-x}{t_1-c_1}dx=\frac{c_2(t_1-c_1)}{2}. 
\]
This immediately implies that $t_1\leq (2V-c_0c_1)/c_2$. Let us set the function 
$h_2\colon [0, (2V-c_0c_1)/c_2]\to\bR$ defined to be 
\[
h_2(x):=\begin{cases}
\frac{c_2-c_0}{c_1}x+c_0 & \text{if }x\in[0,c_1], \\
c_2\cdot\frac{\frac{2V-c_0c_1}{c_2}-x}{\frac{2V-c_0c_1}{c_2}-c_1} & \text{if }x\in[c_1, (2V-c_0c_1)/c_2]. 
\end{cases}
\]

Let us first show the first inequality in (1) and the item (2), i.e., the statements on $b_1$. 

\noindent{\bf (Case 1)} Assume that $h_2(t_1)\leq g(t_1)$. Then $h_2(x)\leq g(x)$ for all $x\in[c_1,t_1]$. 
Thus, we have
\begin{eqnarray*}
&&\int_0^{\frac{2V-c_0c_1}{c_2}}x h_2(x)dx-t_1 V
=\int_0^{\frac{2V-c_0c_1}{c_2}}(x-t_1)h_2(x)dx\\
&\geq&\int_0^{t_1}(x-t_1)g(x)dx=\int_0^{t_1}x g(x)dx-t_1 V.
\end{eqnarray*}

This implies that 
\[
b_1 = \frac{1}{V}\int_{0}^{t_1} xg(x)dx \leq \frac{1}{V} 
\int_0^{\frac{2V-c_0c_1}{c_2}}x h_2(x)dx   = 
\frac{c_0^2c_1^2-4c_0c_1V+2c_1c_2V+4V^2}{6c_2V}.
\]
Moreover, if 
\[
b_1=\frac{c_0^2c_1^2-4c_0c_1V+2c_1c_2V+4V^2}{6c_2V},
\]
then we must have $t_1=(2V-c_0c_1)/c_2$ and 
$g(x)=h_2(x)$ for all $x\in[c_1,t_1]$. 

\vspace{2mm}

\noindent{\bf (Case 2)} Assume the remaining case $h_2(t_1)> g(t_1)$. If $h_2(x)\geq g(x)$ holds 
for any $x\in[c_1,t_1]$, then 
\[
V=\int_0^{t_1}g(x)dx<\int_0^{t_1}h_2(x)dx\leq V, 
\]
a contradiction. Thus there exists $s_2\in(c_1, t_1)$ such that 
\[
\begin{cases}
g(s_2)=h_2(s_2), \\
g(x)>h_2(x) \text{ for all }x\in(c_1, s_2), \\
g(x)<h_2(x) \text{ for all }x\in(s_2, t_1).
\end{cases}
\]
Since 
\begin{eqnarray*}
&&\int_0^{\frac{2V-c_0c_1}{c_2}}x h_2(x)dx-s_2 V=\int_0^{\frac{2V-c_0c_1}{c_2}}(x-s_2)h_2(x)dx\\
&>&\int_0^{t_1}(x-s_2)g(x)dx=\int_0^{t_1}x g(x)dx-s_2 V,
\end{eqnarray*}
we get the inequality 
\[
b_1<\frac{c_0^2c_1^2-4c_0c_1V+2c_1c_2V+4V^2}{6c_2V}.
\]
Hence we obtain the first inequality in (1) and (2). 

\vspace{2mm}

Now we show the rest.  
Let us consider the linear transform
\begin{eqnarray*}
f\colon\bR^2&\to&\bR^2\\
\begin{pmatrix}
x_1 \\ x_2
\end{pmatrix}
&\mapsto&
\begin{pmatrix}
y_1 \\ y_2
\end{pmatrix}=
\begin{pmatrix}
1 & 0\\
-\frac{c_2-c_0}{c_1} & 1
\end{pmatrix}
\begin{pmatrix}
x_1 \\ x_2
\end{pmatrix},
\end{eqnarray*}
and set $\mathbf{O}':=f(\mathbf{O})$. Let $(b'_1,b'_2)\in\mathbf{O}'$ 
be the 
barycenter of $\mathbf{O}'$. 
Since the determinant of the matrix is $1$, we have 
$\operatorname{vol}_{\bR^2}(\mathbf{O}')=V$ and 
\begin{eqnarray*}
&&\mathbf{O}'\cap\left\{(y_1,y_2)\in\bR^2\,\,|\,\,y_1\leq c_1\right\}\\
&=&\left\{(y_1,y_2)\in\bR^2\,\,|\,\,0\leq y_1\leq c_1\text{ and }
-\frac{c_2-c_0}{c_1}y_1\leq y_2\leq c_0
\right\}.
\end{eqnarray*}
In particular, we note that $\mathbf{O}'\subset\{(y_1,y_2)\in\bR^2\,\,|\,\,y_2\leq c_0\}$. 

Set 
\[
\mathbf{O}'':=\mathbf{O}'\cap\{(y_1,y_2)\in\bR^2\,\,|\,\, y_1\geq c_1\}. 
\]
Note that $\operatorname{vol}_{\bR^2}(\mathbf{O}'')=V-c_1(c_0+c_2)/2$. 
Let $p_2\colon\bR^2\to\bR$ be the projection to the second coordinate, 
and set $p_2(\mathbf{O}'')=:[-c'',c_0]$. Obviously, we have $c''\geq c_2-c_0$. 
For any $y\in[-c'',c_0]$, let $h(y)\in\bR_{\geq 0}$ be the length of 
$(p_2|_{\mathbf{O}''})^{-1}(y)$, which is continuous and concave over 
$y\in[-c'',c_0]$. In particular, for any $y\in[-(c_2-c_0),c_0]$, we have 
\[
h(y)\geq \frac{y+c_2-c_0}{c_2}h(c_0).
\]
By this and  
$
V-\frac{c_1(c_0+c_2)}{2}=\int_{-c''}^{c_0}h(y)dy\geq
\int_{-(c_2-c_0)}^{c_0}\frac{y+c_2-c_0}{c_2}
h(c_0)dy=\frac{c_2h(c_0)}{2}
$, we obtain 
\[
h(c_0)\leq \frac{2V-c_1(c_0+c_2)}{c_2}. 
\]
Let us set 
\[
l(y):=\frac{2V-c_1(c_0+c_2)}{c_2^2}(y+c_2-c_0), 
\quad\tilde{h}(y):=h(y)-l(y)
\]
for any $y\in[-c'',c_0]$. 
Since $\tilde{h}\left(-(c_2-c_0)\right)\geq 0$ and 
$\tilde{h}(c_0)\leq 0$, there exists $-c_3\in[-(c_2-c_0),c_0]$ such that 
\[
\begin{cases}
\tilde{h}(-c_3)=0, \\
\tilde{h}(y)\geq 0 \text{ for all }y\in[-(c_2-c_0), -c_3], \\
\tilde{h}(y)\leq 0 \text{ for all }y\in[-c_3,c_0].
\end{cases}
\]
Therefore, we get $\int_{-c''}^{c_0} yh(y)dy \le \int_{-(c_2-c_0)}^{c_0}y l(y)dy$ from 
\begin{eqnarray*}
&&\int_{-c''}^{c_0} yh(y)dy+c_3\frac{2V-c_1(c_0+c_2)}{2}
=\int_{-c''}^{c_0}(y+c_3)h(y)dy\\
&\leq& \int_{-(c_2-c_0)}^{c_0}(y+c_3)l(y)dy
=\int_{-(c_2-c_0)}^{c_0}y l(y)dy+c_3\frac{2V-c_1(c_0+c_2)}{2}.
\end{eqnarray*}
This implies that 
\begin{eqnarray*}
b'_2&=&\frac{1}{V}\left(
\int_{-(c_2-c_0)}^0 y\frac{c_1}{c_2-c_0}(y+c_2-c_0)dy
+\int_0^{c_0}y c_1 dy
+\int_{-c''}^{c_0} y h(y)dy\right)\\
&=&\frac{1}{V}\left(\frac{1}{6}c_1
\left(2c_0^2+2c_0c_2-c_2^2\right)+\int_{-c''}^{c_0}yh(y)dy\right)\\
&\leq&\frac{1}{V}\left(\frac{1}{6}c_1
\left(2c_0^2+2c_0c_2-c_2^2\right)+\int_{-(c_2-c_0)}^{c_0}yl(y)dy\right)\\
&=&c_0\left(1-\frac{c_0c_1}{6V}\right)-\frac{c_2}{3}. 
\end{eqnarray*}
Moreover, if \[
b'_2=c_0\left(1-\frac{c_0c_1}{6V}\right)-\frac{c_2}{3}, 
\]
then we must have $c''=c_2-c_0$, and 
$h(y)=l(y)$ for all $y\in[-(c_2-c_0),c_0]$. 
Since $(p_2|_{\mathbf{O}''})^{-1}(y)
=[c_1,c_1+h(y)]$, the convex set $\mathbf{O}''$ must be the convex hull of the set 
\[
\left\{(c_1,c_0), \quad \left(\frac{2V-c_0c_1}{c_2},c_0\right),\quad 
\left(c_1,-(c_2-c_0)\right)\right\}. 
\]
Thus $\mathbf{O}$ must be the convex hull of the set 
\[
\left\{(0,c_0), \quad (0,0), \quad (c_1,0), \quad
\left(\frac{2V-c_0c_1}{c_2},\frac{2c_2V-c_0(2V-c_0c_1)}{c_1c_2}\right)\right\}.
\]
In this case, we can directly check that 
\[
b_1=\frac{c_0^2c_1^2-4c_0c_1V+2c_1c_2V+4V^2}{6c_2V}.
\]

On the other hand, we know that $b'_2=-b_1(c_2-c_0)/c_1+b_2$. 
Hence we have 
\begin{eqnarray*}
b_2&=&b'_2+\frac{c_2-c_0}{c_1}b_1\\
&\leq& 
c_0\left(1-\frac{c_0c_1}{6V}\right)-\frac{c_2}{3}+\frac{c_2-c_0}{c_1}\cdot
\frac{c_0^2c_1^2-4c_0c_1V+2c_1c_2V+4V^2}{6c_2V}\\
&=&\frac{-c_0^3c_1^2+4c_0^2c_1V-4c_0V^2+4c_2V^2}{6c_1c_2V}. 
\end{eqnarray*}
Thus we get the assertion. 
\end{proof}

Proposition \ref{proposition_gravity} is useful especially when $c_0$ is equal to zero. The following is a generalization of \cite[Lemma 3.2]{AZ23} and gives a lower bound for the stability threshold in terms of the Seshadri constant.

\begin{corollary}\label{corollary_gravity}
Let $(S,\Delta)$ be a $2$-dimensional projective klt pair and let 
$L$ be an ample $\bQ$-line bundle on $S$. 
Let us consider a plt blowup $\sigma\colon \tilde{S}\to S$ of $(S,\Delta)$ 
with the exceptional divisor $C\subset \tilde{S}$, that is, $-C$ on 
$\tilde{S}$ is $\sigma$-ample and the pair 
$(\tilde{S},\Delta_{\tilde{S}}+C)$ is plt, where 
$\Delta_{\tilde{S}}:=\sigma^{-1}_*\Delta$. 
Set $K_C+\Delta_C:=(K_{\tilde{S}}+\Delta_{\tilde{S}}+C)|_C$ and 
\[
\Delta_C=:\sum_{i=1}^m d_i q_i, \quad d_C:=\sum_{i=1}^m d_i,
\]
where $d_i\in(0,1)\cap\Q$ with $d_1\leq\cdots\leq d_m$ and 
$q_1,\dots,q_m\in C$ are mutually distinct points. 
Assume that a positive real number $\varepsilon\in\bR_{>0}$ satisfies 
$\varepsilon\leq\varepsilon(L; C)$, where 
\[
\varepsilon(L; C):=\max\left\{x\in\bR_{\geq 0}\,\,|\,\,
\sigma^*L-x C\text{ is nef}\right\}. 
\]
\begin{enumerate}
\item\label{item:1stineqonS(L;C)} 
We have an inequality
\[
S(L; C)\leq\frac{\varepsilon^2\frac{2
-d_C}{A_{S,\Delta}(C)}
+(L^2)}{\frac{3\varepsilon(2-d_C)}{A_{S,\Delta}(C)}},  
\]
where $S(L;C)$ and $A_{S,\Delta}(C)$ are as in Definition \ref{definition_delta}. 
Moreover, if the equality holds, then we have 
$\varepsilon=\varepsilon(L; C)$, 
\[
T(L; C)=\frac{(L^2)A_{S,\Delta}(C)}{\varepsilon (2-d_C)}, 
\]
and the restricted volume
\[
\operatorname{vol}_{\tilde{S}|C}(\sigma^*L-x C)=\begin{cases}
\frac{2-d_C}{A_{S,\Delta}(C)}x & 
\text{if }x\in\left[0,\varepsilon\right], \\
\frac{\varepsilon(2-d_C)}{A_{S,\Delta}(C)}\cdot
\frac{T(L;C)-x}{T(L; C)-\varepsilon} & \text{if }x\in \left[\varepsilon, 
T(L;C)\right).
\end{cases}
\]
\item 
For any $q\in C$, we have an inequality 
\[
S(L; C\triangleright q)\leq\frac{(L^2)}{3\varepsilon},
\]
where $S(L; C\triangleright q)$ is as in Definition \ref{definition_plt-flag}(3). 
Moreover, if the equality holds, then we must have the equality in (\ref{item:1stineqonS(L;C)}), that is, 
\[
S(L; C)=\frac{\varepsilon^2\frac{2
-d_C}{A_{S,\Delta}(C)}
+(L^2)}{\frac{3\varepsilon(2-d_C)}{A_{S,\Delta}(C)}}. 
\]
\item 
Let $p\in S$ be the center of $C$ on $S$. Then we have 
\[
\delta_p(S,\Delta; L)\geq\min\left\{
\frac{3\varepsilon(2-d_C)}{\varepsilon^2\frac{2-d_C}{A_{S,\Delta}(C)}
+(L^2)},\quad\frac{3\varepsilon(1-d_m)}{(L^2)}\right\}. 
\]
Moreover, if the equality holds,
then we must have 
\[
S(L; C)=\frac{\varepsilon^2\frac{2
-d_C}{A_{S,\Delta}(C)}
+(L^2)}{\frac{3\varepsilon(2-d_C)}{A_{S,\Delta}(C)}}. 
\]
In particular, all the properties in the equality case of (1) also hold. 
(We remark that, if $2d_m\geq d_C$, then the inequality
\[
\frac{3\varepsilon(2-d_C)}{\varepsilon^2\frac{2-d_C}{A_{S,\Delta}(C)}
+(L^2)}\geq\frac{3\varepsilon(1-d_m)}{(L^2)}
\]
holds. 
\end{enumerate}
\end{corollary}

\begin{proof} We first show the assertions (1) and (2). 
By adjunction, we have 
\begin{equation}\label{eq:degK_C+Delta_C}
\deg_C(K_C+\Delta_C)=A_{S,\Delta}(C)\cdot(C^2).   \tag{*}  
\end{equation}
Thus we have 
\[
(C^2)=-\frac{2-d_C}{A_{S,\Delta}(C)}. 
\]
For any $q\in C$, the sequence $C\triangleright q$ is a complete plt 
flag over $(S,\Delta)$ in the sense of Definition \ref{definition_plt-flag}. 
Let $\mathbf{O}\subset\bR^2$ be the Okounkov body of $L$ associated with 
$C\triangleright q$ in the sense of Definition \ref{definition_plt-flag}. 
(If $q\in \tilde{S}$ is a smooth point, then $\mathbf{O}$ is nothing but 
the usual Okounkov body of $\sigma^*L$ associated with 
the admissible flag $C\supset\{q\}$.)
Take any $x\in(0,\varepsilon)\cap\bQ$ and a sufficiently divisible 
positive integer $k\gg 0$. Since $\sigma^*L-x C$ is ample, 
the restriction homomorphism 
\[
H^0(\tilde{S},k(\sigma^*L-x C))\to H^0\left(C,
\cO_{\bP^1}\left(k x\frac{2-d_C}{A_{S,\Delta}(C)}\right)\right)
\]
is surjective. Therefore, we have 
\[
\mathbf{O}\cap\{(x_1,x_2)\in\bR^2\,\,|\,\,x_1\leq \varepsilon\}
=\left\{(x_1,x_2)\in\bR^2\,\,|\,\,
0\leq x_1\leq\varepsilon\text{ and }0\leq x_2\leq
\frac{2-d_C}{A_{S,\Delta}(C)}x_1\right\}. 
\]
Set $p_1(\mathbf{O})=:[0,T_1]$ and let $(b_1,b_2)\in\mathbf{O}$ be the 
barycenter of $\mathbf{O}$. 
Then, as in \cite[Remark 4.8 and Definition 4.9]{Fujita:2024aa}, we have 
\[
b_1=S(L; C), \quad b_2=S(L; C\triangleright q),\quad T_1=T(L; C). 
\]
Moreover, for any $x\in\left[0,T(L; C)\right)$, the length of 
$(p_1|_{\mathbf{O}})^{-1}(x)$ is equal to 
$\operatorname{vol}_{\tilde{S}|C}(\sigma^*L-x C)$ by \cite[Corollary 4.27]{LM09}. 
Thus the assertions (1) and (2) immediately follow from Proposition 
\ref{proposition_gravity} by setting $c_0=0, c_1= \varepsilon, c_2= 
\frac{\varepsilon(2-d_C)}{A_{S,\Delta}(C)}$ and $V= \frac{(L^2)}{2}$. (The 
equality $\varepsilon=\varepsilon(L; C)$ 
in (1) follows from \cite[Lemma 10]{MR3797604}.)

The assertion (3) follows from (1), (2) and Theorem \ref{theorem:AZ} (cf.\cite{AZ22}, 
\cite[Theorem 12.3]{Fujita:2024aa}). 
The last remark follows from the inequality 
$0\leq (\sigma^*L-\varepsilon C)^2$ and the equality (\ref{eq:degK_C+Delta_C}). 
\end{proof}

As an immediate corollary of Corollary \ref{corollary_gravity}, 
we have the following result. When $a=1$, this result is a part of 
\cite[Lemma 3.2]{AZ23}. 

\begin{corollary}\label{corollary_lobster}
Let $S$ be a projective klt surface, let 
$L$ be an ample $\bQ$-line bundle on $S$, 
and let $p\in S$ be a point 
of the singularity $\bA^2_{\mathbf{x}_1\mathbf{x}_2}/\bZ_a(1,1)$ for some $a \in \bZ_{>0}$ (cf. Example \ref{example_swb-pt}). Let 
$\sigma\colon \tilde{S}\to S$ be the ordinary blowup at $p\in S$ 
with the exceptional divisor $C\subset \tilde{S}$. 
Then we have 
\[
\delta_p(S; L)\geq\frac{3\cdot\varepsilon(L; C)}{(L^2)}, 
\]
where the definition of $\varepsilon(L; C)$ can be found 
in Corollary \ref{corollary_gravity}. 
Moreover, if the equality holds, then we have the equalities 
\begin{eqnarray*}
S(L; C)&=&\frac{a\cdot \varepsilon(L;C)^2+(L^2)}{3a\cdot\varepsilon(L;C)},\\
\varepsilon(L; C)\cdot T(L; C)&=&\frac{(L^2)}{a}.
\end{eqnarray*}
\end{corollary}

\begin{proof}
Under the assumption, we have $d_C=0$ and $A_S(C)=2/a$. 
Thus we get the assertion directly from Corollary 
\ref{corollary_gravity}. 
\end{proof}

\begin{example}\label{example_hirzebruch}
\begin{enumerate}
\item 
When $S=\bP(1,1,a)$, $L=\cO_{\bP}(1)$ and $p\in S$ is the point of 
singularity $\bA^2_{\mathbf{x}_1\mathbf{x}_2}/\bZ_a(1,1)$, then we can easily 
check that 
$(L^2)=1/a$, $\delta(S;L)=3$ 
(see \cite[Corollary 7.16]{MR4067358}) and $\varepsilon(L;C)=T(L;C)=1/a$. 
\item
Take any $a\in\bZ_{>0}$, let us set $S:=\bP(1,a,a+1)$, 
take $L:=\cO_{\bP}(1)$ and let $p\in S$ be a point of the 
singularity $\bA^2_{\mathbf{x}_1\mathbf{x}_2}/\bZ_a(1,1)$. 
If $a\geq 2$, then $p\in S$ is 
uniquely determined. (If $a=1$, then take $p\in S$ to be an arbitrary 
smooth point in $S$.)
Note that, if $a\geq 2$, the complete linear system 
$|L|$ consists of one prime divisor $D\subset S$. 
(If $a=1$, then take $D\in |L|$ to be the element with $p\in D$.)
Let $\sigma\colon\tilde{S}\to S$ 
be the ordinary blowup at $p$ and let $C\subset\tilde{S}$ be the 
exceptional divisor. 

We know that there exists a birational morphism 
$\nu \colon \tilde{S}\to\bP^2$ contracting the strict transform of 
$D$ as follows. 
The $S$ corresponds to the complete fan $\Sigma$ in $N=\bZ^2$ whose set of 
the primitive generators of $\Sigma(1)$ is $(1,0)$, $(0,1)$, $(-a,-a-1)$. The blowup 
$\sigma$ corresponds to the star subdivision $\tilde{\Sigma}$ of 
$\Sigma$ adding the $1$-dimensional cone whose 
primitive generator is $(-1,-1)$. The morphism $\tilde{S}\to \bP^2$ is nothing but the 
contraction of the ray $\bR_{\geq 0}(-a,-a-1)$. 

From the birational morphism $\nu$, we can easily compute that 
\begin{eqnarray*}
(L^2)=\frac{1}{a(a+1)}, &\quad&
\delta_p\left(S; L\right)=\delta\left(S; L\right)
=\frac{A_S(D)}{S(L; D)}=3,\\
\varepsilon(L;C)=\frac{1}{a(a+1)}, &\quad& T(L;C)=\frac{1}{a}, \quad
S(L; C)=\frac{a+2}{3a(a+1)}. 
\end{eqnarray*}
Thus we have 
\[
\delta_p\left(S; L\right)=\frac{3\cdot\varepsilon(L;C)}{(L^2)}
=3<\frac{A_S(C)}{S(L; C)}
=\frac{6(a+1)}{a+2}. 
\]
\end{enumerate}
\end{example}

On the other hand, Proposition \ref{proposition_gravity} is also useful 
even when $c_0>0$. 

\begin{proposition}\label{proposition_perhaps-useful}
Let $a,b\geq 1$, $k\geq 2$ and set $d:=bk+a$. Let us consider 
\[
\tilde{\bP}:=\bP_{\bP^2}\left(\cO\oplus\cO(b)\right)\xrightarrow{\pi}
\bP^2,
\]
let $\xi\in\operatorname{Pic}\tilde{\bP}$ be the tautological line bundle 
of $\bP_{\bP^2}\left(\cO\oplus\cO(b)\right)/\bP^2$, let 
$\bE \in |\xi - \pi^* \cO_{\bP^2}(b)|$ on $\tilde{\bP}$ be the negative section of $\pi$, 
let $\psi\colon\tilde{\bP}\to\bP(1,1,1,b)$ be the contraction of $\bE$.  
 
Let $\tilde{S}\subset\tilde{\bP}$ be a smooth projective surface
such that $\tilde{S}\sim k\xi+\pi^*\cO_{\bP^2}(a)$ and that 
$E:=\bE|_{\tilde{S}}$ is smooth (i.e., 
$E\subset\bE$ is a smooth plane curve of degree $a$). Let us take $\tilde{p}\in \tilde{S}\setminus\left(\bE
\cap \tilde{S}\right)$.
Set $S:=\psi\left(\tilde{S}\right)\subset\bP(1,1,1,b)$ and $p:=
\psi\left(\tilde{p}\right)\in S$. (Note that $S$ is a well-formed weighted  
hypersurface of degree $d$ and $p\in S$ is a smooth point.)

Then we have the inequality 
\[
\delta_p\left(S; \cO_S(1)\right)\geq \frac{3b}{d}. 
\]
\end{proposition}

\begin{proof}
Set $\sigma:=\psi|_{\tilde{S}}\colon\tilde{S}\to S$. Moreover, 
let $\tilde{l}\subset\tilde{\bP}$ be the fiber of $\pi$ passing through 
$\tilde{p}$. We set $l:=\psi(\tilde{l})\subset\bP(1,1,1,b)$. 

We firstly consider the case $\tilde{l}$ is contained in $\tilde{S}$. 
By adjunction, 
\[
-2=\left(K_{\tilde{\bP}}+\tilde{S}\right)\cdot \tilde{l}+
(\tilde{l}^2)=-2+k+(\tilde{l}^2). 
\]
Thus $(\tilde{l}^2)=-k$. Moreover, $(E^2)=\left(\bE^2\cdot\tilde{S}\right)=-ab$. 
Take a general element $C$ in the complete linear system 
$\left|\pi^*\cO_{\bP^2}(1)|_{\tilde{S}}-\tilde{l}\right|$.
Since the linear system is base point free, any connected component 
of $C$ is a smooth curve with the self intersection number $0$. 
We also remark that $(E\cdot \tilde{l})=1$, 
$(E\cdot C)=a-1$ and $(C\cdot \tilde{l})=k$
by easy computations. 
Note that 
\begin{eqnarray*}
\sigma^*\cO_S(1)-x\tilde{l}&\sim_\bR&\frac{1}{b}E+(1-x)\tilde{l}+C, \\
\left(\sigma^*\cO_S(1)-x\tilde{l}\right)\cdot E&=&-x, \\
\left(\left(\frac{1}{b}-\frac{x}{ab}\right)E+(1-x)\tilde{l}
+C\right)\cdot \tilde{l}
&=&\frac{1}{b}+\left(k-\frac{1}{ab}\right)x.
\end{eqnarray*}
Thus, for $x\in [0,1]$, the negative part of the Zariski decopmosition 
of $\sigma^*\cO_S(1)-x\tilde{l}$ is equal to $\frac{x}{ab}E$. 

Let $\mathbf{O}\subset\bR^2$ be the Okounkov body of $\cO_S(1)$ with 
respect to the admissible flag $p\in l\subset S$. The $\mathbf{O}$ is 
equal to 
the Okounkov body of $\sigma^*\cO_S(1)$ with respect to the admissible 
flag $\tilde{p}\in \tilde{l}\subset\tilde{S}$ according to the definition of Okounkov bodies since $\sigma$ is an isomorphism in $\tilde{p}$. 

From the above Zariski decomposition, 
since $\tilde{p}\not\in E$, we have 
\begin{eqnarray*}
&&\mathbf{O}\cap\left\{(x_1,x_2)\in\bR^2\,\,|\,\,x_1\leq 1\right\}\\
&=&\left\{(x_1,x_2)\in\bR^2\,\,|\,\,0\leq x_1\leq 1 \text{ and }
0\leq x_2\leq\frac{1}{b}+\left(k-\frac{1}{ab}\right)x_1 \right\}. 
\end{eqnarray*}
On the other hand, by using the remark in Definition \ref{definition_plt-flag}(\ref{item:Okbodydfn}) and Theorem \ref{theorem:AZ} with the plt flag $S \triangleright \tilde{l} \triangleright \tilde{p}$
(see also \cite[Corollary 11.16]{Fujita2023}), we have 
\[
\delta_p\left(S; \cO_S(1)\right)\geq 
\min \left\{ 
\frac{A_S(\tilde{l})}{S(\cO_S(1); \tilde{l})}, \frac{A_{\tilde{l}}(\tilde{p})}{S(\cO_S(1); \tilde{l} \triangleright \tilde{p})}
\right\} = 
\min\left\{
\frac{1}{b_1},\frac{1}{b_2}\right\}, 
\]
where $(b_1,b_2)\in\mathbf{O}$ is the barycenter of $\mathbf{O}$. 

By applying Proposition \ref{proposition_gravity} 
with $c_0=\frac{1}{b}, c_1=1, c_2= \frac{1}{b}+ (k- \frac{1}{ab} )$ and $V= \frac{d}{2b}$, 
we can calculate  
\[
\delta_p\left(S; \cO_S(1)\right)\geq\min\left\{
\frac{3d(ad-a^2+a-1)}{2ad^2-(a^2+a+1)d+a},\,\,
\frac{3bd(ad-a^2+a-1)}{ad^3-(a^2+1)d^2+2ad-a}\right\}. 
\]
Since 
\[
d^2(ad-a^2+a-1)-\left(ad^3-(a^2+1)d^2+2ad-a\right)=a(d-1)^2>0, 
\]
we have 
\[
\frac{3bd(ad-a^2+a-1)}{ad^3-(a^2+1)d^2+2ad-a}>\frac{3b}{d}. 
\]
On the other hand, since $k\geq 2$, we have $b\leq(d-a)/2$. Observe that 
\begin{eqnarray*}
&&2d^2(ad-a^2+a-1)-(d-a)\left(2ad^2-(a^2+a+1)d+a\right)\\
&=&(a^2+3a-1)d^2-a(a^2+a+2)d+a^2
>a^2(2a-2)\geq 0, 
\end{eqnarray*}
where the strict inequality follows from that the function 
\[
(a^2+3a-1)x^2-a(a^2+a+2)x+a^2
\]
is strictly increasing when $x>a$ and the fact $d>a$. 
This also implies that
\[
\frac{3d(ad-a^2+a-1)}{2ad^2-(a^2+a+1)d+a}>\frac{3b}{d}. 
\]
Thus, when $\tilde{l}\subset\tilde{S}$, we get the 
strict inequality
\[
\delta_p\left(S;\cO_S(1)\right)>\frac{3b}{d}. 
\]

We consider the remaining case $\tilde{l}\not\subset\tilde{S}$. 
Let $C'\subset S$ be a general element of the sub linear system 
of $|\cO_S(1)|$ passing through $p$. By \cite[Lemma 3.1(2)]{ST24}, 
the curve $C'$ is irreducible and quasi-smooth outside the point 
$\psi(\bE)$. (In particular, $C'\sim \cO_S(1)$ and $p\in C'$ is a 
smooth point.) 
Let $\mathbf{O}'$ be the Okounkov body of $\cO_S(1)$ with respect to 
the admissible flag $p\in C'\subset S$. Since $C'\sim \cO_S(1)$ and 
$\left(\cO_S(1)^2\right)=d/b$, the Okounkov body $\mathbf{O}'$ is the 
convex hull of
\[
\left\{(0,0),\quad(1,0),\quad\left(0,\frac{d}{b}\right)\right\}. 
\]
Thus the barycenter $(b'_1,b'_2)$ of $\mathbf{O}'$ satisfies that 
\[
(b'_1,b'_2)=\left(\frac{1}{3},\quad\frac{d}{3b}\right). 
\]
Again by Theorem \ref{theorem:AZ} 
(see also \cite[Corollary 11.16]{Fujita2023}), we have 
\[
\delta_p\left(S;\cO_S(1)\right)\geq\frac{3b}{d}. 
\]
Thus we get the assertion. 
\end{proof}

\begin{remark}\label{remark_smooth-points}
Let $S\subset\bP(1,1,1,b)$ be a quasi-smooth surface of degree 
$d=kb+1$. Then, it follows from 
Proposition \ref{proposition_perhaps-useful} that, we have 
\[
\delta_p\left(S; \cO_S(1)\right)\geq \frac{3b}{d}
\]
for any $p\in S$ with $p\neq [0:0:0:1]$. 
In fact, let us take $\tilde{\bP}\to\bP(1,1,1,b)$ 
and $\bE\subset\tilde{\bP}$ as in 
Proposition \ref{proposition_perhaps-useful} and set 
$\tilde{S}:=\psi^{-1}_*S\subset\tilde{\bP}$. Since $S$ is quasi-smooth, 
by Proposition \ref{proposition_qsm-wbl}, 
the divisor $\tilde{S}\subset\tilde{\bP}$ is smooth
and $\tilde{S}\cap\bE$ is a line in $\bE$. 
\end{remark}

The following gives a lower bound of a stability threshold on a smooth point of a quasi-smooth hypersurface as in Theorem \ref{thm:d=ak+1}

\begin{corollary}\label{corollary_smooth-points}
Let $n\in\bZ_{\geq 2}$ and let $X\subset\bP\left(1^{n+1},b\right)$
be a quasi-smooth hypersurface of degree $d>1$. Assume that 
$P_{n+1}:=[0:\cdots:0:1]\in X$. Then, for any $p\in X$ with 
$p \neq P_{n+1}$, we have 
\[
\delta_p\left(X; \cO_X(1)\right)\geq\frac{(n+1)b}{d}. 
\]
\end{corollary}

\begin{proof}
Since $P_{n+1}\in X$, we have $d=kb+1$ for some 
$k\in\bZ_{>0}$. By using Theorem \ref{theorem:AZ}, Lemmas \ref{lem:cutting} and \ref{lem:hyperplane_mod} repeatedly, 
we can assume that $n=2$ since we have $\dim B^{\bP}_{1,p} \le 1=3-2$. Then the assertion follows from Remark 
\ref{remark_smooth-points}. 
\end{proof}

\section{Bounds on local stability thresholds of singular points}\label{section_singular}

In this section, we fix the following setting.

\begin{setup}\label{setup:d=ak+1}
Let $n\geq 2$, $a, k\in\bZ_{>0}$, $\bP:=\bP(1^{n+1},a)$ 
with coordinates $x_0,\dots,x_n,x_{n+1}$ and let $X=X_{ak+1}
\subset\bP$ be a quasi-smooth hypersurface of degree $d=ak+1$ defined by 
the equation
\[
f:=x_{n+1}^kx_0+x_{n+1}^{k-1}f_{a+1}(x_0,\dots,x_n)+\cdots+
f_{ak+1}(x_0,\dots,x_n),
\]
and we set 
\[
m:=\min\left\{t\in\bZ\cap [1, k]\,\,|\,\,
f_{at+1} \text{ is not divisible by }x_0\right\}.
\]
\end{setup}

Let $\psi\colon\tilde{\bP}\to\bP$ be the ordinary blowup 
at $P=[0:\cdots:0:1]\in\bP$ and let $\bE\subset\tilde{\bP}$ 
be the exceptional divisor. 
In other words, $\psi$ is the standard weighted blowup of $\bP$ along 
$\bP(a)$. As in Example \ref{example_typical}, 
we know that $\tilde{\bP}
=\bP_{\bP^n}\left(\cO\oplus\cO(a)\right)\xrightarrow{\pi}
\bP^n_{x_0\cdots x_n}$. 
We set $\bH:=\pi^*(x_0=0)\simeq
\bP_{\bP^{n-1}_{x_1\dots x_n}}\left(\cO\oplus
\cO(a)\right)$. 
Let $x_0,\dots,x_n,y:=y_{n+1},z$ be multi-homogeneous coordinates 
of $\tilde{\bP}$ with $\deg x_i=(1,0)$, $\deg y=(0,1)$ and 
$\deg z=(-a,1)$. (Thus, $\bE\subset\tilde{\bP}$ is defined by the equation 
$z=0$.)
Let $\xi\in\operatorname{Pic}\tilde{\bP}$ be the tautological line bundle 
with respect to $\bP_{\bP^n}\left(\cO\oplus\cO(a)\right)/\bP^n$. 
We note that 
\[
\psi^*\cO_{\bP}(1)\sim_\bQ\frac{1}{a}\xi\quad\text{and}\quad
\bE\sim\xi-\pi^*\cO_{\bP^n}(a).
\]
Take $\tilde{X}:=\psi^{-1}_*X\subset\tilde{\bP}$, set 
$\sigma:=\psi|_{\tilde{X}}\colon\tilde{X}\to X$, 
$g:=\pi|_{\tilde{X}}\colon\tilde{X}\to\bP^n$, 
and set 
$E:=\bE|_{\tilde{X}}$. 

The divisor $E$ is a reduced hyperplane in $\bE$ defined by 
$x_0=0$, i.e., $E=\bE|_{\bH}$. The divisor $\tilde{X}\subset\tilde{\bP}$ is a smooth 
hypersurface defined by the equation 
\[
\tilde{f}:=y^kx_0+y^{k-1}zf_{a+1}(x_0,\dots,x_n)+\cdots+
z^kf_{ak+1}(x_0,\dots,x_n) 
\]
by Proposition \ref{proposition_strict-transform}. 
In particular, we have 
\[
\tilde{X}\sim k\xi+\pi^*\cO_{\bP^n}(1), \quad
\psi^*X=\tilde{X}+\frac{1}{a}\bE.
\]
We note that $g$ is a generically finite morphism of degree $k$ and 
\[
\cO_{\tilde{X}}(E)|_E\simeq\cO_{\bP^{n-1}}(-a), \quad
A_X(E)=\frac{n}{a}
\]
since 
\[
\left(E|_E\cdot\pi^*\cO_{\bP^n}(1)|_E^{n-2}\right)
=\left(\bE^2\cdot\pi^*\cO_{\bP^n}(1)^{n-2}\cdot\tilde{X}\right)
=-a.
\]
From the definition of $m$, the effective divisor 
$\tilde{X}|_{\bH}$ on $\bH$ has order $m$ along $E$. 
We set $G:=\bH|_{\tilde{X}}-m E$.
Note that 
\begin{equation}\label{eq:Glineq}
\cO_{\tilde{X}}(G)\simeq\cO_{\tilde{\bP}}\left(-m \xi\right)|_{\tilde{X}}
\otimes g^*\cO_{\bP^n}(am+1). 
\end{equation}

\begin{lemma}\label{lemma_irr-G}
If $n\geq 3$, then the effective divisor $G$ on $\tilde{X}$ is 
a prime divisor. 
\end{lemma}

\begin{proof}
The push-forward $\sigma_*G$ on $X$ is the Weil divisor on $X$ defined 
by $(x_0=0)|_X$. Since we have $\Cl X = \bZ \cdot \cO_X(1)$ as in Example \ref{example_irrelevant}(1), 
the divisor $\sigma_*G \in |\cO_X(1)|$ must be a prime divisor on $X$. 
Thus $G=\sigma^{-1}_*(\sigma_*G)$ is also a prime divisor on $\tilde{X}$. 
\end{proof}

\subsection{A non-Eckardt case}\label{section_non-eckardt}

In this subsection, we consider the case $m\leq k-1$. 

\begin{lemma}\label{lemma_small}
Assume that $n\geq 3$ and $m\leq k-1$. Then 
any prime divisor $G'$ on $\tilde{X}$ satisfies that 
$\dim g(G')=n-1$. In other words, the Stein factorization 
of $g$ is a small birational morphism. 
\end{lemma}

\begin{proof}
Assume that there exists a prime divisor $G_0\subset\tilde{X}$
such that $\dim g(G_0)\leq n-2$. Then we must have 
$G_0\subset\bH$. 
(Indeed, any fiber $f$ of $g|_{G_0} = \pi|_{G_0}$ is positive-dimensional and a fiber of $\pi$, thus we see that $f \cap \bE$ is a point on $X$ and $f \cap \bE \subset \tilde{X} \cap \bE = E$. This implies that $f \subset \pi^{-1}(\pi(E))=\bH$.)  
Together with Lemma \ref{lemma_irr-G}, 
we must have $G_0=G$. However, 
\[
0=\left(G\cdot g^*\cO_{\bP^n}(1)^{\cdot n-1}\right)=k-m, 
\]
a contradiction. 
\end{proof}

\begin{corollary}\label{corollary_G-irr}
Assume that $n\geq 3$ and $m\leq k-1$. Take a general 
$\tilde{S}\in|g^*\cO_{\bP^n}(1)|$. Then $G|_{\tilde{S}}$ is a 
prime divisor on $\tilde{S}$ with 
\[
G|_{\tilde{S}}\sim_\bQ\sigma^*\cO_X(1)|_{\tilde{S}}
-\frac{am+1}{a}E|_{\tilde{S}}. 
\]
\end{corollary}

\begin{proof}
The $\tilde{S}$ is smooth. Moreover, by applying the Bertini's theorem for 
irreducibility (see \cite[Theorem 3.3.1]{Lazarsfeld04a}) to a morphism $g|_G \colon G \to \bP^n$, 
we see that the divisor $G|_{\tilde{S}}$ is a prime divisor on $\tilde{S}$. 
The remaining assertion follows from the linear equivalence (\ref{eq:Glineq}). 
\end{proof}

\begin{lemma}\label{lemma_sesh-non-ec}
Assume that $n=3$ and $m\leq k-1$. Then, for any 
irreducible curve $\tilde{C}\subset\tilde{X}$ with 
$\dim g\left(\tilde{C}\right)=1$, we have 
\begin{equation}\label{equation_sesh-non-ec}
\left(\xi-\left(1+\frac{a}{am+1}\right)\bE\right)
\cdot\tilde{C}\geq 0.
\end{equation}
\end{lemma}

\begin{proof}
Take any irreducible curve $\tilde{C}\subset\tilde{X}$ with 
$C_0:=g\left(\tilde{C}\right)$ a curve. 
Since the $\bQ$-line bundle 
$\left(\xi-\left(1+\frac{a}{am+1}\right)\bE\right)\big|_{\bE}$ is ample on 
$\bE \simeq \bP^n$, we may assume that $\tilde{C}\not\subset\bE$ in order to show 
the inequality \eqref{equation_sesh-non-ec}. 
Let us set 
\[
b_H:=
\left(\pi^*\cO_{\bP^3}(1)\cdot\tilde{C}\right)\in\bZ_{>0}, \quad
b_E:=
\left(\bE\cdot\tilde{C}\right)\in\bZ_{\geq 0}. 
\]
Note that $\xi \cdot \tilde{C} = b_{E} + ab_H$ and 
\[
\left(\xi-\left(1+\frac{a}{am+1}\right)\bE\right)
\cdot\tilde{C}=\frac{a}{am+1}\left((am+1)b_H-b_E\right).
\]
Let $\nu\colon C\to C_0$ be the normalization, set 
$\cL:=\nu^*\left(\cO_{\bP^3}(1)|_{C_0}\right)$, and set 
$b_0:=\deg\cL\in\bZ_{>0}$. Moreover, let 
\[
\xymatrix{
\tilde{T} \ar[r]^{\nu'} \ar[d]_{\pi'} & \tilde{\bP} 
\ar[d]^{\pi} \\
C \ar[r]_{\nu} & \bP^3
}
\]
be the fiber product. Note that $\tilde{T}=\bP_C
\left(\cO\oplus\cL^{\otimes a}\right)$, $\xi_{\tilde{T}}:=
(\nu')^*\xi$ be the tautological line bundle, and 
$E_{\tilde{T}}:=(\nu')^*\bE$ be the negative section, i.e., 
$E_{\tilde{T}}\sim \xi_{\tilde{T}}-(\pi')^*\cL^{\otimes a}$. 
We can uniquely take an irreducible curve 
$C_{\tilde{T}}\subset\tilde{T}$ satisfying 
$\nu'_*C_{\tilde{T}}=\tilde{C}$. Let us express 
\[
C_{\tilde{T}}
\sim b_1\xi_{\tilde{T}}+(\pi')^*\cM
\]
for some $b_1\in\bZ$ and $\cM\in\operatorname{Pic} C$. 
Then we have 
\[
b_H=\left((\pi')^*\cL\cdot C_{\tilde{T}}\right)=b_0b_1, 
\quad
b_E=\left(E_{\tilde{T}}\cdot C_{\tilde{T}}\right)
=\deg\cM. 
\]
In particular, we have $b_0, b_1\in\bZ_{>0}$. 
On the other hand, by Lemma \ref{lemma_small}, 
\[
X_{\tilde{T}}:=(\nu')^* \tilde{X} \sim k \xi_{\tilde{T}}+(\pi')^*\cL
\]
is an effective divisor on $\tilde{T}$ containing 
$C_{\tilde{T}}$. Let us set 
\[
m_0:=\begin{cases}
0 & \text{if }C_0\not\subset(x_0=0), \\
m & \text{if }C_0\subset (x_0=0).
\end{cases}
\]
Since the order of the effective divisor $\tilde{X}|_{\bH}$ 
on $\bH$ along $E$ is $m$, the order of $X_{\tilde{T}}$ 
along $E_{\tilde{T}}$ is exactly equal to $m_0$. 
Thus the divisor $X_{\tilde{T}}-m_0 E_{\tilde{T}}-
C_{\tilde{T}}$ is effective and its support does not contain 
$E_{\tilde{T}}$. This implies that 
\[
0\leq\left(X_{\tilde{T}}-m_0 E_{\tilde{T}}-
C_{\tilde{T}}\right)\cdot E_{\tilde{T}}=(am_0+1)b_0-b_E. 
\]
Since $b_H\geq b_0$ and $m\geq m_0$, we get the desired 
inequality \eqref{equation_sesh-non-ec}. 
\end{proof}

\begin{corollary}\label{corollary_sesh-non-ec}
Under the assumption in Lemma \ref{lemma_sesh-non-ec}, 
for a general element $\tilde{S}\in 
\left|g^*\cO_{\bP^3}(1)\right|$, the $\Q$-divisor 
\[
\xi|_{\tilde{S}}-\left(1+\frac{a}{am+1}\right)\bE|_{\tilde{S}}
\]
is nef. 
\end{corollary}

\begin{proof}
Follows immediately from Lemmas \ref{lemma_small} and 
\ref{lemma_sesh-non-ec}. 
\end{proof}

\begin{proposition}\label{p:surfacesingular}
Assume $n=3$ and $m \le k-1$.  Let $S$ be a general element in $|\cO_X(1)|$ and $p=[0:0:0:1] \in S$. Then
$$
\delta_p(S; \cO_S(1)) \ge 3\frac{(m+1)a+1}{(ma+1)d}.
$$
\end{proposition}
\begin{proof}
Let $\tilde S$ be the strict transform of $S$ in $\tilde X$, so $\psi: \tilde S \to S$ is the ordinary blow-up of $S$ in $p$.  By Corollary \ref{corollary_sesh-non-ec} we know that the divisor 
$$
\psi^* \cO_S(ma+1)- \frac{(m+1)a+1}{a}E|_{\tilde{S}}
$$
is nef on $\tilde S$, that is
$$
\varepsilon\left(\cO_S(1); E|_{\tilde{S}}\right)\geq
\frac{(m+1)a+1}{(ma+1)a}.
$$ 
Then we can apply Corollary \ref{corollary_lobster}. 
\end{proof}

For special cases, we have a better bound: 

\begin{proposition}\label{p:surface_dbig-new}
Assume that $n=3$ and $d\geq (ma+1)^2$. Let $S$ be a general element in 
$|\cO_X(1)|$ and $p=\left[0:0:0:1\right]\in S$. Then 
\[
\delta_p\left(S;\cO_S(1)\right)\geq\frac{3(ma+1)}{d}. 
\]
\end{proposition}

\begin{proof}
From the assumption, we have $m\leq k-1$. 
Let $\tilde{S}$ be the strict transform of $S$ in $\tilde{X}$. 
By Corollary \ref{corollary_G-irr}, the curve 
$G|_{\tilde{S}}$ is a prime divisor, and 
\[
\left(G|_{\tilde{S}}^{\cdot 2}\right)=\frac{d-(ma+1)^2}{a}\geq 0. 
\]
Thus we have $\varepsilon\left(\cO_S(1); E|_{\tilde{S}}\right)\geq
\frac{ma+1}{a}$. 
Then we can again apply Corollary \ref{corollary_lobster}. 
\end{proof}

\subsection{An Eckardt case}\label{section_eckardt}

In \S \ref{section_eckardt}, we consider the case $m=k$, i.e., 
$X$ has a generalized Eckardt point at $P$. 
The result in this section 
is similar to \cite[\S 4.4]{AZ22}. 

Under the assumption in Setup \ref{setup:d=ak+1}, 
assume moreover that $m=k$. Then we can write
\begin{eqnarray*}
f&=&x_{n+1} x_0 g +f_{ak+1}(x_0,\dots,x_n)\text{ with }\\
g&:=&x_{n+1}^{k-1}+x_{n+1}^{k-2}g_a(x_0,\dots,x_n)
+\cdots+g_{a(k-1)}(x_0,\dots,x_n), 
\end{eqnarray*}
where $g_{aj}$ ($1\leq j \leq k-1$) (resp., $f_{ak+1}$) 
is a homogeneous polynomial with degree $aj$ (resp., $ak+1$). 
We write 
\[
f_{ak+1}(x_0,\dots,x_n)=x_0f_{ak}(x_0,\dots,x_n)
+g_{ak+1}(x_1,\dots,x_n).
\]

\begin{example}\label{example_eckardt-qsm}
There are many such quasi-smooth hypersurfaces $X\subset\bP$. 
For example, assume that 
\[
f=x_{n+1}^kx_0+\sum_{i=0}^nc_ix_i^{ak+1}
\]
with $c_0,\dots,c_n\in\bC^*$. (In other words, our $f$ satisfies that 
$g=x_{n+1}^{k-1}$, $f_{ak}
(x_0,\dots,x_n)=c_0x_0^{ak}$ and $g_{ak+1}(x_1,\dots,x_n)=\sum_{i=1}^n
c_ix_i^{ak+1}$.)
Then we can directly check that the hypersurface $X\subset\bP$ is 
quasi-smooth. 
\end{example}

In the above example, $g_{ak+1}(x_1,\dots,x_n)$ defines a 
(smooth) Fermat hypersurface in $\bP^{n=1}_{x_1\dots x_n}$. 
In fact, in general, the following holds.

\begin{lemma}[{cf.\ \cite[Lemma 4.9]{AZ22}}]\label{lemma_eckardt}
The hypersurface $W:=(g_{ak+1}=0)\subset
\bP^{n-1}_{x_1\cdots x_n}$ is smooth. 
\end{lemma}

\begin{proof}
Assume that $[c_1:\cdots:c_n]\in W$ is a singular point. 
We can take $y_0\in\bC$ such that 
\[
y_0g(0,c_1,\dots,c_n,y_0)+f_{ak}(0,c_1,\dots,c_n)=0
\]
since the polynomial $g(0,c_1, \ldots ,c_n, y)$ in $y$ is non-constant.
Then $\vec{c}:=[0:c_1:\dots:c_n:y_0] \in X$. 
However, we can check that $X$ is not quasi-smooth at $\vec{c}$, a 
contradiction. 
\end{proof}

Let us set $V\subset X$ be the prime divisor defined by the equation $x_0=0$, 
and let us set $\tilde{V}:=\sigma^{-1}_*V\subset\tilde{X}$. 
The Cartier divisor $\bH|_{\tilde{X}}$ on $\tilde{X}$ is nothing but 
$\tilde{V}+k E$. Moreover, the defining equations of $\tilde{V}$ in 
$\tilde{\bP}$ is $g_{ak+1}(x_1,\dots,x_n)=x_0=0$. 
In particular, $\tilde{V}\xrightarrow{\pi|_{\tilde{V}}}\pi(\tilde{V})=W
\subset\bP^{n-1}_{x_1\cdots x_n}$ is a $\bP^1$-bundle over $W$. 
Note that, for a general fiber $\ell\subset\tilde{V}$ of 
$\pi|_{\tilde{V}}$, we have $(\tilde{V}\cdot \ell)=-k$.

From now on, set $Y_0:=X$, $Y_1:=E$ and $W_1:=\tilde{V}|_E(\simeq W)$. 
Take any closed point $q\in E$. Consider a sequence of 
general linear subspaces $Y_1\supsetneq Y_2\supsetneq\dots\supsetneq Y_n$ with $Y_n=\{q\}$. Then, 
\[
Y_\bullet\colon Y_0\triangleright Y_1\triangleright\cdots\triangleright
Y_n
\]
is a complete plt flag over $X$. 
For any $2\leq j\leq n-1$, let us inductively set $W_j:=W_{j-1}|_{Y_j}$. 
Then each $W_j\subset Y_j$ is a smooth hypersurface. 

\begin{proposition}\label{proposition_ninja}
For any $1\leq j\leq n$, we have 
\[
S\left(\cO_X(1);Y_1\triangleright\cdots\triangleright Y_j\right)
=\begin{cases}
\frac{ak+n}{a(n+1)} & \text{if }j=1,\\
\frac{2ak+1}{(ak+1)(n+1)} & \text{if }j=n\text{ and }q\in W_1, \\
\frac{1}{n+1} & \text{otherwise}.
\end{cases}
\]
\end{proposition}

\begin{proof} 

The flag $Y_\bullet$ itself is a smooth dominant of $Y_\bullet$ 
in the sense of \cite[Definition 7.1]{Fujita:2024aa} (taking $\gamma_k = \mathrm{id}$ in the definition cited). 
Then, for any $1\leq j<l\leq n$, the values 
$d_{jl}$ and $g_{jl}$ in \cite[Definition 7.2]{Fujita:2024aa} satisfy 
$d_{jl}= g_{jl}=0$. 
The $\bR$-divisor 
\begin{equation}\label{eq:1stdecomp}
\sigma^*\cO_X(1)-x_1 E\sim_\bR
\left(\left(\frac{1}{a}-x_1\right)\xi+\pi^*\cO_{\bP^n}
(ax_1)\right)\bigg|_{\tilde{X}}
\end{equation}
is nef whenever $x_1\in [0,1/a ]$. On the other hand, 
a general fiber 
$\ell\subset\tilde{V}$ of $\pi|_{\tilde{V}}$ satisfies that 
\[
\left(\sigma^*\cO_X(1)-x_1 E\cdot \ell\right)=\frac{1}{a}-x_1
\]
and 
\begin{equation}\label{eq:2nddecomp}
\sigma^*\cO_X(1)-x_1 E-\frac{1}{k}\left(x_1-\frac{1}{a}\right)\tilde{V}
\sim_\bR g^*\cO_{\bP^n}\left(\frac{1}{k}\left(\frac{ak+1}{a}
-x_1\right)\right)
\end{equation}
is semiample whenever $x_1\in [1/a, (ak+1)/a]$. 
In particular, we have $\sigma^*\cO_X(1)-\frac{ak+1}{a} E \sim_{\bR} \tilde{V}$ and see that 
\[
T(\cO_X(1); E)=\frac{ak+1}{a}. 
\]

Moreover, we see that, for any $x_1\in[0,(ak+1)/a]$, the decomposition 
\[
\sigma^*\cO_X(1)-x_1E=N_{0,0}(x_1)+P_0(x_1)
\]
is the Zariski decomposition in a strong sense as in \cite[Definition 6.9]{Fujita:2024aa}, where 
\[
N_{0,0}(x_1)=\begin{cases}
0 & \text{if }x_1\in\left[0,\frac{1}{a}\right], \\
\frac{1}{k}\left(x_1-\frac{1}{a}\right)\tilde{V} & \text{if }
x_1\in\left[\frac{1}{a}, \frac{ak+1}{a}\right]
\end{cases}. 
\] 
In particular, $Y_\bullet$ itself is a smooth and adequate 
dominant of $Y_\bullet$ with 
respect to $\cO_X(1)$ in the sense of \cite[Definition 8.5]{Fujita:2024aa}. 

For $l=1, \ldots , n$, let $\bD_l \subset \bR_{>0}^l$ be the region which is defined inductively as in \cite[Definition 8.1]{Fujita:2024aa} from the plt flag (and smooth dominant) $Y_{\bul}$ and $L= \cO_X(1)$.  
Since $P_0(x_1)|_E \simeq \cO_{E}(ax_1)$ if $x_1 \in [0, \frac{1}{a}]$ by (\ref{eq:1stdecomp}) and $P_0(x_1)|_E \simeq \cO_{E}\left(\frac{1}{k}\left(\frac{ak+1}{a}-x_1\right)\right)$ if $x_1 \in [\frac{1}{a}, \frac{ak+1}{a}]$ by (\ref{eq:2nddecomp}), 
we can directly 
check that $\bD_n = \bD_{n,1} \cup \bD_{n,2}$ for  
\begin{eqnarray*}
\bD_{n,1}&:=&\left\{(x_1,\dots,x_n)\in\bR^n_{>0}\,\,\bigg|\,\,
x_1\leq \frac{1}{a}\text{ and }x_2+\cdots+x_n<ax_1\right\}, \\
\bD_{n,2} &:=&\left\{(x_1,\dots,x_n)\in\bR^n_{>0}\,\,\bigg|\,\,
\frac{1}{a}<x_1<\frac{ak+1}{a}\text{ and }x_2+\cdots+x_n<
\frac{1}{k}\left(\frac{ak+1}{a}-x_1\right)\right\}. 
\end{eqnarray*}
In the same terminology, we have the following numbers: 
\[
u_{l,j}(x_1, \ldots , x_l) = \ord_{Y_j} (N_{l-1, j-2}(x_1, \ldots ,x_l)|_{Y_{j-1}}) \in \bR_{\ge 0}
\text{ for $1 \le l <j$ and $(x_1, \ldots , x_l) \in \bD_l$, and}  
\]   
\[
v_j(x_1, \ldots ,x_{j-1}):= \sum_{l=1}^{j-1} u_{l,j}(x_1, \ldots , x_l) \in \bR_{\ge 0}
\text{ for $2 \le j \le n$ 
and $(x_1, \ldots , x_{j-1}) \in \bD_{j-1}$.} 
\]
We see that $u_{l,j} = 0$ for $2\le j \le n-1$ since $Y_j \simeq \bP^{n-j}$ and $N_{i-1,i-1}(x_1, \ldots, x_i)=0$ for $2\le i \le n$ and $(x_1, \ldots ,x_i) \in \bD_i$. 
By this and the generality of the linear subspaces $Y_i$, we also see that 
\begin{eqnarray*}
v_j&\equiv& 0 \text{ if } 2\leq j\leq n-1,\\
v_n(x_1,\dots,x_{n-1})&=&\begin{cases}
\frac{1}{k}\left(x_1-\frac{1}{a}\right) & \text{if }\frac{1}{a}<x_1<
\frac{ak+1}{a} \text{ and }q\in W_1,\\
0 & \text{otherwise}.
\end{cases}. 
\end{eqnarray*}

By \cite[Theorem 8.8 and Remark 8.9(1)]{Fujita:2024aa} and $g_{jl}=0$, we know that 
\[
S\left(\cO_X(1);Y_1\triangleright\cdots\triangleright Y_j\right)
=\frac{n!a}{ak+1}\int_{\vec{x}\in\bD_n}
\left(x_j+v_j(x_1,\dots,x_{j-1})\right)d\vec{x} 
\]
for any $1\leq j\leq n$. 
(Note that this also holds when $j=1$ by setting $v_1:=0$.) 
Thus the assertion follows from direct computations. 

\end{proof}

\begin{corollary}\label{corollary_ninja}
We have 
\[
\delta_P\left(X; \cO_X(1)\right)\geq\min\left\{
\frac{n(n+1)}{ak+n},\,\,\frac{(ak+1)(n+1)}{2ak+1}\right\}. 
\]
Moreover, if $d=ak+1\geq n$, we have 
\[
\delta_P\left(X; \cO_X(1)\right)=\frac{n(n+1)}{ak+n}.
\]
\end{corollary}

\begin{proof}

By Theorem \ref{theorem:AZ}, we have  
\[
\delta_P(X;\cO_X(1))\geq\min\left\{\frac{A_X(Y_1)}{S(\cO_X(1); Y_1)}, 
\quad \inf_{q\in Y_1}\delta_q(X\triangleright Y_1; \cO_X(1))\right\}. 
\]
Moreover, for any $q\in Y_1$ and for any $2\leq j\leq n-1$,  by Theorem \ref{theorem:AZ} again, we obtain 
\[
\delta_q(X\triangleright Y_1\triangleright\cdots\triangleright Y_{j-1}; 
\cO_X(1))\geq\min\left\{\frac{A_{Y_{j-1}}(Y_j)}{S(\cO_X(1); Y_1\triangleright\cdots\triangleright Y_j)},\delta_q(X\triangleright Y_1\triangleright\cdots\triangleright Y_{j}; 
\cO_X(1))\right\}
\]
with $q=Y_n$. 
Together with these, Proposition \ref{proposition_ninja} and   
\[
\delta_q(X\triangleright Y_1\triangleright\cdots\triangleright Y_{n-1}; 
\cO_X(1))=\frac{A_{Y_{n-1}}(Y_n)}{S(\cO_X(1); Y_1\triangleright\cdots\triangleright Y_n)}, 
\]
we get the inequality
\[
\delta_P\left(X; \cO_X(1)\right)\geq\min\left\{
\frac{n(n+1)}{ak+n},\,\,n+1,\,\,\frac{(ak+1)(n+1)}{2ak+1}\right\}. 
\]

Moreover, we know that 
\[
\delta_P\left(X; \cO_X(1)\right)\leq\frac{A_X(E)}{S(\cO_X(1);E)}=
\frac{n(n+1)}{ak+n}.
\]
Thus the assertion follows. 
\end{proof}

\begin{corollary}\label{corollary_unstable}
Assume that $a\geq 2$ and $n>\frac{a^2k(k-1)}{a-1}$. 
Then $X$ is a K-unstable klt Fano variety. 
\end{corollary}

\begin{proof}
Note that 
$-K_X\sim \cO_X(n+a-ak)$ and $n+a-ak>0$ from the assumption. 
Moreover, 
\[
\delta(X)=\frac{1}{n+a-ak}\cdot\delta\left(X;\cO_X(1)\right)
\leq\frac{n(n+1)}{(n+a-ak)(ak+n)}
\]
by Proposition \ref{proposition_ninja}. Thus the assertion follows 
from the assumption since we have 
\[
(n+a-ak)(n+ak)-n(n+1)= (n^2-a^2k^2 +an+a^2k) -(n^2+n) = (a-1)n -a^2k(k-1). 
\]
\end{proof}

\begin{example}\label{example:unstable}
Let $X_{ak+1} \subset \bP(1^{n+1},a)$ be a hypersurface with a generalized Eckard point as in Corollary \ref{corollary_unstable} (cf. Example \ref{example_eckardt-qsm}). Note that the Fano index $r = r_X$ of $X$ is 
\[
r= n+1+a - (ak+1) = n-(k-1)a. 
\]
Then the condition $n>\frac{a^2k(k-1)}{a-1}$ in Corollary \ref{corollary_unstable} implies that 
\[
r> \frac{a^2k(k-1)}{a-1} - (k-1)a = \frac{\left( (k-1)a+1\right) (k-1)a}{a-1}. 
\]
For example, if $k=2$, we have 
$
r > \frac{a(a+1)}{a-1} \ge 6.
$
Thus, for any $a\geq 2$ and for any $n\in\bZ_{>0}$ with 
$n>\frac{2a^2}{a-1}$, the Fano hypersurface $X_{2a+1}\subset\bP(1^{n+1},a)$ as in 
Corollary \ref{corollary_unstable} is of index $r=n-a$ and is K-unstable. 
If we take $n=\lfloor\frac{2a^2}{a-1}+1\rfloor$ and if $a$ goes to infinity, the ratio 
$r/n$ goes to $1/2$.

\end{example}


\section{Proof of Theorem \ref{thm:d=ak+1}}\label{section_wt1}

We start with the following proposition:

\begin{proposition}\label{proposition_2wts-B1}
Let $n\geq 3$ and let $\bP:=\bP\left(a_0,\dots,a_{n+1}\right)$ be 
well-formed with $a_0\leq\cdots\leq a_{n+1}$. 
Let $X\subset\bP$ be a quasi-smooth hypersurface 
of degree $d>a_{n+1}+1$. Assume that $B_1^\bP\subset X$ (cf. Definition \ref{definition_Ba}). 
Then we have 
\[
\delta\left(X;\cO_X(1)\right)\geq\frac{n+1}{d-a_{n+1}}. 
\]
More precisely, for any $p\in X$, we have the following: 
\begin{enumerate}
\item 
If $p\neq P_{n+1}:=\left[0:\cdots:0:1\right]$, then we have 
\[
\delta_p\left(X;\cO_X(1)\right)\geq\frac{(n+1)a_{n+1}}{d}. 
\]
\item 
If $p=P_{n+1}$, then we have 
\[
\delta_p\left(X;\cO_X(1)\right)\geq\max\left\{\frac{(n+1)a_n}{d},\,\,
\frac{n+1}{d-a_{n+1}}\right\}. 
\]
\end{enumerate}
We remark that, if $a_n\geq 2$, then we always have the inequality 
\[
\frac{(n+1)a_n}{d}>\frac{n+1}{d-a_{n+1}}. 
\]
\end{proposition}

\begin{proof}
\noindent(1) By applying Propositions \ref{proposition_B1} and 
\ref{proposition_ST-cover} several times, after twisting by an 
element of $\operatorname{Aut}\left(\bP\right)$, the morphism 
\begin{eqnarray*}
\tau\colon \bP':=\bP\left(1^{n+1},a_{n+1}\right)&\to&\bP\\
\left[x_0:\cdots:x_n:x_{n+1}\right]&\mapsto&
\left[x_0^{a_0}:\cdots:x_n^{a_n}:x_{n+1}\right]
\end{eqnarray*}
satisfies that, $X':=\tau^*X\subset\bP'$ is quasi-smooth of degree $d$ 
and 
\[
\delta_{p'}\left(X';\cO_{X'}(1)\right)\leq 
\delta_p\left(X;\cO_X(1)\right)
\]
for any $p'\in X'$ with $\tau(p')=p$. Thus, for any $p\neq P_{n+1}$, 
since $\left\{P_{n+1}\right\}\subset X'$, 
we have 
\[
\delta_p\left(X;\cO_X(1)\right)\geq\frac{(n+1)a_{n+1}}{d}
\]
by Corollary \ref{corollary_smooth-points}. 

\noindent(2) By the above argument using the morphism 
\[
\bP(1^n,a_n,1) \to \bP;  \ \ \ \ [x_0: \cdots : x_{n+1}] \mapsto [x_0^{a_0}: \cdots : x_{n-1}^{a_{n-1}}: x_n : x_{n+1}^{a_{n+1}}],
\]
we also have 
\[
\delta_{P_{n+1}}\left(X;\cO_X(1)\right)\geq\frac{(n+1)a_n}{d}.
\]
Thus, we may assume that $p=P_{n+1}$ and $a_0=\cdots=a_n=1$. 
Since $\left\{P_{n+1}\right\}\subset X'$ and $d>a_{n+1}+1$, 
there exists $k\in\bZ_{\geq 2}$ such that $d=ka_{n+1}+1$. 
By Proposition \ref{p:surfacesingular} and Corollary \ref{corollary_ninja}, 
we have 
\[
\delta_p\left(X;\cO_X(1)\right)\geq\min\left\{
\frac{(n+1)}{(k-1)a_{n+1}+1},\,\,\frac{n(n+1)}{a_{n+1}k+n},\,\,
\frac{(a_{n+1}k+1)(n+1)}{2a_{n+1}k+1}\right\}. 
\]
We see that the 2nd term is larger than the 1st term by $k \ge 2$ and  
\[
n((k-1)a_{n+1}+1) - (a_{n+1}k +n) = a_{n+1}(n(k-1)-k) \ge a_{n+1}(3(k-1)-k) >0. 
\]
Since the 3rd term is $\frac{d}{2d-1}(n+1) > \frac{n+1}{2}$, we get the assertion (2). 

If $a_n\geq 2$, then we have 
\[
d(a_n-1)\geq (2a_{n+1}+1)(a_n-1)>2a_{n+1}(a_n-1)\geq a_na_{n+1}.
\]
Thus we also get the last remark. 
\end{proof}

We briefly analyze the condition $B_1^\bP\subset X$. 

\begin{lemma}\label{lemma_B1-cond}
Let $n\in\bZ_{>0}$, $\bP:=\bP\left(a_0,\dots,a_{n+1}\right)$ be well-formed 
with 
\[
1=a_0=\cdots=a_{c_1-1}<a_{c_1}\leq\cdots\leq a_{n+1}. 
\]
Let $X\subset\bP$ be a quasi-smooth hypersurface of degree $d>1$. 
\begin{enumerate}
\item 
If $\dim B_1^\bP>\frac{n}{2}$ (i.e., if $n+1\geq 2c_1$), then 
$B_1^\bP\not\subset X$. 
\item
If 
\[
d\not\in\sum_{j=c_1}^{n+1}\bZ_{\geq 0}a_j,
\]
then $B_1^\bP\subset X$. 
\item 
If $B_1^\bP\subset X$, then, for any $c_1\leq j\leq n+1$, we have $d\equiv 1 \mod a_j$. 
\end{enumerate}
\end{lemma}

\begin{proof}
Let $f$ be the defining homogeneous polynomial of $X$.  
If $d\not\in\sum_{j=c_1}^{n+1}\bZ_{\geq 0}a_j$, then $f$ does not have 
any monomial of the form $x_{c_1}^{j_{c_1}}\cdots x_{n+1}^{j_{n+1}}$. 
Thus the assertion (2) is trivial. The assertion (3) is also trivial 
since $X$ is quasi-smooth at the $j$-th coordinate point. 

Let us consider (1). Assume that $B_1^\bP\subset X$. 
Since $n+1-c_1\geq c_1$ 
and $d>1$, there exists a point $p=\left[0:\cdots:0:p_{c_1}:\cdots:
p_{n+1}\right]\in B_1^{\bP}$ such that 
\[
\frac{\partial f}{\partial x_0}(p)=\cdots=
\frac{\partial f}{\partial x_{c_1-1}}(p)=0.
\]
Since $B_1^\bP\subset X$, we also have 
\[
\frac{\partial f}{\partial x_{c_1}}(p)=\cdots=
\frac{\partial f}{\partial x_{n+1}}(p)=0.
\]
This leads to a contradiction since $X$ is quasi-smooth. 
\end{proof}

We analyze how much the estimate in Proposition \ref{proposition_2wts-B1} 
is sharp.

\begin{lemma}\label{lem:num_d=ka+1}
Let $n\geq 3$ and let $\bP:=\bP\left(a_0,\dots,a_{n+1}\right)$ be 
well-formed with $a_0\leq\cdots\leq a_{n+1}$. Assume $a_n \ge 2$.
Let $X\subset\bP$ be a quasi-smooth Fano hypersurface of index $1$. Assume that $B_1^\bP\subset X$ and write $d=ka_{n+1}+1$.
Then 
$$
\frac{(n+1)a_n}{d} > 1,
$$
in the following cases:
\begin{enumerate}
    \item $k \ge 4$,
    \item $k \ge 3$ and $a_n \ge 3$.
\end{enumerate}
\end{lemma}

\begin{proof}
Set 
$h:=\#\left\{0\leq i\leq n+1\,\,|\,\,a_i>1\right\}$. 
Since $B_1^\bP\subset X$, we have $2h\leq n+2$ by 
Lemma \ref{lemma_B1-cond} (1).
Since $X$ is of index $1$, we have
$$
d+1=\sum_{i=0}^{n+1} a_i \le h+ (h-1)a_n + a_{n+1},
$$
that is
$$
h \ge \frac{(k-1)a_{n+1}+ 2+a_n}{a_{n}+1},
$$
which gives
$$
n+1 \ge 2h-1 \ge 2\frac{(k-1)a_{n+1}+2+a_n}{a_{n}+1}-1= \frac{2(k-1)a_{n+1}+3+a_n}{a_n+1}.
$$
We have
$$
(2(k-1)a_{n+1}+3+a_n)a_n -(ka_{n+1}+1)(a_n+1) =(k-2)a_{n+1}a_n+a_n^2-1-ka_{n+1} +2a_n, 
$$
which is strictly greater than 0 if $(k-2)a_n \ge k$, i.e. $a_n \ge k/(k-2)$. This is true if $k \ge 4$ or $k=3$ and $a_n \ge 3$, which is what we want. 
\end{proof}

Here is an immediate corollary of Proposition \ref{proposition_2wts-B1}. 

\begin{corollary}\label{corollary_2wts-B1}
Let $a\geq 2$, $n\geq 3$, and $X\subset\bP\left(1^{n+1},a\right)$ 
be a quasi-smooth hypersurface of degree $d\geq a+2$. Then we have 
\[
\delta\left(X; \cO_X(1)\right)\geq\frac{n+1}{d-a}. 
\]
\end{corollary}

\begin{proof}
By Proposition \ref{proposition_2wts-B1}, we may assume that 
$B_1^\bP\not\subset X$. 
In this case, $a \mid d$ holds. By \cite[Theorem 1.1]{ST24}, 
we have 
\[
\delta\left(X; \cO_X(1)\right)\geq\frac{(n+1)a}{d}\geq\frac{n+1}{d-a}. 
\]
Thus, the assertion follows. 
\end{proof}

\begin{proof}[Proof of Theorem \ref{thm:d=ak+1}] 
This directly follows from Corollary \ref{corollary_2wts-B1} since we have $d=n+a$ by the assumption.

\end{proof}

\section{Proof of Theorem \ref{thm:wt2}}\label{section_wt2}

\begin{proposition}\label{prop:2weights}
Let $X \subset \bP(1^n,a_n, a_{n+1})$ 
be a quasi-smooth weighted hypersurface of degree $d\geq a_{n+1}+2$ and dimension $n \geq 2$.  Assume
$a_n \le a_{n+1}$.
Then 
$$
\delta\left(X; \cO_X(1)\right) \ge \frac{(n+1)a_{n}}{d}.
$$
\end{proposition}

\begin{proof}
By Proposition \ref{proposition_2wts-B1}, we may assume that 
$B_1^\bP\not\subset X$. For $2\leq j\leq n$, let us set 
\begin{eqnarray*}
&&\bP_n:=\bP(1^n,a_n,a_{n+1}), \\
&&\bP_{j-1}:=\bP_j\cap\left(x_{j-1}=0\right)\left(\simeq
\bP(1^{j-1},a_n,a_{n+1})\right)\quad(2\leq j\leq n).
\end{eqnarray*}
Since $B_1^\bP\not\subset X$, we can apply Lemma 
\ref{lem:hyperplane_mod} (1): 
there exists $\varphi\in\operatorname{Aut}\left(\bP\right)$ such that, 
after replacing $X\subset\bP$ with $\varphi(X)\subset\bP$, we have 
$X_j\subset\bP_j$ quasi-smooth for any $1\leq j\leq n-1$, where 
\[
X_n:=X\subset\bP_n, \quad X_{j-1}:=X_j|_{\bP_{j-1}}\subset\bP_{j-1}
\quad(2\leq j\leq n).
\]
Note that $\dim \Sing X_j =0$ by $B_1^{\bP} \not\subset X$. 
By \cite[Lemmas 2.9 and 2.10]{CO}, for any $2\leq j\leq n$, 
the pair $(X_j,X_{j-1})$ is plt and 
\[
\left(K_{X_j}+X_{j-1}\right)|_{X_{j-1}}=\begin{cases}
K_{X_{j-1}} & \text{if }3\leq j\leq n, \\
K_{X_1}+\Delta_1 & \text{if }j=2 
\end{cases}
\]
with 
\[
\Delta_1=\sum_{p\in B_1^\bP\cap X}\left(1-\frac{1}{r_p}\right)p, 
\]
where 
\[
r_p:=\begin{cases}
a_{n+1} & \text{if }p=\left[0:\cdots:0:0:1\right],\\
a_n & \text{if }p=\left[0:\cdots:0:1:0\right],\\
\operatorname{gcd}(a_n,a_{n+1}) & \text{otherwise}.
\end{cases}
\]
For any $p\in B_1^\bP\cap X$, by Lemma \ref{lem:cutting}, we have 
\begin{eqnarray*}
\delta_p\left(X; \cO_X(1)\right)\geq\min\left\{
n+1,\quad\frac{n+1}{2}\delta_p\left(X_1,\Delta_1; 
\cO_X(1)|_{X_1}\right)\right\}\geq\frac{(n+1)a_n}{d},  
\end{eqnarray*}
where we used $\delta_p\left(X_1,\Delta_1; 
\cO_X(1)|_{X_1}\right) = \frac{1/r_p}{d/2a_na_{n+1}}$. 
Thus we may further assume that $p\not\in B_1^\bP\cap X$. 

We assume the case $[0:\cdots:0:1]\not\in X$. In this case, we have 
$a_{n+1}\mid d$. Thus the assertion is trivially true 
by \cite[Theorem 1.1]{ST24}. Hence we may further assume that 
$[0:\cdots:0:1]\in X$. 

Let $\psi\colon\tilde{\bP}\to \bP_n$ be the standard weighted blowup of 
$\bP_n$ along $\bP(a_{n+1})$, and let $\tilde{X}\subset\tilde{\bP}$ be 
the strict transform of $X$. By Proposition \ref{proposition_qsm-wbl}, 
the $\tilde{X}\subset\tilde{\bP}$ is a quasi-smooth hypersurface 
and $\tilde{X}|_\bE\subset\bE$ is also quasi-smooth. 
By applying Corollary \ref{corollary_2wt-cutting} and Proposition 
\ref{proposition_perhaps-useful}, combining with Lemma \ref{lem:cutting}, 
we have 
\begin{eqnarray*}
\delta_p\left(X; \cO_X(1)\right)\geq\min\left\{
n+1,\quad\frac{n+1}{3}\cdot\frac{3a_{n+1}}{d}\right\}
=\frac{(n+1)a_{n+1}}{d}.
\end{eqnarray*}
Thus the assertion follows. 
\end{proof}

\begin{lemma}\label{lemma_numeri}
Let $X=X_d \subset \bP(1^n,a_n, a_{n+1})$ be a quasi-smooth Fano weighted hypersurface of degree $d$, dimension $n \ge 3$ 
and index $1$.  Assume $1< a_n \le a_{n+1}$. 
Then we have 
$$
\frac{(n+1)a_{n}}{d} \ge 1,
$$
and the equality is possible only if $X= X_{2n+2} \subset \bP(1^n,2,n+1)$. 

More precisely, if $a_{n+1} \nmid d$, then we have 
\[
\frac{(n+1)a_n}{d}\geq\frac{n+1}{n+\frac{1}{a_n}}. 
\]
\end{lemma}

\begin{proof}
Let $d=ka_{n+1}+r$ be with $0\le r \le d-1$. By quasi-smoothness we know that $r \in \{0,1,a_n\}$ and since the index of $X$ is 1, we obtain $k \ge 2$.

Note that $n+1=d-a_{n+1}-a_n+2$ and we need to prove
$$
(ka_{n+1}+r+2-a_n-a_{n+1})a_n \ge ka_{n+1} +r.
$$
If $a_n=2$, then we get
$$
((k-1)a_{n+1}+r)2 \ge ka_{n+1} +r.
$$
which is clearly true and the equality holds if $k=2$ and  $r=0$. 
In this case, we have $n+2 + a_{n+1} = d+1 = 2a_{n+1} +1$ and we obtain $a_{n+1}=n+1$. 
Hence we see that $X= X_{2n+2} \subset \bP(1^n,2,n+1)$. 

Hence, we can assume $a_n \ge 3$. 
If $k \ge 3$, using $a_{n+1} \ge a_n$ we are left with
$$
((k-2)a_{n+1}+r+2)a_n \ge ka_{n+1} +r
$$
which is always a strict inequality since $a_n \ge 3$.
Hence we are left with the case $k=2$ and $a_n \ge 3$. Then we obtain
$$
(a_{n+1}+r+2-a_n)a_n \ge 2a_{n+1}+r,
$$
which reads as
$$
(a_{n+1}-a_n)(a_n-2)+ r(a_n-1) \ge 0.
$$
The latter always holds true, and equality is possible only if $r=0$ and $a_{n+1}=a_n$. 
Then we obtain $d= a_{n+1} + a_n$ and $n+1=2$, thus an equality does not occur in this case. 

The latter inequality is similar. From the assumption, we have 
$r\geq 1$ and $k\geq 2$. It is enough to show that 
\[
1+a_n\left((k-1)a_{n+1}-a_n+r+1\right)\geq k a_{n+1}+r. 
\]
Since $r\geq 1$, it is enough to show that 
\[
a_{n+1}\left((k-2)a_n-k\right)+a_n(2+a_{n+1}-a_n)\geq 0.
\]
If $k\geq 3$, then the inequality is trivial except when $k=3$ and $a_n=2$. 
When $k=3$ and $a_n=2$, we see that the left hand side becomes 
\[
a_{n+1}(2-3) + a_n (a_{n+1}) = a_{n+1}(a_n-1)>0, 
\]
thus we obtain a (strict) inequality. 
If $k=2$, then 
the inequality is equivalent to $(a_n-2)(a_{n+1}-a_n)\geq 0$, 
which is trivially true.
\end{proof}

\begin{proof}[Proof of Theorem \ref{thm:wt2}]
If $b \mid d$, then there exists $k\geq 2$ with $d=kb$. 
It follows from \cite[Theorem 2.4]{ST24} that 
\[
\delta\left(X;\cO_X(1)\right)\geq\frac{n+1}{k}\geq\frac{n+1}{(k-1)b+\frac{1}{a}+1-a}
=\frac{n+1}{n+\frac{1}{a}},
\]
where the last inequality follows from the inequality $(k-1)(b-1)\geq a-\frac{1}{a}$. (Indeed, if $k >2$ or $b>a$, then the inequality is clear. We also see that $k=2$ and $b=a$ cannot occur since, otherwise we have $2b+n=a+b+n=d+1=2b+1$ and this contradicts $n \ge 3$.)  

If $b\nmid d$, then it is an immediate consequence of Proposition 
\ref{prop:2weights} and Lemma \ref{lemma_numeri}. 
\end{proof}

\section{Proof of Corollary \ref{cor:imperial}}\label{section_imperial}

\begin{lemma}\label{lem:num_a_{n+1}}
Let $n\geq 3$ and let $\bP:=\bP\left(1,a_1,\dots,a_{n+1}\right)$ be well-formed with $a_1\leq \ldots \leq a_{n+1}$ and $1 < a_{n+1}$. 
Let $X\subset\bP$ be a well-formed quasi-smooth Fano hypersurface of index 1. 
Then 
$$
\frac{(n+1)a_{n+1}}{d} > 1.
$$
\end{lemma}
\begin{proof}
Write $d=ka_{n+1}+r$ with $r < a_{n+1}$ and $k \in \mathbb N$.
By the index 1 assumption, 
$$
ka_{n+1} + r +1 = \sum_{i=0}^{n+1}a_i < 1 + (n+1)a_{n+1},
$$
which implies $k \le n$ (we used well-formedness for the second strict inequality) and so
$$
\frac{(n+1)a_{n+1}}{ka_{n+1}+r} > 1.
$$
\end{proof}

\begin{lemma}\label{lemma:imperial}
Assume that $n\geq 3$. 
Let $X\subset\bP=\bP(1^{c_1},a_{c_1},\dots,a_{n+1})$ be a quasi-smooth Fano hypersurface 
of index $1$ of degree $d$ such that $B_1^\bP\subset X$, 
$c_1\leq n$ and 
$2\leq a_{c_1}\leq\cdots\leq a_{n+1}$. 
Then either $d<(n+1)a_n$ or $d<n^2$ holds. 
\end{lemma}

\begin{proof}
By Lemma \ref{lemma_B1-cond} (1), we have $c_1\geq\frac{n+2}{2}$. 
Note that $d=c_1+\sum_{i=c_1}^{n+1}a_i-1$. Moreover, by Lemma \ref{lemma_B1-cond} (3), 
there exists $k\geq 2$ such that $d=ka_{n+1}+1$ holds. By Lemma \ref{lem:num_d=ka+1}, 
we may assume that either $(k,a_n)=(3,2)$ or $k=2$. 

Assume firstly the case $(k,a_n)=(3,2)$. Since $a_{c_1}=\cdots=a_n=2$ 
and $d=3a_{n+1}+1$, we have 
$d=2 n+1-c_1+\frac{d-1}{3}$. Thus, by $n\geq 3$, we have 
\[
d=3n+1-\frac{3}{2}c_1\leq \frac{9}{4}n-\frac{1}{2}<n^2.
\]

We assume the case $k=2$. We also assume that $d\geq (n+1)a_n$. It is enough to show the 
inequality $d<n^2$. Since 
\begin{eqnarray*}
d&\leq&c_1+\sum_{i=c_1}^n a_n+a_{n+1}-1
\leq c_1+\frac{d}{n+1}(n+1-c_1)+\frac{d-1}{2}-1\\
&\leq&d-\left(\frac{d}{n+1}-1\right)\frac{n+2}{2}+\frac{d-3}{2}, 
\end{eqnarray*}
we get the desired inequality 
$d\leq n^2-1<n^2$. 
\end{proof}

\begin{proposition}\label{proposition:imperial}
Consider any $n\geq 3$, any $c_1\in\bZ$ with $\frac{n+2}{2}\leq c_1
\leq n+1$, 
and any $2\leq a_{c_1}\leq \cdots\leq a_{n+1}$ such that the value 
$d:=c_1+\sum_{i=c_1}^{n+1}a_i-1$ satisfies that 
there exists $k_i\geq 2$ such that $d=k_ia_i+1$ holds for any $c_1\leq 
i\leq n+1$. 
Let us consider the weighted hypersurface 
$X=(f=0)\subset\bP=\bP(1^{c_1},a_{c_1},\dots,a_{n+1})$ of degree $d$ defined by 
\[
f=\sum_{i=0}^{c_1-1}x_i^d+\sum_{j=0}^{n+1-c_1}x_j x_{c_1+j}^{k_{c_1+j}}. 
\]
Then the $X$ is a quasi-smooth, well-formed Fano hypersurface of index $1$ 
with $B_1^\bP\subset X$, and $\delta\left(X;\cO_X(1)\right)>1$ holds. 
\end{proposition}

\begin{proof}
The $X\subset \bP$ is obviously quasi-smooth and well-formed 
with $B_1^\bP\subset X$. 
By Theorem \ref{thm:d=ak+1}, we may assume that $c_1\leq n$. 
Take any point 
$p\in X$. By Proposition \ref{proposition_2wts-B1}, if $p\neq P:=\left[0:\cdots:0:1\right]$, 
then we have $\delta_p\left(X;\cO_X(1)\right)\geq\frac{(n+1)a_{n+1}}{d}>1$. 
Assume that $p=P$. By Proposition \ref{proposition_2wts-B1}, if $d<(n+1)a_n$, then we have 
$\delta_P\left(X; \cO_X(1)\right)>1$. Thus we may assume that $d\geq (n+1)a_n$. 
By Lemma \ref{lemma:imperial}, we must have $d<n^2$. 
Let us consider the abelian cover 
\begin{eqnarray*}
\tau\colon\bP':=\bP\left(1^{c_1},1^{n+1-c_1}, a_{n+1}\right)&\to&
\bP=\bP\left(1^{c_1},a_{c_1},\dots,a_n,a_{n+1}\right)\\
\left[x'_0:\cdots:x'_{c_1-1}:x'_{c_1}:\cdots:x'_n:x'_{n+1}\right]&\mapsto&
\left[x'_0:\cdots:x'_{c_1-1}:(x'_{c_1})^{a_{c_1}}:\cdots:(x'_n)^{a_n}:x'_{n+1}\right]
\end{eqnarray*}
and let us take $X':=\tau^*X\subset\bP'$. 
We remark that $X'\subset\bP'$ is defined by the equation 
\[
f':=\sum_{i=0}^{c_1-1}(x'_i)^d+\sum_{j=0}^{n+1-c_1}x'_j (x'_{c_1+j})^{d-1}, 
\]
which is again quasi-smooth and the point $P':=\left[0:\cdots:0:1\right]\in\bP'$ maps 
$\tau$ onto $P$. By Proposition \ref{proposition_ST-cover}, we have 
$\delta_P\left(X; \cO_X(1)\right)\geq \delta_{P'}\left(X'; \cO_{X'}(1)\right)$. 
Moreover, the point $P'\in X'\subset\bP'$ is a generalized 
Eckardt point. Thus, by Corollary 
\ref{corollary_ninja}, we have 
\[
\delta_{P'}\left(X'; \cO_{X'}(1)\right)\geq\min\left\{\frac{n(n+1)}{d+n-1},\,\,
\frac{d(n+1)}{2d-1}\right\}. 
\]
Clearly, we have $\frac{d(n+1)}{2d-1}>1$. Moreover, since $d<n^2$, we also have 
$\frac{n(n+1)}{d+n-1}>1$. As a consequence, we get the inequality 
$\delta\left(X; \cO_X(1)\right)>1$. 
\end{proof}

\begin{proof}[Proof of Corollary \ref{cor:imperial}]
This follows from Proposition \ref{proposition:imperial} and openness of (uniform) K-stability (cf. \cite{MR4411858,MR4505846, LXZ22}). 
\end{proof}

\bibliographystyle{amsalpha}
\bibliography{Library}

\end{document}